\documentclass[reqno, 11pt]{amsart} % draft
\usepackage{amsfonts, amsmath, amssymb, amsthm}
\usepackage[margin=2.75cm, heightrounded]{geometry}
\usepackage{fancybox}
\usepackage{hhline, float}
\usepackage{mathrsfs}
\usepackage{nccmath}
\usepackage[dvipsnames]{xcolor}
\usepackage[colorlinks=true]{hyperref}
\hypersetup{citecolor=red, linkcolor=blue}

\newtheorem{thm}{Theorem}[section]
\newtheorem{lemma}[thm]{Lemma}
\newtheorem{prop}[thm]{Proposition}
\newtheorem{cor}[thm]{Corollary}

\newtheorem{thmx}{Theorem}

\theoremstyle{definition}
\newtheorem{rmk}[thm]{Remark}
\newtheorem{defn}[thm]{Definition}

\newcommand{\ep}{\epsilon}
\newcommand{\ka}{\kappa}
\newcommand{\vep}{\varepsilon}
\newcommand{\vph}{\varphi}
\newcommand{\vrh}{\varrho}
\newcommand{\vsi}{\varsigma}
\newcommand{\om}{\omega}
\newcommand{\pa}{\partial}

\newcommand{\mone}{\mathbf{1}}

\newcommand{\N}{\mathbb{N}}
\newcommand{\R}{\mathbb{R}}
\renewcommand{\S}{\mathbb{S}}

\newcommand{\mca}{\mathcal{A}}
\newcommand{\mcb}{\mathcal{B}}
\newcommand{\mcd}{\mathcal{D}}
\newcommand{\I}{\mathcal{I}}
\newcommand{\mcj}{\mathcal{J}}
\newcommand{\mcs}{\mathcal{S}}
\newcommand{\mcv}{\mathcal{V}}
\newcommand{\mcz}{\mathcal{Z}}

\newcommand{\mfa}{\mathfrak{a}}
\newcommand{\mfb}{\mathfrak{b}}
\newcommand{\mfc}{\mathfrak{c}}
\newcommand{\mfd}{\mathfrak{d}}
\newcommand{\mfe}{\mathfrak{e}}

\newcommand{\msr}{\mathscr{R}}

\newcommand{\tc}{\tilde{c}}
\newcommand{\tv}{\tilde{v}}
\newcommand{\tw}{\tilde{w}}
\newcommand{\tde}{\tilde{\delta}}
\newcommand{\txi}{\tilde{\xi}}
\newcommand{\trh}{\tilde{\rho}}
\newcommand{\tvrh}{\tilde{\varrho}}
\newcommand{\tchi}{\tilde{\chi}}
\newcommand{\wtv}{\widetilde{V}}
\newcommand{\wtw}{\widetilde{W}}
\newcommand{\wtmcd}{\widetilde{\mathcal{D}}}

\newcommand{\olw}{\overline{W}}

\newcommand{\supp}{\textup{supp}}
\newcommand{\tin}{\textup{in}}
\newcommand{\tout}{\textup{out}}
\newcommand{\ext}{\textup{Ext}}
\newcommand{\core}{\textup{Core}}
\newcommand{\neck}{\textup{Neck}}

\renewcommand{\(}{\left(}
\renewcommand{\)}{\right)}
\newcommand{\la}{\left\langle}
\newcommand{\ra}{\right\rangle}

\newcommand{\loc}{\textnormal{loc}}

\newcommand{\bs}[1]{\boldsymbol{#1}}

\allowdisplaybreaks
\numberwithin{equation}{section}

\makeatletter
\@namedef{subjclassname@2020}{%
\textup{2020} Mathematics Subject Classification}
\makeatother

\begin{document}
\title[Sharp quantitative stability estimates for Brezis-Nirenberg problem]{Sharp quantitative stability estimates for the Brezis-Nirenberg problem}

\author{Haixia Chen}
\address[Haixia Chen]{Department of Mathematics and Research Institute for Natural Sciences, College of Natural Sciences, Hanyang University, 222 Wangsimni-ro Seongdong-gu, Seoul 04763, Republic of Korea}
\email{hxchen29@hanyang.ac.kr chenhaixia157@gmail.com}

\author{Seunghyeok Kim}
\address[Seunghyeok Kim]{Department of Mathematics and Research Institute for Natural Sciences, College of Natural Sciences, Hanyang University, 222 Wangsimni-ro Seongdong-gu, Seoul 04763, Republic of Korea}
\email{shkim0401@hanyang.ac.kr shkim0401@gmail.com}

\author{Juncheng Wei}
\address[Juncheng Wei] {Department of Mathematics, Chinese University of Hong Kong, Shatin, NT, Hong Kong}
\email{wei@math.cuhk.edu.hk}

\begin{abstract}
We study the quantitative stability for the classical Brezis-Nirenberg problem associated with the critical Sobolev embedding $H^1_0(\Omega) \hookrightarrow L^{\frac{2n}{n-2}}(\Omega)$ in a smooth bounded domain $\Omega \subset \mathbb{R}^n$ ($n \geq 3$).
To the best of our knowledge, this work presents the first quantitative stability result for the Sobolev inequality on bounded domains.
A key discovery is the emergence of unexpected stability exponents in our estimates, which arise from the intricate interaction among the nonnegative solution $u_0$ and the linear term $\lambda u$ of the Brezis--Nirenberg equation, bubble formation, and the boundary effect of the domain $\Omega$.
One of the main challenges is to capture the boundary effect quantitatively, a feature that fundamentally distinguishes our setting from the Euclidean case treated in \cite{CFM, FG, DSW} and the smooth closed manifold case studied in \cite{CK}.
In addressing a variety of difficulties, our proof refines and streamlines several arguments from the existing literature while also resolving new analytical challenges specific to our setting.
\end{abstract}

\date{\today}
\maketitle

\section{Introduction}
\subsection{Backgrounds}
The Brezis-Nirenberg problem is one of the most celebrated problems in nonlinear analysis. It is formulated as
\begin{equation}\label{BN_eq}
\begin{cases}
-\Delta u -\lambda u = u^{p} & \text{in } \Omega, \\
u \geq 0 & \text{in } \Omega, \\
u = 0 & \text{on } \pa\Omega,
\end{cases}
\end{equation}
where $\lambda \in \R$, $p:=2^*-1=\frac{n+2}{n-2}$, and $\Omega \subset \R^n$ ($n \geq 3$) is a smooth bounded domain.\footnote{The Brezis-Nirenberg problem may also refer to finding (sign-changing) solutions to $-\Delta u -\lambda u = |u|^{p-1}u$ in $\Omega$ and $u = 0$ on $\pa\Omega$. This paper is primarily concerned with its non-negative solutions, that is, solutions to \eqref{BN_eq}.}
Equation \eqref{BN_eq} was first introduced by Brezis and Nirenberg in their groundbreaking work \cite{BN}, which is closely linked to the critical Sobolev embedding via the Rayleigh quotient
\[Q_\lambda(u) := \frac{\int_{\Omega} (|\nabla u|^2 -\lambda u^2)dx}{\|u\|_{L^{p+1}(\Omega)}^2}, \quad u \in H_0^1(\Omega)\setminus\{0\},\]
with associated energy threshold
\[S_\lambda := \inf_{u \in H_0^1(\Omega)\setminus\{0\}} Q_\lambda(u).\]
When $\lambda = 0$, the constant $S_0$ coincides with the best constant of the Sobolev inequality in $\R^n$
\begin{equation}\label{Sobolev}
S_0 \(\int_{\R^n}|u|^{p+1} dx\)^{\frac{2}{p+1}} \leq \int_{\R^n}|\nabla u|^2 dx \quad \text{for all } u\in D^{1,2}(\R^n),
\end{equation}
where $D^{1,2}(\R^n)$ is the closure of the space $C^{\infty}_c(\R^n)$ with respect to the norm $\|\nabla u\|_{L^2(\R^n)}$.
It is well-known that $S_0$ is achieved if and only if $u$ is a constant multiple of the Aubin-Talenti bubbles \cite{AT, TAL} defined as
\begin{equation}\label{AT bubbles}
U_{\delta,\xi}(x) = a_n\(\frac{\delta}{\delta^2 + |x-\xi|^2}\)^{\frac{n-2}{2}}, \quad \xi \in \R^n,\, \delta>0,\ a_n=(n(n-2))^{\frac{n-2}{4}}.
\end{equation}
The selection of the dimensional constant $a_n$ guarantees that $U:=U_{1,0}$ solves the associated Euler-Lagrange equation
\begin{equation}\label{eleu}
-\Delta u = |u|^{p-1}u \quad \text{in } \R^n.
\end{equation}
In view of the Sobolev inequality, all solutions to \eqref{eleu} are critical points of the energy functional
\[J(u) = \frac{1}{2}\int_{\R^n}|\nabla u|^2 dx - \frac{1}{p+1} \int_{\R^n} |u|^{p+1} dx \quad \text{for } u \in D^{1,2}(\R^n),\]
and all Aubin-Talenti bubbles share the same energy level: $J(U_{\delta,\xi}) = \frac{1}{n}S_0^{\frac{n}{2}}.$

A key role is played by the critical parameter
\begin{equation}\label{last}
\lambda_* := \inf\{\lambda>0 : S_\lambda < S_0\}.
\end{equation}
In their seminal work \cite{BN}, Brezis and Nirenberg demonstrated that for $n \geq 4$, positive solutions exist for all $\lambda \in (0, \lambda_1)$,
where $\lambda_*=0$ and $\lambda_1$ is the first Dirichlet eigenvalue of $-\Delta$ on $\Omega$. In dimension $n=3$, they showed that $\lambda_*>0$, and established existence results for $\lambda \in (\lambda_*, \lambda_1)$.
On the unit ball $\Omega = B(0,1)$, explicit computation yields $\lambda_* = \lambda_1/4$.
Nonexistence results emerge from various mechanisms: Testing the equation against the first eigenfunction eliminates the possibility of positive solutions when $\lambda \ge \lambda_1$,
and Pohozaev's identity \cite{Po} prohibits nontrivial solutions for $\lambda \le 0$ in star-shaped domains. Conversely, Bahri and Coron \cite{BC} illustrated that certain topological features can allow for existence even at $\lambda = 0$.

\medskip
Apart from these existence results, the Brezis-Nirenberg problem \eqref{BN_eq} serves as a fundamental model for understanding bubbling phenomena in nonlinear PDEs.
As $\lambda \to \lambda_*$, solutions exhibit rich concentration behaviors. Early contributions by Han \cite{H1991} and Rey \cite{Re} characterized single-bubble blow-up profiles for $n \ge 4$, which was extended to the case $n=3$ by Druet \cite{D}.
The existence of single- or multi-bubble solutions concentrating at distinct isolated points was studied by Rey \cite{Re} and Musso and Pistoia \cite{MP} for $n \ge 5$, and by Musso and Salazar \cite{MS} for $n = 3$, and by Pistoia, Rago, and Vaira \cite{PRV} for $n = 4$.
Furthermore, Cao, Luo, and Peng \cite{CLP} studied the number of concentrated solutions for $n \ge 6$, Druet and Laurain \cite{DL} examined the Pohozaev obstruction for $n=3$, and K\"onig and Laurain \cite{KL1, KL2} conducted a fine multi-bubbles analysis for $n \ge 3$.
In addition, it is worth noting that, to our best knowledge, the existence of positive cluster or tower solutions for the Brezis-Nirenberg problem remains an open question.
This problem appears to be even more challenging than the sign-changing case, which has been extensively studied.
For the results concerning sign-changing solutions, we refer interested readers to the recent papers \cite{LVWW, SWY} and the references therein.

\medskip

In this paper, we aim to investigate the \textit{quantitative stability} of the Brezis-Nirenberg problem, a topic that has attracted considerable attention of researchers, with numerous generalizations and refinements in various directions.

\medskip
One prominent line of research concerns the stability of functional inequalities. The study of sharp functional inequalities naturally proceeds through three stages:
Identifying optimal constants, characterizing extremal functions, and understanding quantitative stability.
Once extremal functions are established, a fundamental question arises: How does the \emph{deficit}--the difference between the two sides of the inequality at the sharp constant--influence the distance to the set of extremals?
This stability question was initially posed by Brezis and Lieb \cite{BL} and subsequently resolved for the critical Sobolev inequality \eqref{Sobolev} by Bianchi and Egnell \cite{BE}, who provided a quantitative estimate regarding the distance to Aubin-Talenti bubbles in $D^{1,2}(\R^n)$.
Extending the Bianchi-Egnell stability result to general $L^p$-Sobolev inequalities has required the development of novel techniques, with major contributions from Cianchi, Figalli, Maggi, Neumayer, Pratell and Zhang \cite{CFMP, FMP, FN, FZ}.
Related advances have been developed for a variety of Sobolev-type inequalities \cite{DE, DEFFL, WW1, WW4}, and so on.
Furthermore, a recent progress has also been achieved in geometric contexts, including product spaces \cite{Fr} and general Riemannian manifolds \cite{ENS, NV, NP, AKRW, BCB}.
Notably, K\"onig’s recent breakthroughs \cite{Ko1, Ko2, Ko3} on the attainability of the sharp Bianchi-Egnell constant represent a significant milestone in the pursuit of optimal stability constants.

\medskip
Another significant direction focuses on stability through the viewpoint of the Euler--Lagrange equation induced by a sharp inequality. This perspective refines the classical concentration--compactness principle (refer to Theorem \ref{thm:Struwe}) by providing explicit convergence rates.
In a seminal work \cite{CFM}, Ciraolo, Figalli, and Maggi established the sharp stability result near a single-bubble for the Sobolev inequality in dimensions $n \geq 3$, with extensions to multiple-bubble configurations by Figalli and Glaudo \cite{FG} and Deng, Sun, and Wei \cite{DSW}.
Specifically, suppose that $\nu \in \N$ and $u$ is a nonnegative element in $D^{1,2}(\R^n)$ with $(\nu-\frac{1}{2})S_0^{n/2} \le \|u\|_{D^{1,2}(\R^n)}^2 \le (\nu+\frac{1}{2})S_0^{n/2}$ and sufficiently small $\Gamma(u):= \|\Delta u+u^{\frac{n+2}{n-2}}\|_{(D^{1,2}(\R^n))^{-1}}$.
Then there is a constant $C>0$ depending only on $n$ and $\nu$ such that
\begin{align}\label{strn}
\begin{medsize}
\displaystyle \left\|u-\sum_{i=1}^{\nu} U_i\right\|_{D^{1,2}(\R^n)} \le C\begin{cases}
\Gamma(u) &\text{if } n\ge 3,\; \nu=1 \text{ (by Ciraolo, Figalli and Maggi \cite{CFM})},\\
\Gamma(u) &\text{if } 3 \le n \le 5,\; \nu\ge 2 \text{ (by Figalli and Glaudo \cite{FG})},\\
\Gamma(u)|\log \Gamma(u)|^{\frac{1}{2}} &\text{if } n=6,\; \nu\ge 2 \text{ (by Deng, Sun, and Wei \cite{DSW})},\\
\Gamma(u)^{\frac{n+2}{2(n-2)}} &\text{if } n \ge 7,\; \nu\ge 2 \text{ (by Deng, Sun, and Wei \cite{DSW})}
\end{cases}
\end{medsize}
\end{align}
for some bubbles $U_1,\ldots,U_{\nu}$ and this estimate is optimal.
These results have been further generalized to a broad range of inequalities, including the fractional Sobolev inequality \cite{Ar, DK, CKW}, the Caffarelli-Kohn-Nirenberg inequality \cite{WW1, WW2},
the logarithmic Sobolev inequality \cite{WW3}, Sobolev inequalities involving $p$-Laplacian \cite{CG, LZ}, the subcritical case \cite{CFL}, as well as settings on the hyperbolic spaces \cite{BGKM1, BGKM2} and general Riemannian manifolds \cite{CK, CKW2}, and so forth.

\medskip
Beyond their intrinsic interest, quantitative stability estimates have powerful applications in nonlinear PDE dynamics, such as the asymptotic behavior of solutions to the Keller-Segel system \cite{CF} and the fast diffusion equation \cite{CFM, FG, DK, KY}.

\medskip
%Within our knowledge, two research directions related to the inequality $H^1_0(\Omega)\hookrightarrow L^{2^*}(\Omega)$ in bounded domains $\Omega$ are both open.
Our present work is interested in the latter direction, devoted to the quantitative stability of almost solutions to the Euler-Lagrange equation associated with the inequality $H^1_0(\Omega)\hookrightarrow L^{2^*}(\Omega)$ in bounded domains $\Omega$.
We begin with a well-known global compactness result associated with the functional corresponding to \eqref{BN_eq}, commonly referred to as Struwe's decomposition.
This result was established in \cite[Proposition 2.1]{S}, \cite[Theorem 2]{BrC} and \cite[Proposition 4]{BC}, which we restate below.
\begin{thmx}\label{thm:Struwe}
Let $\Omega$ be an open bounded domain in $\R^n$ with $n\ge 3$ and $\lambda_1$ be the first eigenvalue of $-\Delta$ with Dirichlet boundary condition in $\Omega$. For $\lambda\in (0, \lambda_1)$, we endow the Sobolev space $H^1_0(\Omega)$ with the norm
\[\|u\|_{H^1_0(\Omega)} := \left[\int_{\Omega} \(|\nabla u|^2-\lambda u^2\) dx\right]^{\frac{1}{2}},\]
and denote by $(H^1_0(\Omega))^*$ its dual space.

Let $\{u_m\}_{m \in \N}$ be a sequence of nonnegative functions in $H^1_0(\Omega)$ such that
\[\|u_m\|_{H^1_0(\Omega)} \le C_0 \quad \text{and} \quad \left\|\Delta u_m+\lambda u_m+u_m^{p}\right\|_{(H^1_0(\Omega))^*} \to 0 \text{ as } m \to \infty\]
for some constant $C_0>0$. Then, up to a subsequence, there exist a nonnegative function $u_0 \in C^{\infty}(\Omega)$, an integer $\nu \in \N \cup \{0\}$ satisfying $\nu \le C_0^2S_0^{-n/2}$,
and a sequence of parameters $\{(\delta_{1,m},\ldots,\delta_{\nu, m}, \xi_{1,m},\ldots,\xi_{\nu, m})\}_{m \in \N}$ $\subset (0,\infty)^{\nu} \times \Omega^{\nu}$ such that the followings hold:
\begin{itemize}
\item[-] $u_0$ is a smooth solution to \eqref{BN_eq}. By the strong maximum principle, we have either $u_0>0$ or $u_0 = 0$ in $\Omega$.
\item[-] For all $1 \le i \ne j \le \nu$, we have that $\delta_{i,m} \to 0$ and
\[\frac{d(\xi_{i,m},\pa\Omega)}{\delta_{i,m}}\to \infty, \quad \frac{\delta_{i,m}}{\delta_{j,m}} + \frac{\delta_{j,m}}{\delta_{i,m}} + \frac{|\xi_{i,m}-\xi_{j,m}|^2}{\delta_{i,m}\delta_{j,m}}\to \infty \quad \text{as } m \to \infty.\]
\item[-] It holds that
\[\bigg\|u_m-\bigg(u_0+\sum_{i=1}^{\nu}U_{\delta_{i,m},\xi_{i,m}}\bigg)\bigg\|_{H^1_0(\Omega)} \to 0 \quad \text{as } m \to \infty.\]
%where each $PU_{\delta_{i,m},\xi_{i,m}}$ is a function in $H^1_0(\Omega)$ satisfying
%\[\left\|PU_{\delta_{i,m},\xi_{i,m}} - U_{\delta_{i,m},\xi_{i,m}}\right\|_{H^1_0(\Omega)} \to 0 \quad \text{as } m \to \infty \quad \text{for } i = 1,\ldots,\nu.\]
\end{itemize}	
\end{thmx}

\subsection{Main results}
Our objective is to derive a quantitative version of above decomposition. To this end, we consider the following two auxiliary equations:
\begin{equation}\label{pu1}
\begin{cases}
-\Delta u=U_{\delta,\xi}^{p} &\text{in } \Omega,\\
u=0 &\text{on } \pa\Omega,
\end{cases}
\end{equation}
and
\begin{equation}\label{pu2}
\begin{cases}
-\Delta u-\lambda u=U_{\delta,\xi}^{p} &\text{in } \Omega,\\
u=0 &\text{on } \pa\Omega.
\end{cases}
\end{equation}
Before presenting our main results, we introduce the following assumption:

\medskip \noindent \textbf{Assumption B.}
Given any open bounded set $\Omega\subset\R^n$ with $n\ge 3$ and any $\lambda\in(0, \lambda_1)$. Suppose that a nonnegative function $u$ in $H^1_0(\Omega)$ satisfies
\begin{equation}\label{tui2}
\bigg\|u-\bigg(u_0+\sum_{i=1}^{\nu}U_{\tde_i,\txi_i}\bigg)\bigg\|_{H^1_0(\Omega)} \le \vep_0
\end{equation}
for some small $\vep_0>0$ and $\nu \in \N$. Here, $u_0$ is a solution of \eqref{BN_eq}
and $(\tde_i,\txi_i) \in (0,\infty) \times \Omega$ satisfies
\[\max_{i=1,\ldots,\nu} \tde_i + \max_{i=1,\ldots,\nu}\frac{\tilde{\delta}_i}{d(\tilde{\xi}_i,\pa\Omega)}\le \vep_0\]
and
\[\max\left\{\(\frac{\tde_i}{\tde_j} + \frac{\tde_j}{\tde_i} + \frac{|\txi_i-\txi_j|^2}{\tde_i\tde_j}\)^{-\frac{n-2}{2}}: i, j=1,\ldots,\nu,\ i \ne j\right\} \le \vep_0.\]
If $u_0>0$ in $\Omega$, we further assume that $u_0$ is \textbf{non-degenerate} in the sense that the only $H_0^1 (\Omega)$-solution to $ \Delta \phi+ \lambda \phi + p u_0^{p-1}\phi=0$ in $\Omega$ is identically zero in $\Omega$.
For later use, we define $\Gamma(u):=\|\Delta u+\lambda u+u^{p}\|_{(H^1_0(\Omega))^*}$.
\medskip

We note that the condition $\max_i\frac{\tilde{\delta}_i}{d(\tilde{\xi}_i,\pa\Omega)}\le \vep_0$ admits two possibilities: Either $\tilde{\xi}_i$ is away from $\pa\Omega$ or close to $\pa\Omega$. Accordingly, we divide our main results into two theorems.

\medskip
Our first theorem addresses the case where $\tilde{\xi}_i$ is away from the boundary of $\Omega$, covering both single and multi-bubble cases.
\begin{thm}\label{thm1}
Under the \textbf{Assumption B}, we further assume the followings:
\begin{itemize}
\item[-] Each $\tilde{\xi}_i$ lies on a compact set of $\Omega$ for $i=1,\ldots,\nu$.
\item[-] If $n=3$, $u_0=0$, $\nu \ge 2$, and a number $\lambda \in (\lambda_*, \lambda_1)$ is fixed, then $\vph^3_\lambda(\tilde{\xi}_i)<0$ for each $i=1,\ldots,\nu$.
Here, $\lambda_* \ge 0$ is given in \eqref{last}, and $\vph_\lambda^3(x)=H^3_\lambda(x,x)$ is the function defined by \eqref{h1xy}.\footnote{Druet \cite{D} proved that
the number $\lambda_*$ in \eqref{last} can be characterized as $\lambda_*=\sup\{\lambda>0: \min_{\Omega}\vph_{\lambda}^3>0\}$.}
\end{itemize}
Then, by possibly reducing $\vep_0>0$, one can find a large constant $C=C(n, \nu, \lambda, u_0, \Omega)>0$ and
functions $PU_1 := PU_{\delta_1,\xi_1}, \ldots, PU_{\nu} := PU_{\delta_{\nu},\xi_{\nu}}$ satisfying \eqref{pu1} if either \textup{[}$n=3,4$ and $u_0>0$\textup{]} or $n\ge 5$, and satisfying \eqref{pu2} if $n=3,4$ and $u_0=0$, such that
\begin{equation}\label{eq:sqe32}
\bigg\|u-\bigg(u_0+\sum_{i=1}^{\nu} PU_i\bigg)\bigg\|_{H^1_0(\Omega)} \le C\zeta(\Gamma(u)),
\end{equation}
where $\zeta \in C^0([0,\infty))$ satisfies
\begin{equation}\label{eq:zeta}
\zeta(t) = \begin{cases}
t &\text{if } [n=3,4,\; \nu\ge 1] \text{ or } [n=5,\; \nu\ge 1,\; u_0>0] \text{ or } [n\ge 7,\; \nu=1],\\
t^{\frac34} &\text{if } [n=5,\; \nu\ge 1,\; u_0=0],\\
t|\log t|^{\frac{1}{2}} &\text{if } [n=6,\; \nu\ge 1],\\
t^{\frac{n+2}{2(n-2)}} &\text{if } [n\ge 7,\; \nu\ge 2]
%\text{ or }[n\ge 7, \nu=1, u_0>0]
\end{cases}
\end{equation}
for $t>0$.

The estimate above is optimal in the sense that the function $\zeta$ cannot be improved.
\end{thm}

Before we proceed further, we leave some remarks.
\begin{rmk}\label{rmk:thm1}
\

\medskip \noindent (1) Compared to the Euclidean case summarized in \eqref{strn}, the new exponents appear when $[n=5,\; u_0=0,\; \nu \ge 1]$ or $[n=6,\; \nu = 1]$.

\medskip \noindent (2) Solutions to certain specific perturbation of the equation $\Delta u +\lambda u + u^p=0$ in $\Omega$ cannot exhibit boundary blow-up, thereby fulfilling the first additional assumption in Theorem \ref{thm1}.
Moreover, in some cases, only one of the conditions $u_0=0$ or $\nu = 0$ is permitted; refer to e.g. \cite{KL1, Dr}.

\medskip \noindent (3) When $u_0>0$, the non-degeneracy assumption on $u_0$ is generic; see \cite[Lemma 4.9]{JX}.
In the case $u_0=0$ and $n=3$ or $4$, defining $PU_i$ via solutions to \eqref{pu2} rather than \eqref{pu1} turns out to be more natural; see Subsection \ref{subc1.3}(2).
Similar observation was made in constructing positive solutions to the Brezis-Nirenberg-type problem in low dimensions; see e.g. \cite{dDM}.

\medskip \noindent (4) For $n=3$, $u_0=0$, and $\nu \ge 2$, we use the condition $\vph^3_\lambda(\tilde{\xi}_i)<0$ so that no sign competition occurs between the terms $\int_{\Omega} \I_2 PZ_j^0$ in Lemma \ref{lem2.33} and $\int_{\Omega} \I_3 PZ_j^0$ in Lemma \ref{lem2.32}.

\medskip \noindent (5) If $n=5$, $u_0=0$, and $\nu \ge 1$, the linear term $\lambda u$ is the dominant factor determining $\zeta(t) = t^{3/4}$ in \eqref{eq:zeta}.
In this case, one may instead choose the projected bubble $PU_{\delta,\xi}$ as in \eqref{pu2} rather than \eqref{pu1}.
Since \eqref{pu2} already incorporates the effect of the linear term, it leads to the stability function $\zeta(t) = t$, as opposed to $t^{3/4}$, and this improved rate can again be shown to be sharp.
Such a sensitive dependence on the choice of the test function is a distinctive characteristic of the Brezis-Nirenberg problem in $\Omega$, and does not appear in the Euclidean setting or in the Yamabe problem.
\end{rmk}

Our second main result concerns the boundary effect when $\tilde{\xi}_i$ may approach $\pa\Omega$. We fully characterize the single-bubble case in this setting.
\begin{thm}\label{thm2}
Under the \textbf{Assumption B}, we further assume that $\nu=1$ and $\tilde{\xi}_1\in\Omega$.
Then, by possibly reducing $\vep_0>0$, one can find a large constant $C=C(n, \lambda, u_0, \Omega)>0$ and
a function $PU_1 := PU_{\delta_1,\xi_1}$ satisfying \eqref{pu1} if either \textup{[}$n=4,5$ and $u_0>0$\textup{]} or $n\ge 6$, and satisfying \eqref{pu2} if either $n=3$ and \textup{[}$n=4,5$ and $u_0=0$\textup{]}, such that
\begin{equation}\label{eq:sqe322}
\left\|u-\(u_0+PU_1\)\right\|_{H^1_0(\Omega)} \le C\zeta(\Gamma(u)),
\end{equation}
where $\zeta \in C^0([0,\infty))$ satisfies
\begin{equation}\label{eq:zeta122}
\zeta(t) = \begin{cases}
t &\text{if } n=3\text{ or } [n=4,\; u_0=0],\\
t^{\frac{n-2}{n-1}} &\text{if } [n=4,\; u_0>0] \text{ or } n=5,\\
t|\log t|^{\frac{1}{2}} &\text{if } n=6,\\
t^{\frac{n+2}{2(n-1)}} &\text{if } n\ge 7
\end{cases}
\end{equation}
for $t>0$. The above estimate is also optimal.
\end{thm}
\begin{rmk}\label{rmk:thm2}
\

\medskip \noindent (1) Even in single-bubble case, the surprising new exponents in \eqref{eq:zeta122} emerge due to the possibility $d(\tilde{\xi}_1,\pa\Omega)$ $\to 0$.
This phenomenon occurs exclusively in domains with nonempty boundary. The multi-bubble case remains an open problem due to a serious technical issue. See Subsection \ref{subc1.3}(7).

\medskip \noindent (2) Unlike in Theorem \ref{thm1}, we choose $PU_1$ to satisfy \eqref{pu2} for the cases $[n=3,\; u_0>0]$ or $[n=5,\; u_0=0]$ to avoid difficulties arising from the boundary effects. We believe that this choice is nearly unavoidable.

\medskip \noindent (3) Similar to Remark \ref{rmk:thm1}(5), when $[n=4,\; u_0>0]$ or $[n=5,\; u_0>0]$, choosing $PU_{\delta,\xi}$ as in \eqref{pu2} again yields the optimal stability function $\zeta(t)=t$.
In both cases, the sharp stability function depends explicitly on the choice of the projected bubble $PU_{\delta,\xi}$ within the framework of this theorem.
\end{rmk}
As an application of Theorem \ref{thm2} and Struwe’s profile decomposition Theorem \ref{thm:Struwe}, we obtain the following corollary.
\begin{cor}
Let $S_0>0$ be the sharp Sobolev constant in \eqref{Sobolev}. We assume that every positive solution to \eqref{BN_eq} is non-degenerate.

\medskip \noindent If $u$ is a nonnegative function in $H^1_0(\Omega)$ with
\begin{equation}\label{eneu}
\|u\|_{H^1_0(\Omega)}^2 \le \frac{3}{2}S_0^{\frac{n}{2}},
\end{equation}
then there exists a constant $C>0$ depending only on $n, \lambda, \Omega$ such that
\[\inf\left\{\bigg\|u-\bigg(u_0+\sum_{i=1}^{\nu} PU_{\delta_i,\xi_i}\bigg)\bigg\|_{H^1_0(\Omega)}: u_0 \text{ solves } \eqref{BN_eq},\, PU_{\delta_i,\xi_i} \in \mcb,\, \nu = 0,1\right\} \le C\zeta(\Gamma(u)),\]
where $\zeta(t)$ satisfies \eqref{eq:zeta122} for $t \in [0,\infty)$ and
\begin{multline*}
\mcb := \big\{PU_{\delta,\xi}: PU_{\delta,\xi} \text{ satisfies \eqref{pu1} for $n\ge 6$ or $[n=4,5,\; u_0>0]$}\\
\text{and satisfies \eqref{pu2} for $n=3$ or $[n=4,5,\; u_0=0]$},\, (\delta,\xi) \in (0,\infty) \times \Omega\big\}.
\end{multline*}
Here $\sum_{i=1}^0 PU_{\delta_i,\xi_i} = 0$.
\end{cor}

\begin{rmk}
In this corollary, we modify the class of admissible functions $u$ in \eqref{tui2} to those with uniformly bounded energy as in \eqref{eneu}.
This necessitates assuming the non-degeneracy for all positive solutions to \eqref{BN_eq}, since $u_0$ cannot be determined a priori.
The proof proceeds by contradiction, following an argument similar to that in \cite[Section 6]{CK}, and is therefore omitted.
\end{rmk}

\subsection{Comments on the proof}\label{subc1.3}\
Our proof is primarily inspired by the approaches developed in \cite{CFM, FG, DSW, CK, CKW}. To clarify the new technical challenges involved in our setting, we begin by outlining a general strategy for proving quantitative stability of sharp inequalities in the critical point setting:
\begin{itemize}
\item[(i)] The starting point is that the infimum $\inf\|u-(u_0+\sum_{i=1}^{\nu}\mcv_{\tilde{\delta}_i,\tilde{\xi}_i})\|_{H^1}$ can be achieved by $\mcv_{\delta_i,\xi_i}$ where $\mcv_{\delta,\xi}$ is an appropriate bubble-like function,
    then $\rho:=u-(u_0+\sum_{i=1}^{\nu}\mcv_{\delta_i,\xi_i})$ satisfies an auxiliary equation (e.g. \eqref{eqrho2}) along with certain orthogonality conditions, at least in a Hilbert space framework.

\item[(ii)] By testing the equation of $\rho$ with $\rho$ itself, one can derive a rough estimate $\|\rho\|_{H^1} \lesssim \|f\|_{H^{-1}}+\|\I\|_{H^{-1}}$ where $f := -\Delta u-\lambda u-u^p$, and $\I$ is an error term (in our setting, $\I:=\I_1+\I_2+\I_3$ given by \eqref{eq:I2}).
    While this estimate may not be sufficient to achieve a sharp stability function $\zeta$, it can be often improved through a linear theory.
    In fact, the linear theory provides a pointwise estimate of the main part $\rho_0$ of $\rho$, leading to a refined estimate of the form $\|\rho\|_{H^1} \lesssim \|f\|_{H^{-1}} + \mathcal{J}_1$ where $\mcj_1$ is a small quantity.

\item[(iii)] By choosing suitable test functions originated from bubbles (see Subsection \ref{subc1.3}(7) below), one can find another small quantity $\mcj_2$ such that $\mcj_2\lesssim \|f\|_{H^{-1}}$.
    If one can determine a function $\tilde{\zeta}$ such that $\mcj_1\lesssim \tilde{\zeta}(\mcj_2)$, the final stability function will be determined as $\zeta(t):=\max\{t, \tilde{\zeta}(t)\}$ for small $t>0$.

\item[(iv)] Once one finds special parameters $(\delta_i,\xi_i)$ and functions $\rho$ and $f$ satisfying $\|\rho\|_{H^1} \gtrsim \zeta(\|f\|_{H^{-1}})$,
    then the nonnegative function $u_*=(u_0+\sum_{i=1}^{\nu}\mcv_{\delta_i,\xi_i}+\rho)_+$ usually provides an optimal example.
\end{itemize}

\medskip
Although our proof could follow the procedures outlined above, several non-trivial and novel challenges arise in our specific setting. We now present the new strategies devised to overcome or mitigate these difficulties.

\medskip
\noindent{(1)} Due to the presence of $u_0$ and the linear term $\lambda u$, more precise computations are needed for the interactions among bubbles with various powers, as well as those between a bubble and $u_0$, for all dimensions $n\ge 3$.

\medskip
\noindent{(2)} The selection of bubble-like functions is subtle.
For our problem, depending on the dimension $n$ and the solution $u_0$ of \eqref{BN_eq}, we make appropriate use of two different projected bubbles given by \eqref{pu1} and \eqref{pu2}.

Let us explain why we must define $PU_i$ via solutions of \eqref{pu2} in deducing Theorem \ref{thm1} for $n=3$ or $4$ and $u_0=0$:

If $n=3$ and $u_0=0$, then the function $PU_i$ defined via \eqref{pu1} fails to produce any quantitative estimates even in the single-bubble case due to the excessive size of $\|U_{\delta,\xi}\|_{L^{6/5}(\Omega)}$.

If $n=4$, $u_0=0$, and $\nu=1$, then such a definition yields a valid but a non-sharp estimate.

If $n=4$, $u_0=0$, and $\nu \ge 2$, then the use of the above-defined $PU_i$ fails completely, because the interaction terms $\int U_iU_j$ are not negligible compared to the presumably dominant term $\max_i \int U_i^2$.

In Lemmas \ref{lem2.1} and \ref{lema6}, we rigorously analyze the behavior of the function $PU_{\delta, \xi}$ defined via \eqref{pu2}, which may be independent of interests.

As previously noted, there seems be no results on positive cluster or tower solutions for the low-dimensional Brezis-Nirenberg problem. Our calculations take into account all possible bubbles and may be helpful for constructing such solutions, should they exist.

\medskip
\noindent{(3)} In \cite{DSW, CKW}, the authors employed a pointwise estimate across all bubble configurations for the main part of $\rho$ in all dimensions $n \ge 6$.
Our proof of stability estimates \eqref{eq:sqe32} shows that such an estimate is required only when $n = 6$ in our setting and \cite{DSW, CKW}.
For the optimality proof, a pointwise estimate for the main part of $\rho$ is still needed in many other dimensions, but only for specific configurations.
This insight simplifies the technical aspects of the argument (cf. \cite[Lemma 4.2]{DSW}). See Section \ref{sec2.2}, Subsection \ref{subs3.2}.

\medskip
\noindent (4) To develop the linear theory for $n=6$ in Section \ref{sec2.2}, we make extensive use of the representation formula, which offers a unified and simplified proof compared to approaches based on the maximum principle (cf. \cite[Lemma 5.1]{DSW}).
This idea was initially developed in our previous work \cite{CKW}, where we studied the quantitative stability of the fractional Sobolev inequality $\dot{H}^s(\R^n) \hookrightarrow L^{\frac{2n}{n-2s}}(\R^n)$ of all orders $s \in (0,\frac{n}{2})$.

\medskip
\noindent{(5)} In Step (ii), many seminal works in the critical regime (see \cite{CFM, FG, DSW} and their generalizations, e.g., \cite{BGKM2, CK, CKW}) devote substantial effort to deriving appropriate coercivity inequalities.
In \cite[Section 6]{DSW}, such inequalities play a crucial role in deducing a Sobolev norm estimate for the term $\rho_1 := \rho-\rho_0$.
In contrast, our approach provides a direct derivation of the Sobolev norm estimate for $\rho_1$ based solely on blow-up analysis (refer to Subsection \ref{subs3.1}).
As a result, the proof avoids coercivity inequalities entirely, greatly simplifying the argument again.

\medskip
\noindent{(6)} Regarding the sharpness of our results, we conduct a comprehensive analysis of all admissible forms of the function $\zeta$, dealing with the linear ($\zeta(t) = t$) and sublinear ($\zeta(t) \gg t$) regimes through two distinct strategies.
In the linear case, sharpness is verified by constructing a smooth perturbation of $u_0 + \sum_{i=1}^{\nu} PU_i$.
For the sublinear case, a more delicate analysis is required for the multiple bubble scenario whose idea differs from that in $\R^n$,
and it is important to identify which of the dominant factors--interactions among $u_0$, the boundary effect, the bubbles, and the linear term $\lambda u$--govern the exponent of $\zeta$.

\medskip
\noindent{(7)} In the proof of Theorem \ref{thm2}, the scenario in which $d(\xi,\pa \Omega)$ is small introduces a crucial challenge:
The projection of $\I_1 + \I_2 + \I_3$ in the direction of the dilation derivative $\delta_i \pa_{\delta_i} \mcv_i$ of the bubble-like function $\mcv_i$ has a negative leading term of the form $-\delta^{n-2}/d(\xi, \pa \Omega)^{n-2}$; see \eqref{iz0}.
In the single-bubble case, we address this projection by carefully analyzing all possible scenarios, as detailed in Section \ref{sec4}.
The reason that one primarily uses $\delta_i \pa_{\delta_i} \mcv_i$ as a test function in both Euclidean and manifold settings--instead of using a spatial derivative $\delta_i \pa_{\xi_i^k} \mcv_i$--is that the latter generally lead to weaker estimates.
However, in our setting, it is sometimes necessary to consider the projections of $\I_1 + \I_2 + \I_3$ in the direction $\delta_i \pa_{\xi_i^k} \mcv_i$, since the dilation projection may suffer from sign cancellations among its leading-order terms, weakening their overall effect.
As such, precise term-by-term estimates like \eqref{iz0} and \eqref{izk} are indispensable.

In the multi-bubble case $\nu \ge 2$, these challenges become significantly more difficult.
We currently lack a clear strategy to effectively handle the competition between the negative term involving $d(\xi_i,\pa\Omega)$ and the interaction between different bubbles.
Additionally, integrals such as $\int_{\Omega} [(PU_i)^p - U_i^p] U_j$ for $i \ne j$ and $\int_{\Omega} [(\sum_{i=1}^{\nu}PU_i)^p-\sum_{i=1}^{\nu}(PU_i)^p]PZ_j^0$ (when $n \ge 3$),
and cross terms like $U_i w_{3j}^{\tin}$ and $U_i w_{3j}^{\tout}$ (when $n=6$, cf. Definition \ref{norm}) in the linear theory, pose formidable analytical obstacles.

\medskip
Our structure of this paper is described as follows:
In Section \ref{sec2}, we present some necessary estimates for proving our main theorems. In Section \ref{sec2.2}, we improve the estimate for the main part of $\rho$ in dimension 6 based on a linear theory.
In Sections \ref{sec3} and \ref{sec4}, we provide the detailed proofs of Theorem \ref{thm1} and Theorem \ref{thm2}, respectively.
In Appendix \ref{a}, we include several elementary estimates that are frequently used throughout the main text.
In Appendix \ref{b}, we give a proof for an important lemma used in Section \ref{sec4}.

\medskip

\noindent\textbf{Notations.} Here, we list some notations that will be used throughout the paper.

\medskip \noindent - $\N$ denotes the set of positive integers.

\medskip \noindent - Let (A) be a condition. We set $\mone_{\text{(A)}} = 1$ if (A) holds and $0$ otherwise.

\medskip \noindent - For $x \in \Omega$ and $r>0$, we write $B(x,r) = \{\om \in \Omega: |\om-x| < r\}$ and $B(x,r)^c = \{\om \in \Omega: |\om-x| \ge r\}$.

\medskip \noindent - We use the Japanese bracket notation $\la x \ra = \sqrt{1+|x|^2}$ for $x \in \R^n$.

\medskip \noindent - Unless otherwise stated, $C>0$ is a universal constant that may vary from line to line and even in the same line.
We write $a_1 \lesssim a_2$ if $a_1 \le Ca_2$, $a_1 \gtrsim a_2$ if $a_1 \ge Ca_2$, and $a_1 \simeq a_2$ if $a_1 \lesssim a_2$ and $a_1 \gtrsim a_2$.

\section{Setting and analysis of bubbles}\label{sec2}
\subsection{Problem setting}
By \eqref{tui2}, there exist parameters $(\delta_1,\ldots,\delta_{\nu}, \xi_1,\ldots,\xi_{\nu}) \subset (0,\infty)^{\nu} \times \Omega^{\nu}$ and $\vep_1>0$ small such that $\vep_1 \to 0$ as $\vep_0 \to 0$,
\[\begin{medsize}
\displaystyle \bigg\|u-\bigg(u_0+\sum_{i=1}^{\nu} PU_i\bigg)\bigg\|_{H^1_0(\Omega)} = \inf\left\{\bigg\|u-\bigg(u_0+\sum_{i=1}^{\nu} PU_{\tde_i,\txi_i}\bigg)\bigg\|_{H^1_0(\Omega)}: \(\tde_i,\txi_i\) \in (0,\infty) \times \Omega,\ i=1,\ldots,\nu\right\} \le \vep_1,
\end{medsize}\]
where $PU_i = PU_{\delta_i,\xi_i}$, and
\[\max_i\delta_i + \max_i\frac{\delta_i}{d(\xi_i,\pa\Omega)} \le \vep_1 \]
as well as
\[\max\left\{\(\frac{\delta_i}{\delta_j} + \frac{\delta_j}{\delta_i} + \frac{|\xi_i-\xi_j|^2}{\delta_i\delta_j}\)^{-\frac{n-2}{2}}: i,j = 1,\ldots,\nu\right\} \le \vep_1.\]
Setting $\sigma:=\sum_{i=1}^{\nu} PU_i$, $\rho:=u-(u_0+\sigma)$, and $f:=-\Delta u-\lambda u-u^{p}$, we have
\begin{equation}\label{eqrho2}
\begin{cases}
\displaystyle -\Delta \rho-\lambda \rho - p(u_0+\sigma)^{p-1}\rho = f+\I_0[\rho]+\I_1+\I_2+\I_3 \quad \text{in } \Omega,\\
\displaystyle \rho=0 \quad \text{on } \pa\Omega,\\
\displaystyle \big\langle \rho, PZ_i^k \big\rangle_{H^1_0(\Omega)} := \int_{\Omega}\nabla \rho \cdot \nabla PZ_i^k-\lambda\rho PZ_i^k= 0 \quad \text{for } i=1,\ldots,\nu \text{ and } k=0,\ldots,n,
\end{cases}
\end{equation}
where
\begin{equation}\label{eq:I2}
\begin{aligned}
&PZ_i^0:=\delta_i\frac{\pa PU_i}{\pa \delta_i}, \quad PZ_i^k:=\delta_i\frac{\pa PU_i}{\pa\xi_i^k} \quad \text{for } k=1,\ldots,n,\\
&\I_0[\rho] := |u_0+\sigma+\rho|^{p-1}(u_0+\sigma+\rho) - (u_0 + \sigma)^{p} - p (u_0 + \sigma)^{p-1}\rho,\\
&\I_1 := (u_0 + \sigma)^{p}-u_0^{p}-\sigma^{p},\\
&\I_2 := \sigma^{p} -\sum_{i=1}^{\nu} (PU_i)^p, \quad \text{and} \quad \I_3 := \sum_{i=1}^{\nu} \left[\Delta PU_i+\lambda PU_i+ (PU_i)^p\right].
\end{aligned}
\end{equation}

We recall a well-known non-degeneracy result: Given any $\delta>0$ and $\xi = (\xi^1,\ldots,\xi^n) \in \R^n$, the solution space of the linear problem
\[-\Delta v=p U_{\delta,\xi}^{p-1} v \quad \text { in } \R^n, \quad v \in D^{1,2}(\R^n)\]
is spanned by the functions
\[Z_{\delta,\xi}^0 := \delta \frac{\pa U_{\delta,\xi}}{\pa \delta} \quad \text{and} \quad Z_{\delta,\xi}^k := \delta \frac{\pa U_{\delta,\xi}}{\pa \xi^k} \quad \text{for } k=1,\ldots,n.\]
We rewrite $U_i:=U_{\delta_i,\xi_i}$, $Z^k: = Z_{1,0}^k$, and $Z_i^k:=Z_{\delta_i,\xi_i}^k$ for $i=1,\ldots,\nu$ and $k=0,\ldots,n$.

Let us define four quantities
\begin{equation}\label{rq}
\begin{cases}
\displaystyle q_{ij} := \left[{\frac{\delta_i}{\delta_j}} + {\frac{\delta_j}{\delta_i}} + \frac{|\xi_i-\xi_j|^2}{{\delta_i\delta_j}}\right]^{-\frac{n-2}{2}}, \quad Q := \max\{q_{ij}: i,j = 1,\ldots,\nu\} \le \vep_1, \\
\displaystyle \msr_{ij} := \max\left\{\sqrt{\frac{\delta_i}{\delta_j}}, \sqrt{\frac{\delta_j}{\delta_i}}, \frac{|\xi_i-\xi_j|}{\sqrt{\delta_i\delta_j}}\right\} \simeq q_{ij}^{-\frac{1}{n-2}},\quad \msr := \frac{1}{2} \min \msr_{ij}.
\end{cases}
\end{equation}

\subsection{Expansions of $PU_{\delta,\xi}$}
Given the projected bubbles $PU_{\delta,\xi}$ via either \eqref{pu1} or \eqref{pu2}, we expand them.
\begin{lemma}\label{lem2.1}
Suppose that $x,\, \xi\in\Omega$ and $\delta>0$ is small. Then, $0<PU_{\delta,\xi}\le U_{\delta,\xi}$ in $\Omega$, and for any $\tau\in(0,1)$, the following holds:
\[PU_{\delta,\xi}(x) = U_{\delta,\xi}(x) - a_n\delta^{\frac{n-2}{2}}H(x,\xi) + O\(\delta^{\frac{n+2}{2}}d(\xi,\pa\Omega)^{-n}\)\]
provided $n \ge 3$ and $PU_{\delta,\xi}$ is given by equation \eqref{pu1}, and
\begin{align*}
\begin{medsize}
\displaystyle PU_{\delta,\xi}(x)
\end{medsize}
&\begin{medsize}
\displaystyle = U_{\delta,\xi}(x) + \frac{\lambda}{2}a_n \delta^{\frac{n-2}{2}}
\left.\begin{cases}
-|x-\xi| &\text{if } n=3\\
-\log|x-\xi| &\text{if } n=4 \\
\frac{1}{|x-\xi|}-4\lambda|x-\xi| &\text{if } n=5
\end{cases}\right\}
-\delta^{\frac{n-2}{2}} a_nH^n_\lambda(x,\xi)+\delta^{2-\frac{n-2}{2}}\mcd_n\(\frac{x-\xi}{\delta}\)
\end{medsize} \\
&\begin{medsize}
\displaystyle \ + \left.\begin{cases}
O(\delta^{\frac52-\tau}) &\text{if } n=3, 5\\
O(\delta^{3-\tau}) &\text{if } n=4
\end{cases}\right\}
+O\(\delta^{\frac{n+2}{2}} \left[d(\xi,\pa\Omega)^{-(n-2)}\bigg|\log\frac{d(\xi,\pa\Omega)}{\delta}\bigg| + d(\xi,\pa\Omega)^{-n}\right]\)
\end{medsize}
\end{align*}
provided $n=3,4,5$ and $PU_{\delta,\xi}$ is given by equation \eqref{pu2}. Here, $a_n=(n(n-2))^{\frac{n-2}{4}}$ (see \eqref{AT bubbles}), the function $H(x,y)$ satisfies
\[\begin{cases}
-\Delta_x H(x,y)=0 &\text{in } \Omega,\\
H(x,y)=\frac{1}{|x-y|^{n-2}} &\text{on } \pa\Omega,
\end{cases}\]
the function $H^3_\lambda(x,y)$ satisfies
\begin{equation}\label{h1xy}
\begin{cases}
\Delta_x H^3_\lambda(x,y)+\lambda H^3_\lambda(x,y)=-\frac{\lambda^2}{2}|x-y| &\text{in } \Omega,\\ H^3_\lambda(x,y)=\frac{1}{|x-y|}-\frac{\lambda}{2}|x-y| &\text{on }\pa \Omega,
\end{cases}
\end{equation}
the function $H^4_{\lambda}(x,y)$ satisfies
\begin{equation}\label{h2xy}
\begin{cases}
\Delta_x H^4_\lambda(x,y)+\lambda H^4_\lambda(x,y)=\lambda \log|x-y| &\text{in } \Omega,\\ H^4_\lambda(x,y)=\frac{1}{|x-y|^2}-\frac{\lambda}{2}\log|x-y| &\text{on } \pa \Omega,
\end{cases}
\end{equation}
and the function $H^5_{\lambda}(x,y)$ satisfies
\begin{equation}\label{h3xy}
\begin{cases}
\Delta_x H^5_\lambda(x,y)+\lambda H^5_\lambda(x,y)=-\frac{\lambda^2}{2}|x-y| &\text{in } \Omega,\\ H^5_\lambda(x,y)=\frac{1}{|x-y|^3}+\frac{\lambda}{2}\frac{1}{|x-y|}-\frac{\lambda^2}{2}|x-y| &\text{on } \pa \Omega,
\end{cases}
\end{equation}
for each fixed $y \in \Omega$. Besides, the function $\mcd_n=\mcd_n(z)$ satisfies
\[\begin{cases}
-\Delta \mcd_n=\lambda a_n\left[\frac{1}{(1+|z|^2)^{\frac{n-2}{2}}}-\frac{1}{|z|^{n-2}}\right] &\text{in } \R^n,\\
\mcd_n\approx |z|^{-(n-2)}|\log|z|| &\text{as } |z|\to \infty.
\end{cases}\]
\end{lemma}
\begin{proof}
The inequality $0<PU_{\delta,\xi}\le U_{\delta,\xi}$ in $\Omega$ holds by the maximum principle.

The proof for the case where $PU_{\delta,\xi}$ satisfies \eqref{pu1}, or it satisfies \eqref{pu2} with $n=3$, can be found in \cite[Proposition 1]{Re} or \cite[Lemma 2.2]{dDM}, respectively.
Here, we provide a proof for $PU_{\delta,\xi}$ satisfying \eqref{pu2} that applies to $n=3, 4, 5$ simultaneously.

Let $G_{\lambda}(x,y)$ be the Green's function of $-\Delta-\lambda$ in $\Omega \subset \R^n$ with Dirichlet boundary condition, which solves
\begin{align}\label{gla}\begin{cases}
-\Delta_x G_{\lambda}(x,y)-\lambda G_{\lambda}(x,y)=\delta_{y} &\text{in } \Omega,\\
G_{\lambda}(x,y)=0 &\text{on } \pa\Omega
\end{cases}\end{align}
in the sense of distributions. The function $G_{\lambda}(x,y)$ is symmetric with respect to the two variables $x$ and $y$.
Also, one can write
\[G_\lambda(x,y) = \frac{1}{(n-2)|\S^{n-1}|}\left[\frac{1}{|x-y|^{n-2}} - H_\lambda(x,y)\right],\]
where $|\S^{n-1}|$ is the surface area of the unit sphere $\S^{n-1}$ in $\R^n$ and $H_\lambda$ solves
\[\begin{cases}
\Delta_x H_\lambda(x,y)+\lambda H_\lambda(x,y)=\lambda \frac{1}{|x-y|^{n-2}} &\text{in } \Omega,\\ H_\lambda(x,y)=\frac{1}{|x-y|^{n-2}} &\text{on } \pa \Omega.
\end{cases}\]
We decompose $H_{\lambda}(x,y)$ as
\[H_\lambda(x,y)= \left.\begin{cases}
\frac{\lambda}{2}|x-y| &\text{if } n=3\\
\frac{\lambda}{2} \log|x-y| &\text{if } n=4\\
-\frac{\lambda}{2}\frac{1}{|x-y|}+2\lambda^2|x-y| &\text{if } n=5
\end{cases}\right\} + H^n_\lambda(x,y)\]
and apply elliptic regularity theory to ensure that $H^n_\lambda(x,y)\in C^{1,\alpha}(\Omega\times\Omega)$ for any $\alpha\in(0,1)$.

Next, we define
\[\mcs_{\delta,\xi}(x)=PU_{\delta,\xi}-U_{\delta,\xi}+a_n\delta^{\frac{n-2}{2}} H_\lambda(x,\xi)-\wtmcd_n(x).\]
Here, $\wtmcd_n(x) := \delta^{2-\frac{n-2}{2}}\mcd_n(\frac{x-\xi}{\delta})$ so that
\[\begin{cases}
-\Delta \wtmcd_n=\lambda a_n \left[\(\frac{\delta}{\delta^2+|x-\xi|^2}\)^{\frac{n-2}{2}} - \frac{\delta^{\frac{n-2}{2}}}{|x-\xi|^{n-2}}\right] &\text{in } \Omega,\\
\wtmcd_n\approx \frac{\delta^{2+\frac{n-2}{2}}}{|x-\xi|^{n-2}}\left|\log\frac{|x-\xi|}{\delta}\right| &\text{on } \pa\Omega.
\end{cases}\]
By observing that
\[\mcs_{\delta,\xi}(x) = -a_n\left[\(\frac{\delta}{\delta^2+|x-\xi|^2}\)^{\frac{n-2}{2}}-\frac{\delta^{\frac{n-2}{2}}}{|x-\xi|^{n-2}}\right] -\wtmcd_n(x) \quad \text{for } x \in \pa\Omega,\]
we obtain
\[\begin{cases}
\Delta \mcs_{\delta,\xi}+\lambda\mcs_{\delta,\xi}=\lambda\wtmcd_n &\text{in } \Omega,\\
\mcs_{\delta,\xi} = O\(\delta^{2+\frac{n-2}{2}}\left[d(\xi,\pa\Omega)^{-(n-2)}\bigg|\log\frac{d(\xi,\pa\Omega)}{\delta}\bigg| + d(\xi,\pa\Omega)^{-n}\right]\) &\text{on } \pa\Omega.
\end{cases}\]
We notice that
\[|\mcd_n(z)|\simeq \begin{cases}
|z| &\text{if } n=3,\\
|\log|z|| &\text{if } n=4,\\
|z|^{-1} &\text{if } n=5
\end{cases} \quad \text{as } |z|\to 0.\]
Thus, elliptic estimates yield that $\|\wtmcd_3\|_{L^t} \lesssim \delta^{\frac32+\frac{3}{t}}$ for any $t>3$, $\|\wtmcd_4\|_{L^t} \lesssim \delta^{1 + \frac{4}{t}}$ for any $t > 2$, and $\|\wtmcd_5\|_{L^t} \lesssim \delta^{\frac12+\frac{5}{t}}$ for any $t\in(\frac52,5)$. Thus, we conclude for any $\tau\in (0,1)$,
\[\begin{medsize}
\displaystyle \|\mcs_{\delta,\xi}\|_{L^{\infty}(\Omega)}=O\(\left.\begin{cases}
\delta^{\frac52-\tau} &\text{if } n=3, 5\\
\delta^{3-\tau} &\text{if } n=4
\end{cases}\right\}
+\delta^{2+\frac{n-2}{2}} \left[d(\xi,\pa\Omega)^{-(n-2)}\bigg|\log\frac{d(\xi,\pa\Omega)}{\delta}\bigg| + d(\xi,\pa\Omega)^{-n}\right]\),
\end{medsize}\]
which completes the proof.
\end{proof}
\begin{rmk}
\

\medskip \noindent (1) To construct solutions to the Brezis-Nirenberg problem via a perturbative approach, additional information about $H^n_\lambda(x,y)$ might be necessary.
However, since the coefficient $\lambda$ is fixed in this paper, the $C^{1,\alpha}$ regularity suffices for our purpose.

\medskip \noindent (2) Define $\vph^n_\lambda(x) := H^n_\lambda(x,x)$ for $n=3,4,5$ and $\vph(x) := H(x,x)$ for $n\ge 3$.
Indeed, it is not difficult to realize that $\vph_\lambda^n\in C^{\infty}(\Omega)$ for $n=3,4,5$ and $\vph\in C^{\infty}(\Omega)$ for $n\ge 3$. When $d(x, \pa \Omega)$ is small enough, the following estimates hold:
\begin{equation}\label{vaor}
\begin{aligned}
&\left.\begin{cases}
\vph_\lambda^n(x) &\text{if } n=3,4,5\\
\vph(x) &\text{if } n\ge 3
\end{cases}\right\}
=\frac{1}{(2d(x,\pa\Omega))^{n-2}}[1+O(d(x,\pa\Omega))],\\
&\left.\begin{cases}
|\nabla\vph_\lambda^n (x)| &\text{if } n=3,4,5\\
|\nabla\vph(x)| &\text{if } n\ge 3
\end{cases}\right\}
= \frac{2(n-2)}{(2d(x,\pa\Omega))^{n-1}}[1+O(d(x,\pa\Omega))].
\end{aligned}
\end{equation}
We postpone their proofs to Appendix \ref{b}.
\end{rmk}

\begin{cor}\label{cor:PZ^0exp}
Suppose that $x,\, \xi\in\Omega$ and $\delta>0$ is small. For any $\tau \in (0,1)$, it holds that
\[PZ^0_{\delta,\xi}(x) = Z^0_{\delta,\xi}(x) - \frac{n-2}{2}a_n\delta^{\frac{n-2}{2}}H(x,\xi) + O\(\delta^{\frac{n+2}{2}}d(\xi,\pa\Omega)^{-n}\)\]
provided $n \ge 3$ and $PU_{\delta,\xi}$ is given by equation \eqref{pu1}, and
\begin{align*}
\begin{medsize}
\displaystyle PZ^0_{\delta,\xi}(x)
\end{medsize}
&\begin{medsize}
\displaystyle = Z^0_{\delta,\xi}(x) + \frac{n-2}{4}\lambda a_n \delta^{\frac{n-2}{2}} \left.\begin{cases}
- |x-\xi| &\text{if } n=3\\
- \log|x-\xi| &\text{if } n=4 \\
\frac{1}{|x-\xi|}-4\lambda|x-\xi| &\text{if } n=5
\end{cases}\right\}
-\frac{n-2}{2}a_n\delta^{\frac{n-2}{2}} H^n_\lambda(x,\xi)
\end{medsize} \\
&\begin{medsize}
\displaystyle \ +\delta\pa_{\delta}\left[\delta^{2-\frac{n-2}{2}}\mcd_n(\frac{x-\xi}{\delta})\right] + \left.\begin{cases}
O(\delta^{\frac52-\tau}) &\text{if } n=3, 5\\
O(\delta^{3-\tau}) &\text{if } n=4
\end{cases}\right\}
\end{medsize} \\
&\begin{medsize}
\displaystyle \ + O\(\delta^{\frac{n+2}{2}} \left[d(\xi,\pa\Omega)^{-(n-2)}\bigg|\log\frac{d(\xi,\pa\Omega)}{\delta}\bigg| + d(\xi,\pa\Omega)^{-n}\right]\)
\end{medsize}
\end{align*}
provided $n=3,4,5$ and $PU_{\delta,\xi}$ is given by equation \eqref{pu2}.
\end{cor}
\begin{proof}
We can argue as in the proof of Lemma \ref{lem2.1}. We omit the details.
\end{proof}

\begin{cor}\label{cor:PZ^kexp}
Suppose that $x,\, \xi\in\Omega$, $\delta>0$ is small, and $k=1,\ldots,n$. For any $\tau \in (0,1)$, it holds that
\[PZ^k_{\delta,\xi}(x) = Z^k_{\delta,\xi}(x) - a_n\delta^{\frac{n}{2}}\pa_{\xi^k}H(x,\xi) + O\(\delta^{\frac{n+2}{2}}d(\xi,\pa\Omega)^{-n}\)\]
provided $n \ge 3$ and $PU_{\delta,\xi}$ is given by equation \eqref{pu1}, and
\begin{align*}
\begin{medsize}
\displaystyle PZ^k_{\delta,\xi}(x)
\end{medsize}
&\begin{medsize}
\displaystyle = Z^k_{\delta,\xi}(x)+a_n\delta^{\frac{n}{2}}
\left.\begin{cases}
\displaystyle \frac{\lambda}{2} \frac{(x-\xi)^k}{|x-\xi|} &\text{if } n=3\\
\displaystyle \frac{\lambda}{2} \frac{(x-\xi)^k}{|x-\xi|^2} &\text{if } n=4
\end{cases}\right\}
-\delta^{\frac{n}{2}} a_n\pa_{\xi^k}H^n_\lambda(x,\xi) + \delta\pa_{\xi^k}\left[\delta^{2-\frac{n-2}{2}}\mcd_n\(\frac{x-\xi}{\delta}\)\right]
\end{medsize} \\
&\begin{medsize}
\ + \left.\begin{cases}
\displaystyle O(\delta^{\frac52-\tau}) + O\(\delta^{\frac{n+2}{2}} \left[d(\xi,\pa\Omega)^{-(n-2)}\bigg|\log\frac{d(\xi,\pa\Omega)}{\delta}\bigg| + d(\xi,\pa\Omega)^{-n}\right]\) &\text{if } n=3\\
\displaystyle O(\delta^{3-\tau}) + O\(\delta^{\frac{n+2}{2}} \left[d(\xi,\pa\Omega)^{-(n-2)}\bigg|\log\frac{d(\xi,\pa\Omega)}{\delta}\bigg| + d(\xi,\pa\Omega)^{-n}\right]\) &\text{if } n=4
\end{cases}\right\}
\end{medsize}
\end{align*}
provided $n=3,4$ and $PU_{\delta,\xi}$ is given by equation \eqref{pu2}. Furthermore, if $n=5$ and $PU_{\delta,\xi}$ is given by equation \eqref{pu2}, then
\begin{align*}
PZ^k_{\delta,\xi}(x) &= Z^k_{\delta,\xi}(x) + a_n\delta^{\frac{n}{2}} \left[\frac{\lambda}{2} \frac{(x-\xi)^k}{|x-\xi|^3} +2\lambda^2 \frac{(x-\xi)^k}{|x-\xi|}\right]
-\delta^{\frac{n}{2}} a_n\pa_{\xi^k}H^n_\lambda(x,\xi) \\
&\ + \delta\pa_{\xi^k}\left[\delta^{2-\frac{n-2}{2}}\mcd_n\(\frac{x-\xi}{\delta}\)\right] + \delta\pa_{\xi^k}\mcs_{\delta,\xi}(x),
\end{align*}
where the function $\mcs_{\delta,\xi}$ satisfies
\[\|\delta\pa_{\xi^k}\mcs_{\delta,\xi}\|_{L^t(\Omega)} \lesssim \delta^{\frac12+\frac{5}{t}} + O\(\delta^{\frac{n+2}{2}}\left[d(\xi,\pa\Omega)^{-(n-2)}+\delta d(\xi,\pa\Omega)^{-(n+1)}\right]\)\]
for any $t\in(\frac53, \frac52)$.\footnote{We have not deduced a pointwise estimate of $|\delta \pa_{\xi^k}\mcs_{\delta,\xi}|$ for this case. Its $L^t$-estimate is sufficient for later use.}
\end{cor}
\begin{proof}
We notice that
\[|\nabla\mcd_n(z)| \simeq \begin{cases}
|\log|z|| &\text{if } n=3, \\
|z|^{-1} &\text{if } n=4, \\
|z|^{-2} &\text{if } n=5
\end{cases} \quad \text{as } |z|\to 0,
\quad \text{and} \quad
|\nabla\mcd_n(z)|\simeq |z|^{-(n-2)} \quad \text{as } |z|\to \infty.\]
Thus, elliptic estimates yield that $\|\delta\pa_{\xi^k}\wtmcd_3\|_{L^t} \lesssim \delta^{\frac32+\frac{3}{t}}$ for any $t>3$,
$\|\delta\pa_{\xi^k}\wtmcd_4\|_{L^t} \lesssim \delta^{1 + \frac{4}{t}}$ for any $t\in(2,4)$, and $\|\delta\pa_{\xi^k}\wtmcd_5\|_{L^t} \lesssim \delta^{\frac12+\frac{5}{t}}$ for any $t\in(\frac53, \frac52)$.
Using these results, we employ the same strategy in the proof of Lemma \ref{lem2.1}.
\end{proof}

\subsection{$L^{\frac{2n}{n+2}}(\Omega)$-norm estimates for the terms $\I_1$, $\I_2$, and $\I_3$}
We recall the quantities $\I_1$, $\I_2$, and $\I_3$ from \eqref{eq:I2}. We estimate their $L^{\frac{2n}{n+2}}(\Omega)$-norms.
\begin{lemma}\label{lem2.2}
For each $i \in \{1,\ldots,\nu\}$, we assume that $PU_i = PU_{\delta_i,\xi_i}$ satisfies \eqref{pu1} if $n \ge 5$ or \textup{[}$n = 3, 4$ and $u_0>0$\textup{]}, and satisfies \eqref{pu2} if $n=3,4$ and $u_0=0$. Then it holds that
\begin{multline*}
\|\I_1\|_{L^{\frac{p+1}{p}}(\Omega)}+\|\I_2\|_{L^{\frac{p+1}{p}}(\Omega)}+\|\I_3\|_{L^{\frac{p+1}{p}}(\Omega)} \lesssim \begin{medsize}
\left.\begin{cases}
0 &\text{if } n=3,u_0=0\\
\displaystyle \max_i\delta_i^2|\log\delta_i| &\text{if } n=4,u_0=0\\
\displaystyle \max_i\delta_i^{\frac{n-2}{2}} &\text{if } [n=3,4 \text{ and } u_0>0] \text{ or } n=5\\
\displaystyle \max_i\delta_i^2|\log\delta_i|^{\frac23} &\text{if } n=6\\
\displaystyle \max_i\delta_i^2, &\text{if } n\ge 7
\end{cases}\right\}
\end{medsize}\\
\begin{medsize}
+ \left.\begin{cases}
\displaystyle \max_i \ka_i^{n-2} & \text{if } n=3,4,5\\
\displaystyle \max_i \ka_i^4|\log\ka_i|^{\frac23} &\text{if } n=6\\
\displaystyle \max_i \ka_i^{\frac{n+2}{2}} &\text{if } n\ge 7
\end{cases}\right\}
+ \left.\begin{cases}
Q &\text{if } n=3,4,5\\
Q|\log Q|^{\frac23} &\text{if } n=6\\
Q^{\frac{n+2}{2(n-2)}} &\text{if } n \ge 7
\end{cases}\right\}\mone_{\{\nu\ge 2\}}
\end{medsize}
\end{multline*}
provided $\ep_1>0$ is small. Here, $\ka_i := \frac{\delta_i}{d(\xi_i,\pa\Omega)}$.
\end{lemma}
\begin{proof}
We begin by introducing an elementary inequality: For fixed $m \in \N$, $s > 1$, and any $a_1,\ldots,\,a_m \ge 0$, it holds that
\[0 \le \(\sum_{i=1}^{m}a_i\)^{s}-\sum_{i=1}^ma_i^s \lesssim \sum_{i\ne j}\left[(a_i+a_j)^s-a_i^s-a_j^s\right] \lesssim \begin{cases}
\displaystyle \sum_{i\ne j} a_i^{s-1}a_j &\text{if } s>2,\\
\displaystyle \sum_{i\ne j} \min\{a_i^{s-1}a_j,a_ia_j^{s-1}\} &\text{if } s\le 2.
\end{cases}\]
From this, we derive that
\begin{equation}\label{ilow}
0 \le \I_1+\I_2\lesssim\sum_{i=1}^{\nu}(U_i^{p-1}+U_j)+\sum_{i\ne j}U_i^{p-1}U_j
\quad \text{for } n=3,4,5.
\end{equation}

We next consider the cases $n\ge 6$. Fixing any $i\in\{1,\ldots,\nu\}$, we decompose $\I_1$ into three parts:
\[\I_1= \I_{11} + \I_{12} + \I_{13},\]
where
\begin{align*}
I_{11} &:= (u_0+ PU_i)^{p}-u_0^{p}- (PU_i)^p, \\
I_{12} &:= (u_0+\sigma)^{p}-(u_0+ PU_i)^{p}-\bigg(\sum_{j\neq i} PU_j\bigg)^{p}, \\
I_{13} &:= (PU_i)^p+\bigg(\sum_{j\neq i} PU_j\bigg)^{p}-\sigma^{p}.
\end{align*}
Considering the relationship between $u_0$ and $U_i$ in different regions, i.e., $u_0 \lesssim U_i$ when $|x-\xi_i| \le \sqrt{\delta_i}$ and $u_0 \gtrsim U_i$ when $|x-\xi_i| \ge \sqrt{\delta_i}$, we obtain that
\[|\I_{11}| \lesssim \min\{u_0(PU_i)^{p-1}, u_0^{p-1}PU_i\} \lesssim U_i^{p-1}\mone_{|x-\xi_i| \le \sqrt{\delta_i}}+U_i\mone_{|x-\xi_i| \ge \sqrt{\delta_i}}.\]
Similarly, we have
\begin{align*}
|\I_{12}| &\lesssim \sum_{j\neq i}\min\left\{(u_0+ PU_i)^{p-1} PU_j, (u_0+ PU_i) (PU_j)^{p-1}\right\}\\
&\lesssim \sum_{j\neq i} \left[\min\{U_i^{p-1} U_j, U_j^{p-1} U_i\}\mone_{|x-\xi_i| \le \sqrt{\delta_i}}+\min\{U_j, U_j^{p-1}\}\mone_{|x-\xi_i| \ge \sqrt{\delta_i}}\right].
\end{align*}
In addition,
\[|\I_{13}|+\I_2\lesssim \sum_{j\neq i}\min\{U_i^{p-1} U_j, U_j^{p-1} U_i\}.\]
By introducing the rescaled variable $x_i:=\delta_i^{-1}(x-\xi_i)$, we further obtain a pointwise estimate for $\I_1+\I_2$ with the aid of \cite[Proposition 3.4]{DSW}:
\begin{equation}\label{i1po}
\begin{aligned}
&\ \I_1+\I_2 \\
&\lesssim \sum_{i=1}^{\nu}\min\{U_i, U_i^{p-1}\}\mone_{\{u_0>0\}}+\sum_{j\neq i}\min\{ U_i^{p-1} U_j, U_j^{p-1} U_i\}\mone_{\{\nu\ge 2\}} \\
&\lesssim \sum_{i=1}^{\nu}\left[\frac{\delta_i^{-2}}{\la x_i\ra^4} \mone_{\{|x_i|\le {\delta_i}^{-\frac12}\}}+\frac{\delta_i^{-\frac{n-2}{2}}}{\la x_i\ra^{n-2}} \mone_{\{|x_i|\ge {\delta_i}^{-\frac12}\}}\right] \mone_{\{u_0>0\}} \\
&\ + \left.\begin{cases}
\displaystyle \sum_{i=1}^{\nu} \left[\delta_i^{-4} \frac{\msr^{-4}}{\la x_i \ra^4} \mone_{\{|x_i| < \msr^2\}}(x)+ \delta_i^{-4} \frac{\msr^{-2}}{|x_i|^{5}} \mone_{\{|x_i| \ge \msr^2\}}(x)\right] &\text{if } n=6\\
\displaystyle \sum_{i=1}^{\nu} \left[\delta_i^{-\frac{n+2}{2}} \frac{\msr^{2-n}}{\la x_i \ra^4} \mone_{\{|x_i| < \msr\}}(x) + \delta_i^{-\frac{n+2}{2}} \frac{\msr^{-4}}{|x_i|^{n-2}} \mone_{\{|x_i| \ge \msr\}}(x)\right] &\text{if } n\ge7
\end{cases}\right\}\mone_{\{\nu\ge 2\}}.
\end{aligned}
\end{equation}

Employing \eqref{ilow} and \eqref{i1po}, we perform direct computations to find
\begin{multline}\label{i12e}
\|\I_1\|_{L^{\frac{p+1}{p}}(\Omega)}+\|\I_2\|_{L^{\frac{p+1}{p}}(\Omega)} \\
\lesssim \left.\begin{cases}
\displaystyle \max_i\delta_i^{\frac{n-2}{2}} &\text{if } n=3,4,5\\
\displaystyle \max_i\delta_i^2|\log\delta_i|^{\frac23} &\text{if } n=6\\
\displaystyle \max_i\delta_i^{\frac{n+2}{4}} &\text{if } n\ge 7
\end{cases}\right\}\mone_{\{u_0>0\}}
+ \left.\begin{cases}
Q &\text{if } n=3,4,5\\
Q|\log Q|^{\frac23} &\text{if } n=6\\
Q^{\frac{n+2}{2(n-2)}} &\text{if } n\ge 7
\end{cases}\right\}\mone_{\{\nu\ge2\}}.
\end{multline}

Now, we analyze the term $\I_3$. By applying Lemma \ref{lem2.1}, we observe that
\begin{align}
&\ \|(PU_i-U_i)U_i^{p-1}\|_{L^{\frac{p+1}{p}}(\Omega)} \nonumber\\
&\lesssim \|(PU_i-U_i)U_i^{p-1}\|_{L^{\frac{p+1}{p}}\(B(\xi_i, \frac{d(\xi_i,\pa\Omega)}{\delta_i})\)}+\|U_i^p\|_{L^{\frac{p+1}{p}}\(B(\xi_i, \frac{d(\xi_i,\pa\Omega)}{\delta_i})^c\)} \label{pup1}\\
&\lesssim \begin{cases}
\displaystyle \max_i \ka_i^{n-2} &\text{if } n=3 \text{ or } [n=4, \text{ each } PU_i \text{ satisfies \eqref{pu1}}] \text{ or } n= 5,\\
\displaystyle \max_i \(\delta_i^2|\log\delta_i| + \ka_i^2\) &\text{if } n=4 \text{ and each } PU_i\text{ satisfies \eqref{pu2}},\\
\displaystyle \max_i \ka_i^4|\log\ka_i|^{\frac23} &\text{if } n=6,\\
\displaystyle \max_i \ka_i^{\frac{n+2}{2}} &\text{if } n\ge 7
\end{cases} \nonumber\\
&=: J_1 \nonumber
\end{align}
where $\ka_i = \frac{\delta_i}{d(\xi_i,\pa\Omega)}$. Using estimate \eqref{iqu} and Lemma~\ref{a4}, we obtain
\begin{equation}\label{i3e}
\begin{aligned}
\|\I_3\|_{L^{\frac{p+1}{p}}(\Omega)}
&\lesssim \max_i\|(PU_i-U_i)U_i^{p-1}\|_{L^{\frac{p+1}{p}}(\Omega)} + \max_i\|(PU_i-U_i)^{p-1}U_i\|_{L^{\frac{p+1}{p}}(\Omega)} \\
&\ +\lambda\max_i\|U_i\|_{L^{\frac{p+1}{p}}(\Omega)}\mone_{\{\text{each } PU_i \text{ satisfies \eqref{pu1}}\}} \\
&\lesssim J_1
+\left.\begin{cases}
\displaystyle \max_i\delta_i^{\frac{n-2}{2}} &\text{if } n=3,4,5\\
\displaystyle \max_i\delta_i^2|\log\delta_i|^{\frac23} &\text{if } n=6\\
\displaystyle \max_i\delta_i^2 &\text{if } n\ge 7
\end{cases}\right\} \mone_{\{\text{each } PU_i \text{ satisfies \eqref{pu1}}\}}.
\end{aligned}
\end{equation}

By collecting estimates \eqref{i12e} and \eqref{i3e}, we conclude the proof.
\end{proof}

\subsection{Projections of $\I_1$, $\I_2$, and $\I_3$ onto the $PZ_j^0$-direction}
Given $j = 1,\ldots, \nu$, we evaluate the integrals $\int_{\Omega} \I_1 PZ_j^0$, $\int_{\Omega} \I_2 PZ_j^0$, and $\int_{\Omega} \I_3 PZ_j^0$,
which correspond to the projections of $\I_1$, $\I_2$, and $\I_3$ onto the directions of $PZ_j^0$, respectively.
\begin{lemma}\label{lem2.31}
Assume that $u_0>0$. Moreover, when $n=3$, each $PU_i$ satisfies \eqref{pu1} or \eqref{pu2}, and when $n\ge 4$, each $PU_i$ satisfies \eqref{pu1}.
For any $j \in \{1,\ldots,\nu\}$, it holds that
\begin{equation}\label{eq:I2Z}
\begin{aligned}
\int_{\Omega} \I_1 PZ_j^0 &= \mfa_nu_0(\xi_j)\delta_j^{\frac{n-2}{2}}+o(Q)+ \left.\begin{cases}
\displaystyle O(\max_i\delta_i) &\text{if } n=3 \\
\displaystyle O(\max_i\delta_i^2|\log\delta_i|) &\text{if } n=4 \\
\displaystyle O(\max_i\delta_i^2) &\text{if } n=5
\end{cases}
\right\} \mone_{\{p>2\}}\\
&\ + O\(\max_i\delta_i^{\frac{n}{2}} + \max_i\ka_i^n\),
\end{aligned}
\end{equation}
where $\ka_i = \frac{\delta_i}{d(\xi_i,\pa\Omega)}$ and $\mfa_n := p\int_{\R^n}U^{p-1}Z^0 > 0$.
\end{lemma}
\begin{proof}
By \eqref{ab6}, there exists a constant $\eta>0$ such that
\begin{equation}\label{i2ex}
\begin{aligned}
\I_1 &= \left[pu_0\sigma^{p-1}+O(u_0^2\sigma^{p-2})\mone_{\{p>2\}}+ O\(u_0^{p}\)\right] \mone_{\cup_{i=1}^{\nu} B(\xi_i,\eta\sqrt{\delta_i})} \\
&\ +\left[pu_0^{p-1}\sigma +O(u_0^{p-2}\sigma^2)\mone_{\{p>2\}} + O(\sigma^{p})\right] \mone_{\cap_{i=1}^{\nu} B(\xi_i,\eta\sqrt{\delta_i})^c}.
\end{aligned}
\end{equation}
The remainder of the proof is split into two steps.

\medskip \noindent \doublebox{\textsc{Step 1.}}
It follows from $|PZ_j^0| \lesssim U_j$, Lemma \ref{lem2.1} and Corollary \ref{cor:PZ^0exp}, and Young's inequality that 
\begin{align*}
&\ \bigg|\int_{B(\xi_j, d(\xi_j,\pa\Omega))} \left[(PU_j)^{p-1} PZ_j^0 -U_j^{p-1}Z_j^0\right]\bigg| \\
&\lesssim \int_{B(\xi_j, d(\xi_j,\pa\Omega))} (|PU_j-U_j|+|PZ_j^0-Z_j^0|) U_j^{p-1} +\int_{B(\xi_j, d(\xi_j,\pa\Omega))} |PU_j-U_j|^{p-1}U_j \\
&\lesssim \delta_j^{\frac{n-2}{2}}\ka_j^2 \lesssim \delta_j^{\frac{n}{2}}+\ka_j^n.
\end{align*}
Therefore,
\begin{equation}\label{21}
\begin{aligned}
&\ p \int_{\cup_{i=1}^{\nu} B(\xi_i,\eta\sqrt{\delta_i})} u_0 (PU_j)^{p-1} PZ_j^0 \\
&= p\int_{B(\xi_j, d(\xi_j,\pa\Omega))} u_0 (PU_j)^{p-1} PZ_j^0 + O\bigg(\int_{B(\xi_j,d(\xi_j,\pa\Omega))^c} U_j^p\bigg) \\
&= \displaystyle p \delta_j^{\frac{n-2}{2}} \left[u_0(\xi_j) \int_{\R^n} U^{p-1}Z^0 + O\bigg(\int_{B(0,\ka_j)} |\delta_j y|^2 U^{p}(y)dy\bigg)\right] + O\(\delta_j^{\frac{n}{2}}+\ka_j^n\) \\
&= \mfa_n \delta_j^{\frac{n-2}{2}} u_0(\xi_j) +O\(\delta_j^{\frac{n}{2}}+\ka_j^n\).
\end{aligned}
\end{equation}

We claim that
\begin{equation}\label{22}
\left|\int_{\Omega} u_0 \left[\sigma^{p-1}- (PU_j)^{p-1}\right] PZ_j^0\right|
\lesssim \sum_{i \ne l} \int_{\Omega} U_i^{p-1} U_l = o(Q) + O\(\max_i\delta_i^{\frac{n}{2}}\).
\end{equation}
The inequality immediately follows from \eqref{ab1}. To analyze the equality, we set $z_{ij}:=\delta_i^{-1}(\xi_j-\xi_i)$ and $d_{ij}:=|\xi_i-\xi_j|$.
We distinguish three cases based on the value of $\msr_{ij}$.

\medskip \noindent \fbox{Case 1:} Suppose that $\msr_{ij} = \frac{d_{ij}}{\sqrt{\delta_i \delta_j}}$.
Then, it holds that $d_{ij} \ge \delta_i$ and $(\sqrt{\delta_i\delta_j}/d_{ij})^{n-2} \simeq q_{ij} \le Q$. In view of Lemma \ref{a1}, we confirm that
\[\int_{\Omega} U_i^{p-1} U_j \lesssim
\left.\begin{cases}
\displaystyle \delta_i \delta_j^{\frac{1}{2}} d_{ij}^{-1} &\text{if } n=3 \\[0.2em]
\displaystyle \delta_i^2 \delta_j d_{ij}^{-2} \log\(2+d_{ij}\delta_i^{-1}\) &\text{if } n=4 \\[0.2em]
\displaystyle \delta_i^2 \delta_j^{\frac{n-2}{2}} d_{ij}^{-2} &\text{if } n\ge 5
\end{cases}\right\}
=O\(\max_i\delta_i^{\frac{n}{2}}\)+ o(Q).\]
\noindent \fbox{Case 2:} Suppose that $\msr_{ij} = \sqrt{\frac{\delta_i}{\delta_j}}$. Then, it holds that $d_{ij}\le \delta_i$, i.e., $|z_{ij}| \le 1$, and $({\frac{\delta_j}{\delta_i}})^{\frac{n-2}{2}} \simeq q_{ij} \le Q$. Therefore,
\begin{align*}
\int_{\Omega} U_i^{p-1} U_j &\lesssim \int_{\Omega} \(\frac{\delta_i}{\delta_i^2+|x-\xi_i|^2}\)^2 \(\frac{\delta_j}{\delta_j^2+|x-\xi_j|^2}\)^{\frac{n-2}{2}} dx \\
&\lesssim \delta_j^{\frac{n-2}{2}} \int_{B(0,C\delta_i^{-1})} \frac{1}{(1+|y|^2)^2} \frac{dy}{[(\frac{\delta_j}{\delta_i})^2+|y-z_{ij}|^2]^{\frac{n-2}{2}}} \\
&\lesssim \delta_j^{\frac{n-2}{2}} \(1 +\int_2^{C\delta_i^{-1}} t^{-3}dt\) \simeq \delta_j^{\frac{n-2}{2}} = o(Q).
\end{align*}
\noindent \fbox{Case 3:} Suppose that $\msr_{ij}=\sqrt{\frac{\delta_j}{\delta_i}}$. Then, it holds that $d_{ij}\le \delta_j$ and $({\frac{\delta_i}{\delta_j}})^{\frac{n-2}{2}} \simeq q_{ij} \le Q$. We divide the domain $\Omega$ into $B(\xi_i,\sqrt{\delta_i})$ and $(B(\xi_i,\sqrt{\delta_i}))^c$, and compute
\begin{align*}
\int_{B(\xi_i, \sqrt{\delta_i})} U_i^{p-1} U_j &\lesssim \frac{\delta_i^{n-2}}{\delta_j^{\frac{n-2}{2}}} \int_{B(0,\delta_i^{-\frac12})} \frac{1}{(1+|y|^2)^2} \frac{dy}{[1+(\frac{\delta_i}{\delta_j}|y-z_{ij}|)^2]^{\frac{n-2}{2}}} \\
&\lesssim \frac{\delta_i^{n-2}}{\delta_j^{\frac{n-2}{2}}} \(1+\int_1^{\delta_i^{-\frac12}}t^{n-5}dt\) = o(Q)
\end{align*}
and
\[\int_{B(\xi_i,\sqrt{\delta_i})^c} U_i^{p-1} U_j \lesssim \delta_i^2\delta_j^{\frac{n-2}{2}} \int_{B(0,\sqrt{\delta_i})^c}\frac{1}{|y|^4} \frac{1}{|y-(\xi_j-\xi_i)|^{n-2}} dy \lesssim \delta_i\delta_j^{\frac{n-2}{2}} = O\(\max_i\delta_i^{\frac{n}{2}}\).\]
These estimates justify \eqref{22}.

\medskip \noindent \doublebox{\textsc{Step 2.}} Applying $|PZ_j^0|\lesssim \sum_{i=1}^{\nu}U_i$, we observe
\begin{equation}\label{24}
\int_{\Omega} u_0^2 \sigma^{p-2} \left|PZ_j^0\right| \mone_{\{p>2\}} \lesssim \sum_{i=1}^{\nu} \int_{\Omega} U_i^{p-1} \mone_{\{p>2\}} \lesssim\begin{cases}
\displaystyle \max_i\delta_i &\text{if } n=3,\\
\displaystyle \max_i\delta_i^2|\log\delta_i| &\text{if } n=4,\\
\displaystyle \max_i\delta_i^2 &\text{if } n=5.
\end{cases}
\end{equation}
Furthermore, since $u_0(x)\lesssim U_i(x)$ for $x\in B(\xi_i,\eta\sqrt{\delta_i})$, we infer from \eqref{22} that
\begin{equation}\label{25}
\begin{aligned}
\int_{\cup_{i=1}^{\nu} B(\xi_i,\eta\sqrt{\delta_i})} u_0^{p} |PZ_j^0|
&\lesssim \int_{B(\xi_j,\eta\sqrt{\delta_j})} U_j +\sum_{i\neq j}\int_{B(\xi_i,\eta\sqrt{\delta_i})} U_i^{p-1} U_j \\
&=O\(\max_i\delta_i^{\frac{n}{2}}\)+ o(Q).
\end{aligned}
\end{equation}
We also estimate the integrals over the exterior region:
\begin{equation}\label{28}
\begin{aligned}
\int_{\cap_{i=1}^{\nu}B(\xi_i,\eta\sqrt{\delta_i})^c} u_0^{p-2} \sigma^2 |PZ_j^0| \mone_{\{p>2\}} &\lesssim \sum_{i=1}^{\nu}\int_{B(\xi_i,\eta\sqrt{\delta_i})^c} U_i^3 \mone_{\{p>2\}} \\
&\lesssim \begin{cases}
\displaystyle \max_i\delta_i^{\frac32}|\log\delta_i| &\text{if } n=3,\\
\displaystyle \max_i\delta_i^{\frac{n}{2}} &\text{if } n=4,5
\end{cases}
\end{aligned}
\end{equation}
and
\begin{equation}\label{29}
\int_{\cap_{i=1}^{\nu}B(\xi_i,\eta\sqrt{\delta_i})^c} (u_0^{p-1} \sigma +\sigma^{p}) |PZ_j^0|
\lesssim \sum_{i=1}^{\nu} \int_{B(\xi_i,\eta\sqrt{\delta_i})^c} (U_i^2+U_i^p) \lesssim \max_i\delta_i^{\frac{n}{2}}.
\end{equation}

\medskip
Combining estimates \eqref{21}--\eqref{29}, we conclude the proof of \eqref{eq:I2Z}.
\end{proof}

\begin{lemma}\label{lem2.32}
For any $j \in \{1,\ldots,\nu\}$, it holds that
\begin{align}
\int_{\Omega} \I_3 PZ_j^0 &= \left.\begin{cases}
\displaystyle -\delta_j\mfc_n \vph(\xi_j) + O(\ka_j^3) + O(\delta_j) &\text{if } n=3 \\
\displaystyle -\delta_j^2\mfc_n \vph(\xi_j) + O(\ka_j^4) + O(\delta_j^2|\log\delta_j|) &\text{if } n=4 \\
\displaystyle \lambda\mfb_n \delta_j^2 - \delta_j^{n-2}\mfc_n \vph(\xi_j) + O\(\delta_j^2 \ka_j^{n-4}\) + O(\ka_j^n) &\text{if } n\ge 5
\end{cases}\right\} \label{i3pz1} \\
&\ + \left[\left.\begin{cases}
\displaystyle O(\max_i\delta_i) &\text{if } n=3 \\
\displaystyle O(\max_i \delta_i^2|\log\delta_i|) &\text{if } n=4 \\
\displaystyle o(\max_i\delta_i^2) &\text{if } n\ge 5
\end{cases}\right\}
+o(Q)\right] \mone_{\{\nu\ge 2, \textup{ each } \xi_i \textup{ is in a compact set of } \Omega\}} \nonumber
\end{align}
provided $n \ge 3$ and each $PU_i$ satisfies \eqref{pu1}, and
\begin{align}
\int_{\Omega} \I_3 PZ_j^0 &= \left.\begin{cases}
\displaystyle - \mfc_3\vph^3_{\lambda}(\xi_j)\delta_j + O(\delta_j^2) + O(\ka_j^3) &\text{if } n=3 \\
\displaystyle \mfb_4\lambda\delta_j^2|\log\delta_j| - \mfc_4\delta_j^2\vph^4_{\lambda}(\xi_j) - 96|\S^3|\lambda\delta_j^2 + O(\delta_j^3) + O(\ka_j^4) &\text{if } n=4
\end{cases}\right\} \label{i3pz2} \\
&\ + \left[\left.\begin{cases}
\displaystyle O(\max_i\delta_i^2) &\text{if } n=3\\
\displaystyle O(\max_i\delta_i^3|\log\delta_i|) &\text{if } n=4
\end{cases}\right\}
+o(Q)\right] \mone_{\{\nu\ge 2, \textup{ each } \xi_i \textup{ is in a compact set of } \Omega\}} \nonumber
\end{align}
provided $n=3,4$ and each $PU_i$ satisfies \eqref{pu2}.
Here, $\ka_j = \frac{\delta_j}{d(\xi_j,\pa\Omega)}$, $\mfb_4:=3\sqrt{2}\int_{\R^4}U^{p-1}Z^0 > 0$, $\mfb_n:=\int_{\R^n}U Z^0 > 0$ for $n\ge 5$, and $\mfc_n:=a_np\int_{\R^n}U^{p-1}Z^0 > 0$.
\end{lemma}
\begin{proof}
We present the proof by dividing it into two steps.

\medskip \noindent \doublebox{\textsc{Step 1.}} Assuming that each $PU_i$ satisfies \eqref{pu2}, we assert that
\begin{equation}\label{lmpz}
\int_{\Omega} \sum_{i=1}^{\nu}\lambda PU_i PZ_j^0
= \begin{cases}
\displaystyle O(\max_i\delta_i)+o(Q)\mone_{\{\nu\ge 2\}} &\text{if } n=3, \\
\displaystyle O(\max_i \delta_i^2|\log\delta_i|)+o(Q)\mone_{\{\nu\ge 2\}} &\text{if } n=4, \\
\displaystyle \lambda\mfb_n \delta_j^2 + O\(\delta_j^2\ka_j^{n-4}\) + o(Q+\max_i\delta_i^2)\mone_{\{\nu\ge 2\}} &\text{if } n\ge 5.
\end{cases}
\end{equation}

To verify \eqref{lmpz}, we first estimate
\[\int_{\Omega}PU_jPZ_j^0 =\begin{cases}
\displaystyle O(\delta_j) &\text{if } n=3,\\
\displaystyle O(\delta_j^2|\log\delta_j|) &\text{if } n=4,\\
\displaystyle \mfb_n\delta_j^2 + O\(\delta_j^2\ka_j^{n-4}\) &\text{if } n \ge 5.
\end{cases}\]
Indeed, for the case $n \ge 5$, we have
\[\bigg|\int_{B(\xi_j,d(\xi_j,\pa\Omega))^c} PU_jPZ_j^0\bigg|\lesssim \delta_j^2\ka_j^{n-4}\]
and
\begin{equation}\label{puzl}
\begin{aligned}
\int_{B(\xi_j,d(\xi_j,\pa\Omega))}PU_jPZ_j^0
&= \int_{B(\xi_j,d(\xi_j,\pa\Omega))}U_jZ_j^0 + O\(\frac{\delta_j^{\frac{n-2}{2}}}{d(\xi_j,\pa\Omega)^{n-2}}\cdot \delta_j^{\frac{n+2}{2}} \int_{B(0,\ka_j)} U\) \\
&=\mfb_n\delta_j^2+O\(\delta_j^2\ka_j^{n-4}\).
\end{aligned}
\end{equation}

It remains to estimate the interaction terms $\int_{\Omega}U_iU_j$ for $1 \le i \ne j \le \nu$ provided $\nu\ge 2$. As in \eqref{22}, we separate the analysis into three cases.

\medskip \noindent \fbox{Case 1:} Suppose that $\msr_{ij} = \frac{d_{ij}}{\sqrt{\delta_i \delta_j}}$. We verify that
\begin{equation}\label{lauij}
\begin{aligned}
\int_{\Omega} U_i U_j &\lesssim \delta_i^{\frac{n-2}{2}} \delta_j^{\frac{n-2}{2}}
\times \begin{cases}
1 &\text{if } n=3,\\
1+|\log d_{ij}| &\text{if } n=4,\\
\displaystyle d_{ij}^{-(n-4)} &\text{if } n\ge 5
\end{cases} \\
&\simeq \begin{cases}
\displaystyle O(\max_i\delta_i) &\text{if } n=3,\\
\displaystyle O(\max_i\delta_i^2|\log\delta_i|) &\text{if } n=4,\\
\displaystyle \max_i\delta_i^2Q^{\frac{n-4}{n-2}} = o(\max_i\delta_i^2) &\text{if } n \ge 5.
\end{cases}
\end{aligned}
\end{equation}
\noindent \fbox{Case 2:} Suppose that $\msr_{ij}=\sqrt{\frac{\delta_i}{\delta_j}}$. We estimate
\begin{align*}%\label{lauij2}
\int_{\Omega} U_i U_j &\lesssim \int_{\Omega} \(\frac{\delta_i}{\delta_i^2+|x-\xi_i|^2}\)^{\frac{n-2}{2}} \(\frac{\delta_j}{\delta_j^2+|x-\xi_j|^2}\)^{\frac{n-2}{2}} dx \\
&\lesssim \delta_i^{\frac{n+2}{2}-(n-2)}\delta_j^{\frac{n-2}{2}} \int_{B(0,C\delta_i^{-1})} \frac{1}{(1+|y|^2)^{\frac{n-2}{2}}} \frac{dy}{[(\frac{\delta_j}{\delta_i})^2+|y-z_{ij}|^2]^{\frac{n-2}{2}}} \\
&\lesssim \delta_i^2\frac{\delta_j^{\frac{n-2}{2}}}{\delta_i^{\frac{n-2}{2}}} \(1 +\int_2^{C\delta_i^{-1}} t^{-(n-5)}dt\) = o(Q).
\end{align*}
\noindent \fbox{Case 3:} Suppose that $\msr_{ij}=\sqrt{\frac{\delta_j}{\delta_i}}$. We evaluate
\begin{align*}%\label{lauij3}
\int_{B(\xi_i, \sqrt{\delta_i})} U_i U_j &\lesssim \frac{\delta_i^{\frac{n+2}{2}}}{\delta_j^{\frac{n-2}{2}}} \int_{B(0,\frac{1}{\sqrt{\delta_i}})}\frac{1}{(1+|y|^2)^{\frac{n-2}{2}}}
\frac{dy}{[1+(\frac{\delta_i}{\delta_j}|y-z_{ij}|)^2]^{\frac{n-2}{2}}} \\
&\lesssim \frac{\delta_i^{\frac{n+2}{2}}}{\delta_j^{\frac{n-2}{2}}} \(1+\int_1^{\delta_i^{-\frac12}}tdt\) = o(Q)
\end{align*}
and
\begin{equation}\label{lauij4}
\begin{aligned}
\int_{B(\xi_i,\sqrt{\delta_i})^c} U_iU_j &\lesssim \delta_i^{\frac{n-2}{2}}\delta_j^{\frac{n-2}{2}} \int_{B(0,\sqrt{\delta_i})^c} \frac{1}{|y|^{n-2}}\frac{1}{|y-(\xi_j-\xi_i)|^{n-2}} dy \\
&=\begin{cases}
\displaystyle O(\max_i\delta_i) &\text{if } n=3 \text{ and } u_0>0, \\
\displaystyle O(\max_i\delta_i^2|\log\delta_i|) &\text{if } n=4 \text{ and } u_0>0, \\
\displaystyle O\(\max_i\delta_i^{\frac{n}{2}}\) &\text{if } n\ge 5.
\end{cases}
\end{aligned}
\end{equation}
This concludes the proof of \eqref{lmpz}.

\medskip \noindent \doublebox{\textsc{Step 2.}} We claim that
\begin{align*}
\int_{\Omega} \sum_{i=1}^{\nu}[(PU_i)^{p}-U_i^{p}] PZ_j^0 &= -\delta_j^{n-2}\mfc_n \vph(\xi_j) + O(\ka_j^n) \\
&\ + \left[O(\max_i\delta_i^{n-1})+o(Q)\right] \mone_{\{\nu\ge 2, \text{ each } \xi_i \text{ is in a compact set of } \Omega\}}
\end{align*}
provided $n \ge 3$ and $PU_i$ satisfies \eqref{pu1} for each $i = 1,\ldots,\nu$, and
\begin{align*}
&\ \int_{\Omega} \sum_{i=1}^{\nu}[(PU_i)^{p}-U_i^{p}] PZ_j^0 \\
&= \left.\begin{cases}
\displaystyle -\mfc_3\vph^3_{\lambda}(\xi_j)\delta_j + O(\delta_j^2) + O(\ka_j^3) &\text{if } n=3 \\
\displaystyle \mfb_4\lambda\delta_j^2|\log\delta_j| - \mfc_4\delta_j^2\vph^4_{\lambda}(\xi_j) - 96|\S^3|\lambda\delta_j^2
+ O(\delta_j^3) + O(\ka_j^4) &\text{if } n=4
\end{cases}\right\} \\
&\ +\left[\left.\begin{cases}
\displaystyle O(\max_i\delta_i^2) &\text{if } n=3 \\
\displaystyle O(\max_i\delta_i^3|\log\delta_i|) &\text{if } n=4
\end{cases}\right\}
+o(Q)\right] \mone_{\{\nu\ge 2, \text{ each } \xi_i \text{ is in a compact set of } \Omega\}}
\end{align*}
provided $n=3,4$ and $PU_i$ satisfies \eqref{pu2} for each $i = 1,\ldots,\nu$.

\medskip
To prove this, we decompose the domain by $\Omega=B(\xi_j, d(\xi_j,\pa\Omega))\cup [\Omega\setminus B(\xi_j, d(\xi_j,\pa\Omega))]$.

First, we observe that
\[\left|\int_{\Omega\setminus B(\xi_j, d(\xi_j,\pa\Omega))}[(PU_j)^{p}-U_j^{p}] PZ_j^0\right| \lesssim \int_{B(0,\ka_j)^c}U^{p+1} \lesssim \ka_j^n.\]

Suppose that $PU_i$ satisfies \eqref{pu1}. By Lemma \ref{lem2.1} and \eqref{ab6}, we obtain
\begin{equation}\label{puz5}
\begin{aligned}
&\ \int_{B(\xi_j, d(\xi_j,\pa\Omega))}[(PU_j)^{p}-U_j^{p}] PZ_j^0 \\
&= p\int_{B(\xi_j, d(\xi_j,\pa\Omega))} (PU_j-U_j)U_j^{p-1}PZ_j^0 + O\(\int_{B(\xi_j, d(\xi_j,\pa\Omega))} (PU_j-U_j)^2U_j^{p-2}|PZ_j^0|\) \mone_{\{p>2\}}\nonumber\\
&\ +O\(\int_{B(\xi_j, d(\xi_j,\pa\Omega))}|PU_j-U_j|^p|PZ_j^0|\) \\
&= -\delta_j^{n-2}\mfc_n \vph(\xi_j) + O(\ka_j^n).
\end{aligned}
\end{equation}

Suppose next that $n=3$ and $PU_i$ satisfies \eqref{pu2}. Noticing that
\begin{equation} \label{tdn}
\begin{aligned}
p\int_{B(\xi_j, d(\xi_j,\pa\Omega))}\wtmcd_3U_j^{p-1} Z_j^0 &= \delta_j^2\int_{B(0,\ka_j)} (-\Delta\mcd_3) Z^0 \\
&\ +\delta_j^2 O\(\int_{\pa B(0,\ka_j)} \frac{\pa\mcd_3}{\pa\nu}|Z^0| dS + \int_{\pa B(0,\ka_j)}\bigg|\frac{\pa Z^0}{\pa\nu} \mcd_3\bigg| dS\) \\
&= \delta_j^2 \int_{B(0,\ka_j)} (-\Delta\mcd_3)Z^0 + O\(\delta_j^3+\ka_j^3\),
\end{aligned}
\end{equation}
where $\frac{\pa}{\pa\nu}$ denotes the outward normal derivative and $dS$ is the surface measure, we deduce
\begin{align*}
&\ \int_{B(\xi_j, d(\xi_j,\pa\Omega))}[(PU_j)^{p}-U_j^{p}] PZ_j^0 \\
&= -\delta_j^{\frac{1}{2}}a_3pH^3_{\lambda}(\xi_j,\xi_j) \int_{B(\xi_j, d(\xi_j,\pa\Omega))} U_j^{p-1}Z_j^0-\frac{\lambda}{2}a_n\delta_j^{\frac12} \int_{B(\xi_j, d(\xi_j,\pa\Omega))} |x-\xi_j|(U_j^{p-1}Z_j^0)(x) dx \\
&\ +\lambda a_3^2p\delta_j^2\int_{B(0, \ka_j)} \left[\frac{1}{\sqrt{1+|z|^2}}-\frac{1}{|z|}\right]\frac{|z|^2-1}{(1+|z|^2)^{\frac32}} dz + O(\delta_j^{3-\tau})+O(\ka_j^3) \\
&=-\mfc_3\delta_j\vph^3_{\lambda}(\xi_j) + O(\delta_j^2) + O(\ka_j^3). \nonumber
\end{align*}

\medskip
Suppose that $n=4$ and $PU_i$ satisfies \eqref{pu2}. Then
\begin{align}
&\ \int_{B(\xi_j, d(\xi_j,\pa\Omega))}[(PU_j)^{p}-U_j^{p}] PZ_j^0 \nonumber\\
&=\frac{\lambda}{2} a_4\delta_j|\log\delta_j| p\int_{B(\xi_j,d(\xi_j,\pa\Omega))} U_j^{p-1}Z_j^0 - \frac{\lambda}{2} a_4p\delta^2_j\int_{B(0,\ka_j)} \log|x|(U^2Z^0)(x) dx \label{puz4} \\
&\ +\lambda a_4^2p\delta_j^2\int_{B(0,\ka_j)} \left[\frac{1}{1+|z|^2}-\frac{1}{|z|^2}\right]\frac{|z|^2-1}{(1+|z|^2)^2} dz
-\delta_j a_4p H^4_{\lambda}(\xi_j,\xi_j) \int_{B(\xi_j, d(\xi_j,\pa\Omega))}U_j^{p-1}Z_j^0 \nonumber \\
&\ + O(\delta_j^3) + O(\ka_j^4) \nonumber \\
&=\mfb_4\lambda\delta_j^2|\log\delta_j| - \mfc_4\delta_j^2\vph^4_{\lambda}(\xi_j) - 96|\S^3|\lambda\delta_j^2+O(\delta_j^3) + O(\ka_j^4). \nonumber
\end{align}
Here, we used
\[\int_{\R^4}\log|z|\frac{|z|^2-1}{(1+|z|^2)^4} dz = \frac{|\S^3|}8
\quad \text{and} \quad
\int_{\R^4}\left[\frac{1}{1+|z|^2}-\frac{1}{|z|^2}\right]\frac{|z|^2-1}{(1+|z|^2)^2} dz=0.\]

\medskip
Finally, we assume that $\nu\ge 2$ and each $\xi_1,\ldots,\xi_{\nu}$ is in a compact set of $\Omega$. Given $1 \le i \ne j \le \nu$, we infer from \eqref{22} that
\begin{equation}\label{pupz1}
\bigg|\int_{\Omega} [(PU_i)^{p}-U_i^{p}] PZ_j^0\bigg| \lesssim \delta_i^{\frac{n-2}{2}} \int_{B(\xi_i,d(\xi_i,\pa\Omega))}U_i^{p-1}U_j = O(\max_i\delta_i^{n-1})+o(Q)
\end{equation}
provided $n \ge 3$ and each $PU_i$ satisfies \eqref{pu1}, and
\begin{equation}\label{pupz2}
\begin{aligned}
\bigg|\int_{\Omega} [(PU_i)^{p}-U_i^{p}] PZ_j^0\bigg| &\lesssim  \left.\begin{cases}
\displaystyle O(\delta_i^{\frac{n-2}{2}}) &\text{if } n=3 \\
\displaystyle O(\delta_i|\log\delta_i|) &\text{if } n=4
\end{cases}\right\} \int_{B(\xi_i,d(\xi_i,\pa\Omega))}U_i^{p-1}U_j \\
&= \left.\begin{cases}
\displaystyle O(\max_i\delta_i^{n-1}) &\text{if } n=3 \\
\displaystyle O(\max_i\delta_i^3|\log\delta_i|) &\text{if } n=4
\end{cases}\right\}+o(Q)
\end{aligned}
\end{equation}
provided $n=3,4$ and each $PU_i$ satisfies \eqref{pu2}. Here we used 
\[\int_{\Omega}\left|\log \left|\frac{x-\xi_i}{\delta_i}\right|\right|(U_i^{p-1}U_j)(x) dx=o(Q+\max_i \delta_i^2)\] for $n=4$, which can be argued as \eqref{22}.

This completes the proof of the claim.
\end{proof}

\begin{lemma}\label{lem2.33}
Assume that $\nu \ge 2$ and each of the $\xi_1,\ldots,\xi_{\nu}$ lies on a compact set of $\Omega$.
For any $j \in \{1,\ldots,\nu\}$, it holds that
\begin{align*}
&\int_{\Omega} \I_2 PZ_j^0 = \sum_{i\ne j} \mfd_n \(q_{ij}^{-\frac{2}{n-2}}-2\frac{\delta_j}{\delta_i}\)q_{ij}^{\frac{n}{n-2}} +o(Q) \\ %\label{eq:le35}
&+\begin{cases}
\displaystyle O(\max_i\delta_i^{n-2}) &\text{if } n\ge3, \text{ each } PU_i\text{ satisfies \eqref{pu1}}, \\
\displaystyle  \sum_{i\ne j} [-\mfb_3\lambda|\xi_j-\xi_i|-\mfc_3H_{\lambda}^3(\xi_i, \xi_j)]\delta_i^{\frac12}\delta_j^{\frac12} + o(\max_i\delta_i) &\text{if } n=3, \text{ each } PU_i\text{ satisfies \eqref{pu2}}, \\
\displaystyle \sum_{i\ne j} [-\mfb_4\lambda\log |\xi_j-\xi_i| - \mfc_4H_{\lambda}^4(\xi_i, \xi_j)]\delta_i\delta_j  + o(\max_i \delta_i^2|\log\delta_i|) &\text{if } n=4, \text{ each } PU_i\text{ satisfies \eqref{pu2}},
\end{cases}
%&+\begin{medsize}\left.\begin{cases}O(\max_i\delta_i^{n-1})&\text{if } [n=3, \text{ each } PU_i \text{ satisfies \eqref{pu2}}] \text{ or } [n\ge 3, \text{ each } PU_i \text{ satisfies \eqref{pu1}}] \\
%O(\max_i\delta_i^3|\log\delta_i|)&\text{if } [n=4, \text{ each } PU_i \text{ satisfies \eqref{pu2}}]
%\end{cases}\right\}\end{medsize}\nonumber
\end{align*}
where $\mfd_n > 0, \mfb_3=\frac{1}{2}a_3p\int_{\R^3}U^{p-1}Z^0$, provided $q_{ij}$ in \eqref{rq} is small.
\end{lemma}
\begin{proof}
Adapting the proof of \cite[Lemma 2.1]{DSW}, and employing Lemma \ref{lem2.1}, \eqref{pupz1}--\eqref{pupz2}, and \cite[Lemma A.2]{DSW}, we discover
\begin{align}
&\ \int_{\Omega} \I_2 PZ_j^0 \mone_{\{\nu\ge 2\}} \nonumber \\
&= \sum_{i\ne j}p\int_{\Omega}(PU_j)^{p-1}PU_iPZ_j^0 \mone_{\{\nu\ge 2\}}+o(Q) \nonumber \\
&=\begin{medsize}
\displaystyle \sum_{i\ne j} \int_{\R^n}U_i^p \delta_j\frac{\pa {U_j}}{\pa \delta_j} + p\sum_{i\ne j}\int_{\Omega}(PU_j)^{p-1}(PU_i-U_i)PZ_j^0 + O\(\sum_{i\ne j}\int_{\Omega}(PU_j)^{p-1}PU_i|PZ_j^0-Z_j^0|\)+ o(Q)
\end{medsize} \nonumber\\
%&+\begin{cases}
%-\mfc_3\sum_{i\ne j}\vph_{\lambda}^3(\xi_i)\delta_i^{\frac12}\delta_j^{\frac12}+o(\max_i\delta_i^{n-2})&\text{if } n=3, \text{ each } PU_i \text{ satisfies \eqref{pu2}}\\
%\sum_{i\ne j}[\mfb_4\lambda\delta_i\delta_j|\log\delta_i|-\mfc_4\delta_i\delta_j\vph^4_{\lambda}(\xi_i)]+O(\max_i\delta_i^{n-2})&\text{if } n=4, \text{ each } PU_i \text{ satisfies \eqref{pu2}}\\
%O(\max_i\delta_i^{n-2})&\text{ otherwise}
%\end{cases}
%\nonumber\\
%&\ +\begin{medsize}\left.\begin{cases}O(\max_i\delta_i^{n-1})&\text{if } [n=3, \text{ each } PU_i \text{ satisfies \eqref{pu2}}] \text{ or } [n\ge 3, \text{ each } PU_i \text{ satisfies \eqref{pu1}}] \\
%O(\max_i\delta_i^3|\log\delta_i|) &\text{if } [n=4, \text{ each } PU_i \text{ satisfies \eqref{pu2}}]
%\end{cases}\right\}\end{medsize}\nonumber\\
&=\sum_{i\ne j} \mfd_n \(q_{ij}^{-\frac{2}{n-2}}-2\frac{\delta_j}{\delta_i}\) q_{ij}^{\frac{n}{n-2}} + o(Q) \label{i3pz0} \\
&\ +\begin{cases}
\displaystyle O(\max_i\delta_i^{n-2}) &\text{if } n\ge3, \text{ each } PU_i \text{ satisfies \eqref{pu1}}, \\
\displaystyle  \sum_{i\ne j} [-\mfb_3\lambda|\xi_j-\xi_i|-\mfc_3H_{\lambda}^3(\xi_i, \xi_j)]\delta_i^{\frac12}\delta_j^{\frac12} + o(\max_i\delta_i) &\text{if } n=3, \text{ each } PU_i \text{ satisfies \eqref{pu2}},\\
\displaystyle \sum_{i\ne j} [-\mfb_4\lambda\log |\xi_j-\xi_i| - \mfc_4H_{\lambda}^4(\xi_i, \xi_j)]\delta_i\delta_j   + o(\max_i \delta_i^2|\log\delta_i|) &\text{if } n=4, \text{ each } PU_i \text{ satisfies \eqref{pu2}}.
\end{cases}\nonumber
%\\
%&+\begin{medsize}
%\left.\begin{cases}
%O(\max_i\delta_i^{n-1}) &\text{if } [n=3, \text{ each } PU_i \text{ satisfies \eqref{pu2}}] \text{ or } [n\ge 3, \text{ each } PU_i \text{ satisfies \eqref{pu1}}] \\
%O(\max_i\delta_i^3|\log\delta_i|) &\text{if } [n=4 \text{ each } PU_i \text{ satisfies \eqref{pu2}}]
%\end{cases}\right\}
%\end{medsize}\nonumber
\end{align}
Here we used 
\begin{align*}
&\delta_i^{\frac12}\int_{\Omega}(|x-\xi_i|-|\xi_j-\xi_i|)U_j^{p}(x) dx=o(Q+\max_i \delta_i)\text{~ for~} n=3,\\
&\delta_i\int_{\Omega}|\log |x-\xi_i|-\log|\xi_j-\xi_i||U_j^{p}(x) dx=o(Q+\max_i \delta_i^2|\log\delta_i|)\text{~ for~} n=4.
\end{align*} This finishes the proof.
\end{proof}

\section{Linear theory and an improved estimate for $n=6$}\label{sec2.2}
In Section \ref{sec3}, we will derive an $H^1_0(\Omega)$-norm estimate for $\rho$ of the form
\[\|\rho\|_{H^1_0(\Omega)} \lesssim \|f\|_{(H^1_0(\Omega))^*}+ \|\I_1\|_{L^{\frac{p+1}{p}}(\Omega)}+\|\I_2\|_{L^{\frac{p+1}{p}}(\Omega)}+\|\I_3\|_{L^{\frac{p+1}{p}}(\Omega)}.\]
When $n=6$, this estimate is coarse and requires refinement. In the remainder of this section, we develop a suitable linear theory for $n=6$, which enables the derivation of a pointwise estimate and an improved $H^1_0(\Omega)$-norm bound for the main part of $\rho$. In what follows, we assume that each of the $\xi_1,\ldots,\xi_{\nu}$ lies on a compact set of $\Omega$ if $\nu\ge 2$.
\begin{defn}\label{norm}
For each $i \in \{1, \dots, \nu\}$, recall the rescaled variable $x_i := \delta_i^{-1}(x-\xi_i) \in \delta_i^{-1}(\Omega - \xi_i)$. We introduce the weighted norms
\[
\|h\|_{**} := \sup_{x \in \Omega} \frac{|h(x)|}{V(x)}, \qquad \|\rho\|_* := \sup_{x \in \Omega} \frac{|\rho(x)|}{W(x)},
\]
where the weights $V(x)$ and $W(x)$ are defined by
\begin{align*}
V(x) &:= \sum_{i=1}^\nu \(v_{1i}(x) + [v_{2i}^{\tin}(x) + v_{2i}^{\tout}(x)] \mone_{\{\nu \ge 2\}} + [v_{3i}^{\tin}(x) + v_{3i}^{\tout}(x)] \mone_{\{\nu = 1\}} \), \\
W(x) &:= \sum_{i=1}^\nu \(w_{1i}(x) + [w_{2i}^{\tin}(x) + w_{2i}^{\tout}(x)] \mone_{\{\nu \ge 2\}} + [w_{3i}^{\tin}(x) + w_{3i}^{\tout}(x)] \mone_{\{\nu = 1\}} \).
\end{align*}
The component functions are given explicitly as follows:
\begin{align*}%\label{vw1i}
v_{1i}(x) &:= \frac{\delta_i^{-2}}{\la x_i \ra^4}, & w_{1i}(x) &:= \frac{1}{\la x_i \ra^2}, \\
v_{2i}^{\tin}(x) &:= \delta_i^{-4} \frac{\msr^{-4}}{\la x_i \ra^4} \mone_{\{|x_i| < \msr^2\}}, &
w_{2i}^{\tin}(x) &:= \delta_i^{-2} \frac{\msr^{-4}}{\la x_i \ra^2} \mone_{\{|x_i| < \msr^2\}}, \\
v_{2i}^{\tout}(x) &:= \delta_i^{-4} \frac{\msr^{-2}}{|x_i|^5} \mone_{\{|x_i| \ge \msr^2\}}, &
w_{2i}^{\tout}(x) &:= \delta_i^{-2} \frac{\msr^{-2}}{|x_i|^3} \mone_{\{|x_i| \ge \msr^2\}}, \\
v_{3i}^{\tin}(x) &:= \frac{1}{[d(\xi_i,\pa\Omega)\la x_i \ra]^4} \mone_{\left\{|x_i| \le \frac{d(\xi_i,\pa\Omega)}{\delta_i}\right\}}, &
w_{3i}^{\tin}(x) &:= \frac{\delta_i^2}{d(\xi_i,\pa\Omega)^4 \la x_i \ra^2} \mone_{\left\{|x_i| \le \frac{d(\xi_i,\pa\Omega)}{\delta_i}\right\}}, \\
v_{3i}^{\tout}(x) &:= \frac{\delta_i^{-1}}{d(\xi_i,\pa\Omega)^3 |x_i|^5} \mone_{\left\{|x_i| \ge \frac{d(\xi_i,\pa\Omega)}{\delta_i}\right\}}, &
w_{3i}^{\tout}(x) &:= \frac{\delta_i}{d(\xi_i,\pa\Omega)^3 |x_i|^3} \mone_{\left\{|x_i| \ge \frac{d(\xi_i,\pa\Omega)}{\delta_i}\right\}}.
\end{align*}
\end{defn}

Consider the equation
\begin{equation}\label{eq:31}
\begin{cases}
\displaystyle (-\Delta-\lambda)\rho_0 - 2(u_0+\sigma)\rho_0 = \I_1+\I_2+\I_{31}+\I_0[\rho_0]+ \sum_{i=1}^{\nu} \sum_{k=0}^6 c_i^k(-\Delta-\lambda) PZ_i^k \quad \text{in } \Omega \subset \R^6,\\
\displaystyle\rho_0=0 \quad \text{on } \pa\Omega, \quad c_1^0, \ldots, c_{\nu}^6 \in \R,\\
\displaystyle \big\langle \rho_0, PZ_i^k \big\rangle_{H^1_0(\Omega)} = 0 \quad \text{for } i=1,\ldots,\nu \text{ and } k=0,\ldots,6\end{cases}
\end{equation}
where
\[\I_{31} := \sum_{i=1}^{\nu}\bigg[\lambda PU_i+\bigg(a_6\delta_i^2\vph(\xi_i)U_i\mone_{\left\{|x_i| \le \frac{d(\xi_i, \pa\Omega)}{\delta_i}\right\}} + (PU_i^p-U_i^p)\mone_{\left\{|x_i| \ge \frac{d(\xi_i, \pa\Omega)}{\delta_i}\right\}}\bigg)\mone_{\{\nu=1\}}\bigg].\]
In Propositions \ref{prop:31} and \ref{prop:33}, we will prove the existence of $\rho_0$, its pointwise estimate and the $H^1_0(\Omega)$-norm estimate.

\medskip

We start with a linear theory.
\begin{prop}\label{prop:31}
Given any $h\in (H^1_0(\Omega))^*$ with $\|h\|_{**}\le C$.
There exist a constant $C=C(\nu,\lambda, u_0, \Omega)>0$, $\rho_0 \in H^1_0(\Omega)$ and numbers $\{c_i^k\}_{\{i=1,\ldots,\nu,\, k=0,\ldots,6\}}$ such that
\begin{equation}\label{eq:lin}
\begin{cases}
\displaystyle(-\Delta-\lambda)\rho_0 - 2(u_0+\sigma)\rho_0= h + \sum_{i=1}^{\nu} \sum_{k=0}^6 c_i^k (-\Delta-\lambda) PZ_i^k \quad \text{in } \Omega \subset \R^6, \\
\displaystyle \rho_0=0 \quad\text{on } \pa\Omega,\quad c_1^0, \ldots, c_{\nu}^6 \in \R, \\
\displaystyle \big\langle \rho_0, PZ_i^k \big\rangle_{H^1_0(\Omega)} = 0 \quad \text{for } i=1,\ldots,\nu \text{ and } k=0,\ldots,6
\end{cases}
\end{equation}
satisfying
\begin{equation}\label{eq:311}
\|\rho_0\|_* \le C\|h\|_{**}
\end{equation}
provided $\ep_1>0$ is small.
\end{prop}
Subsequently, we utilize Proposition \ref{prop:31} along with the Banach Fixed-point theorem to derive the following existence result.
\begin{prop}\label{prop:33}
Assume that $\ep_1>0$ is small enough. Equation \eqref{eq:31} admits a unique solution $\rho_0 \in H^1_0(\Omega)$ such that
\begin{equation}\label{eq:33}
\|\rho_0\|_* \le C,
\end{equation}
where $C>0$ depends only on $\nu$, $\lambda$, $u_0$ and $\Omega$. Moreover,
\begin{equation}\label{eq:331}
\|\rho_0\|_{H^1_0(\Omega)} \le C
\left[\max_i\delta_i^2|\log \delta_i|^{\frac12} + \max_i\(\frac{\delta_i}{d(\xi_i,\pa\Omega)}\)^4 \left|\log\frac{d(\xi_i,\pa\Omega)}{\delta_i}\right|^{\frac12} \mone_{\{\nu=1\}} + Q|\log Q|^{\frac12}\right].
\end{equation}
\end{prop}

To establish Proposition \ref{prop:31}, we need two preliminary lemmas.
\begin{lemma}
For each $j\in\{1,\ldots,\nu\}$ and $k\in\{0,1,\ldots,n\}$, there exists a constant $C>0$ depending only on $\nu,\lambda, u_0$ and $\Omega$ such that
\begin{equation}\label{eq:cjb}
|c_j^k| \le C [o(\|\rho_0\|_{*})+\|h\|_{**}] \cdot \left[Q\mone_{\{\nu\ge 2\}} + \delta_j^2 + \ka_j^4\mone_{\{\nu=1\}}\right]
\end{equation}
provided $\ep_1>0$ is small, where $\ka_j = \frac{\delta_j}{d(\xi_j,\pa\Omega)}$.
\end{lemma}
\begin{proof}
For each $j\in \{1,\dots,\nu\}$, we assert that
\begin{equation}\label{lc0}
\left|\int_{\Omega}(\lambda+2(u_0+\sigma)-2PU_j)\rho_0 PZ_j^k\right| \lesssim \int_{\Omega}|\rho_0|U_j + \int_{\Omega} \sum_{i\ne j}U_iU_j |\rho_0| \mone_{\{\nu\ge 2\}} = o(Q+\delta_j^2).
\end{equation}

Note that
\[\int_{\Omega}|\rho_0|U_j \lesssim \|\rho_0\|_* \bigg[\bigg\|\sum_{i=1}^{\nu}(w_{2i}^{\tin} + w_{2i}^{\tout})\bigg\|_{L^3(\Omega)} \|U_j\|_{L^{\frac32}(\Omega)}
+ \int_{\Omega}(w_{3j}^{\tin}+w_{3j}^{\tout})U_j \mone_{\{\nu=1\}}+\sum_{i=1}^{\nu}\int_{\Omega}w_{1i}U_j\bigg].\]
By Young's inequality,
\begin{align}
\bigg\|\sum_{i=1}^{\nu}(w_{2i}^{\tin}+w_{2i}^{\tout})\bigg\|_{L^3(\Omega)}\|U_j\|_{L^{\frac32}(\Omega)} &\lesssim Q|\log Q|^{\frac13}\cdot \delta_j^2|\log\delta_j|^{\frac23} \nonumber\\
&\lesssim Q^2|\log Q|^{\frac23}+\delta_j^4|\log \delta_j|^{\frac43}=o(Q+\delta_j^2)\label{w2ju}
\end{align}
and
\begin{equation}\label{w3jr}
\int_{\Omega}(w_{3j}^{\tin}+w_{3j}^{\tout})U_j \mone_{\{\nu=1\}} \lesssim \delta_j^2\ka_j^4 |\log\ka_j| + \delta_j^2\ka_j^4 = o(\delta_j^2).
\end{equation}
Let us prove that
\begin{equation}\label{wuj}
\int_{\Omega} \sum_{i=1}^{\nu}w_{1i} U_j \mone_{\{u_0>0\}}=o(Q+\delta_j^2).
\end{equation}
If $i = j$, it holds that
\begin{equation}\label{lc1}
\int_{\Omega}w_{1j} U_j \lesssim \delta_j^4|\log\delta_j|.
\end{equation}
If $i \ne j$, we have
\begin{align}
\int_{\Omega}w_{1i} U_j &\lesssim \int_{\Omega} \frac{\delta_i^2}{\delta_i^2+|x-\xi_i|^2} \(\frac{\delta_j}{\delta_j^2+|x-\xi_j|^2}\)^2 \label{w1vj} \\
&\lesssim \begin{medsize}
\begin{cases}
\displaystyle \delta_i^2 \delta_j^2 (1+|\log |\xi_i-\xi_j||) &\text{if } \msr_{ij} = \frac{|\xi_i-\xi_j|}{\sqrt{\delta_i\delta_j}} \\
\displaystyle \delta_i^2\delta_j^2 \int_{B(0,\delta_i^{-1})} \frac{1}{1+|y|^2} \frac{dy}{[(\frac{\delta_j}{\delta_i})^2+|y-z_{ij}|^2]^2} &\text{if } \msr_{ij}=\sqrt{\frac{\delta_i}{\delta_j}} \\
\displaystyle \frac{\delta_i^6}{\delta_j^2} \int_{B(0,\delta_i^{-\frac12})}\frac{1}{1+|y|^2}
\frac{dy}{[1+(\frac{\delta_i}{\delta_j}|y-z_{ij}|)^2]^2}+\delta_i^2\delta_j^2\int_{ \sqrt{\delta_i}\le |x|\le C}\frac{1}{|x-\xi_i|^2}\frac{dx}{|x-\xi_j|^4} &\text{if } \msr_{ij}=\sqrt{\frac{\delta_j}{\delta_i}}
\end{cases}
\end{medsize} \nonumber \\
&\lesssim \begin{medsize}
\left.\begin{cases}
\displaystyle \delta_j^2 \delta_i^2|\log\delta_i| &\text{if } \msr_{ij} = \frac{|\xi_i-\xi_j|}{\sqrt{\delta_i\delta_j}} \\
\displaystyle \delta_i^2\delta_j^2 \(1 +\int_2^{\delta_i^{-1}} t^{-1}dt\) \lesssim\delta_j^2 \delta_i^2|\log\delta_i| &\text{if } \msr_{ij}=\sqrt{\frac{\delta_i}{\delta_j}} \\
\displaystyle \frac{\delta_i^6}{\delta_j^2} \(1+\int_1^{\delta_i^{-\frac12}}t^3dt\) + \delta_i^2\delta_j^2|\log\delta_i| \lesssim \delta_i^2Q + \delta_j^2\delta_i^2|\log\delta_i| &\text{if } \msr_{ij}=\sqrt{\frac{\delta_j}{\delta_i}}
\end{cases}\right\}
\end{medsize}
=o(Q+\delta_j^2). \nonumber
\end{align}
As a result, \eqref{wuj} is valid.

Next, let us verify that
\begin{align*}
\int_{\Omega} \sum_{i\ne j}U_iU_j|\rho_0| \mone_{\{\nu\ge 2\}} &\lesssim \bigg\|\sum_{i=1}^{\nu}(v_{2i}^{\tin}+v_{2i}^{\tout})\bigg\|_{L^{\frac32}(\Omega)} \bigg\|\sum_{i=1}^{\nu}(w_{2i}^{\tin}+w_{2i}^{\tout})\bigg\|_{L^3(\Omega)} \\
&\ + \sum_{i\ne j}\int_{\Omega}U_iU_jw_{1i} + \sum_{i\ne j}\int_{\Omega}U_iU_jw_{1j} + \sum_{\substack{i \ne j,\, j \ne l,\\ i \ne l}}U_iU_jw_{1l} \\
&\simeq Q^2|\log Q|+o(Q)=o(Q).
\end{align*}
Here, we used that
\begin{align*}
\int_{\Omega} \sum_{\substack{i \ne j,\, j \ne l,\\ i \ne l}} U_iU_jw_{1l} &\lesssim \int_{\Omega} \sum_{\substack{i \ne j,\, j \ne l,\\ i \ne l}} U_iU_jU_l \\
&\lesssim \sum_{\substack{i \ne j,\, j \ne l,\\ i \ne l}} \|U_iU_j\|_{L^{\frac32}(\Omega)}^{\frac12} \|U_iU_l\|_{L^{\frac32}(\Omega)}^{\frac12} \|U_jU_l\|_{L^{\frac32}(\Omega)}^{\frac12} \lesssim Q^{\frac32}|\log Q|.
\end{align*}
Arguing as in \eqref{w1vj}, one can verify that for $i\ne j$,
\begin{align*}
\int_{\Omega} U_iU_jw_{1i} &\lesssim \int_{\Omega} \frac{\delta_i^4}{(\delta_i^2+|x-\xi_i|^2)^3} \(\frac{\delta_j}{\delta_j^2+|x-\xi_j|^2}\)^2 dx \\
&\lesssim \begin{medsize}
\left.\begin{cases}
\displaystyle \frac{\delta_i^4\delta_j^2}{|\xi_i-\xi_j|^4} \log\(2+\frac{|\xi_i-\xi_j|}{\delta_i}\) &\text{if } \msr_{ij} = \frac{|\xi_i-\xi_j|}{\sqrt{\delta_i\delta_j}} \\
\displaystyle \delta_j^2 \int_{B(0,\delta_i^{-1})} \frac{1}{(1+|y|^2)^3} \frac{dy}{[(\frac{\delta_j}{\delta_i})^2+|y-z_{ij}|^2]^2} &\text{if } \msr_{ij}=\sqrt{\frac{\delta_i}{\delta_j}} \\
\displaystyle \frac{\delta_i^4}{\delta_j^2} \int_{B(0,\delta_i^{-1})}\frac{1}{(1+|y|^2)^3}
\frac{dy}{[1+(\frac{\delta_i}{\delta_j}|y-z_{ij}|)^2]^2} &\text{if } \msr_{ij}=\sqrt{\frac{\delta_j}{\delta_i}}
\end{cases}\right\}
\end{medsize} \\
&\lesssim \begin{medsize}
\left.\begin{cases}
\delta_i^2|\log\delta_i|\msr_{ij}^{-4} &\text{if } \msr_{ij} = \frac{|\xi_i-\xi_j|}{\sqrt{\delta_i\delta_j}} \\
\delta_i^2\msr_{ij}^{-4} &\text{if } \msr_{ij}=\sqrt{\frac{\delta_i}{\delta_j}} \\
\delta_i^2|\log\delta_i|\msr_{ij}^{-4} &\text{if } \msr_{ij}=\sqrt{\frac{\delta_j}{\delta_i}}
\end{cases}\right\}
\end{medsize}
=o(Q),
\end{align*}
and similarly,
\begin{align*}
\int_{\Omega}U_iU_jw_{1j} \lesssim \left.\begin{cases}
\delta_j^2|\log\delta_j|\msr_{ij}^{-4} &\text{if } \msr_{ij} = \frac{|\xi_i-\xi_j|}{\sqrt{\delta_i\delta_j}} \nonumber\\
\delta_j^2\msr_{ij}^{-4} &\text{if } \msr_{ij}=\sqrt{\frac{\delta_j}{\delta_i}}\\
\delta_j^2|\log\delta_j|\msr_{ij}^{-4} &\text{if } \msr_{ij}=\sqrt{\frac{\delta_i}{\delta_j}}
\end{cases}\right\}=o(Q).
\end{align*}
All the above estimates imply that \eqref{lc0} holds true.

\medskip
Furthermore, by \eqref{w2ju}-\eqref{lc1}, we have
\begin{align}
\bigg|\int_{\Omega}(PU_jPZ_j^k-U_jZ_j^k)\rho_0\bigg| &\lesssim \|\rho_0\|_*\bigg[\frac{\delta_j^2}{d(\xi_j,\pa\Omega)^4} \int_{\Omega}(w_{1j}+w_{3j}^{\tin}+w_{3j}^{\tout})U_j \mone_{\{\nu=1\}}\label{pupzj}\\
&+\delta_j^2\int_{\Omega}U_j\sum_{i=1}^{\nu}(w_{1i}+w_{2i}^{\tin}+w_{2i}^{\tout})\mone_{\{\nu\ge 2, \text{~each~}\xi_i \text{~is in a compact set of~}\partial\Omega\}}\bigg]\nonumber\\
&= o(\|\rho_0\|_*(Q+\delta_j^2+\ka_j^4\mone_{\{\nu=1\}})).\nonumber
\end{align}

\medskip
Finally, we claim that
\begin{equation}\label{hzj}
\bigg|\int_{\Omega}h PZ_j^k\bigg| \lesssim \int_{\Omega}V U_j \lesssim \|h\|_{**} \(Q+\delta_j^2+\ka_j^4\mone_{\{\nu=1\}}\).
\end{equation}
A direct computation gives
\[\int_{\Omega}v_{1j}U_j\simeq \delta_j^2, \quad \int_{\Omega}(v_{2j}^{\tin}+v_{2j}^{\tout})U_j \simeq Q, \quad \int_{\Omega}(v_{3j}^{\tin}+v_{3j}^{\tout})U_j \mone_{\{\nu=1\}}\simeq\ka_j^4.\]
Assume that $i\ne j$. From \eqref{lauij} and \eqref{lauij4}, we see that
\[\int_{\Omega}v_{1i}U_j =\int_{\Omega}U_iU_j \lesssim
\left.\begin{cases}
\delta_j^2 \frac{\delta_i^2}{|\xi_i-\xi_j|^2} &\text{if } \msr_{ij}=\frac{|\xi_i-\xi_j|}{\sqrt{\delta_i\delta_j}} \\
o(Q) &\text{if } \msr_{ij}=\sqrt{\frac{\delta_i}{\delta_j}} \\
o(Q)+\delta_i\delta_j^2 &\text{if } \msr_{ij}=\sqrt{\frac{\delta_j}{\delta_i}}
\end{cases}\right\}
\lesssim \delta_j^2+o(Q).\]
Similarly to \eqref{w1vj}, we obtain
\begin{align*}%\label{lc6}
\int_{\Omega}v^{\tin}_{2i} U_j &\lesssim \int_{\Omega} \delta_i^{-4} \frac{\msr^{-4}}{\la x_i \ra^4} \mone_{\{|x_i| < \msr^2\}}(x)\(\frac{\delta_j}{\delta_j^2+|x-\xi_j|^2}\)^2 dx \\
&\lesssim \left.\begin{cases}
\frac{\delta_j^2\msr^{-4}}{|\xi_i-\xi_j|^2} &\text{if } \msr_{ij}=\frac{|\xi_i-\xi_j|}{\sqrt{\delta_i\delta_j}} \\
\frac{\delta_j^2}{\delta_i^2}\msr^{-4}\(1+\int_{2}^{\msr^2}t^{-3} dt\) &\text{if } \msr_{ij}=\sqrt{\frac{\delta_i}{\delta_j}} \\
\frac{\delta_i^2}{\delta_j^2}\msr^{-4}\(1+\int_{2}^{\msr^2}t dt\) &\text{if } \msr_{ij}=\sqrt{\frac{\delta_j}{\delta_i}}
\end{cases}\right\}
\lesssim \msr^{-4}\simeq Q,
\end{align*}
and
\begin{align*}%\label{lc7}
\int_{\Omega}v^{\tout}_{2i} U_j &\lesssim \int_{\Omega} \delta_i^{-4} \frac{\msr^{-2}}{|x_i|^5} \mone_{\{|x_i| \ge \msr^2\}}(x)\(\frac{\delta_j}{\delta_j^2+|x-\xi_j|^2}\)^2 dx \\
&\lesssim \left.\begin{cases}
\frac{\delta_j^2\msr^{-4}}{|\xi_i-\xi_j|^2} &\text{if } \msr_{ij}=\frac{|\xi_i-\xi_j|}{\sqrt{\delta_i\delta_j}} \\
\delta_i^{-2}\delta_j^2\msr^{-2} \int_{\{t\ge \msr^2\}}t^{-4} dt &\text{if } \msr_{ij}=\sqrt{\frac{\delta_i}{\delta_j}} \\
\frac{\delta_i^2}{\delta_j^2}\msr^{-2} \int_{\{t\ge \msr^2\}}1 dt &\text{if } \msr_{ij}=\sqrt{\frac{\delta_j}{\delta_i}}
\end{cases}\right\} \lesssim \msr^{-4}\simeq Q.
\end{align*}
Thus, the claim \eqref{hzj} holds as desired.

\medskip
Consequently, by testing the linearized equation \eqref{eq:lin} against the functions $PZ_j^k$ and using \eqref{lc0}, \eqref{pupzj}, and \eqref{hzj}, we obtain \eqref{eq:cjb}.
\end{proof}

\begin{lemma}
For any $x \in \Omega$ and sufficiently large $M > 1$, the following inequality holds:
\begin{align}
&\ \int_{\Omega} \frac{1}{|x-\om|^4} \(\sigma W\)(\om) d\om \nonumber \\
&\lesssim \sum_{i=1}^{\nu} \(w_{2i}^{\tin}+ w_{2i}^{\tout}\) \left[M^2\frac{\log(2+|x_i|)}{\la x_i \ra} \mone_{\{|x_i| < \msr^2\}} + M^2\frac{\log |x_i|}{|x_i|^2}\mone_{\{|x_i| \ge \msr^2\}} + M^4\msr^{-2} + M^{-1}\right] \nonumber \\
&\ + \sum_{i=1}^{\nu} \left[\msr^{-2}+M^{-2}+\max_i\delta_iM^4+\frac{\log(2+|x_i|)}{\langle x_i\rangle^2}\right]w_{1i} \label{eq:323} \\
&\ + \sum_{i=1}^{\nu} \left[w_{3i}^{\tin}\frac{\log(2+|x_i|)}{\langle x_i\rangle^2} + w_{3i}^{\tout}\frac{\log|x_i|}{|x_i|^2}\right] \mone_{\{\nu=1\}}
=: \olw(x). \nonumber
\end{align}
\end{lemma}
\begin{proof}
Without loss of generality, we can assume that $\delta_i\geq \delta_j$ for $1 \le i\neq j \le \nu$. We recall the notations $x_i=\delta_i^{-1}(x-\xi_i)$ and $z_{ij}=\delta_i^{-1}(\xi_j-\xi_i)$.

We first consider the cross terms involving $w_{1i}$, namely, $U_iw_{1j}$ and $U_jw_{1i}$. We will show that
\begin{equation}\label{ujw1i}
U_jw_{1i}=\sum_{i=1}^{\nu}(v_{1i}+v_{2i}^{\tin}+v_{2i}^{\tout})(\max_i\delta_i M^4+\msr^{-2}+M^{-2})
\end{equation}
by dividing two cases.

\medskip
\noindent\fbox{Case 1:} Suppose that $|\xi_i-\xi_j|\leq M \delta_i$. Then $\sqrt{\frac{\delta_i}{\delta_j}}\le \msr \le M \sqrt{\frac{\delta_i}{\delta_j}}$ and $w_{1i} \lesssim 1$.

If $\frac{|x-\xi_j|}{\delta_j}\le \msr^2$, then $|x-\xi_j|\le M^2\delta_i$, leading to
\[U_{j}w_{1i} \lesssim v_{2j}^{\tin} \delta_j^2\msr^4 \lesssim v_{2j}^{\tin}\delta_i^2M^4.\]

If $\frac{|x-\xi_j|}{\delta_j}\ge \msr^2$, then $|x-\xi_j|\ge \delta_i$, resulting in
\[U_{j}w_{1i} \lesssim v_{2j}^{\tout} \delta_j\msr^2 \lesssim v_{2j}^{\tout}\delta_iM^2.\]

\noindent\fbox{Case 2:} Suppose that $|\xi_i-\xi_j|\geq M \delta_i$. Then, $\msr=\frac{|\xi_i-\xi_j|}{\sqrt{\delta_i\delta_j}}$.

When $|x-\xi_i|\ge \frac{|\xi_i-\xi_j|}{2}$, then
\[U_{j}w_{1i} \lesssim v_{1j} \frac{\delta_i^2}{|\xi_i-\xi_j|^2} \lesssim v_{1j}M^{-2}.\]

When $|x-\xi_i|\le \frac{|\xi_i-\xi_j|}{2}$, then $|x-\xi_j|\gtrsim \frac{|\xi_i-\xi_j|}{2}$. Using $\delta_j\le \delta_i$, we deduce
\[U_{j}w_{1i} \lesssim v_{1i}(\delta_i^2+|x-\xi_i|^2) \frac{\delta_j^2}{|\xi_i-\xi_j|^4} \lesssim v_{1i}\msr^{-2}.\]

\medskip
We turn to handling $U_iw_{1j}$. Applying Young's inequality and using $\delta_j\le \delta_i$ once again, we obtain
\begin{equation}\label{uiw1j}
U_{i}w_{1j} \lesssim \frac{\delta_j}{\delta_i} \frac{\delta_i^4}{(\delta_i^2+|x-\xi_i|^2)^3} + \frac{\delta_j^4}{(\delta_j^2+|x-\xi_j|^2)^3}
\lesssim \frac{\delta_i^{-2}}{\la x_i\ra^4}w_{1i}+\frac{\delta_j^{-2}}{\la x_j\ra^4}w_{1j}.
\end{equation}
By adapting the arguments from \cite[Lemma 4.2]{DSW}, we establish the following estimates:
\begin{align}
U_jw_{2i}^{\tin} &\lesssim \msr^{-2}\(v_{2j}^{\tin} + v_{2j}^{\tout} + v_{2i}^{\tin}\),\label{eq:upb1}\\
U_jw_{2i}^{\tout} &\lesssim \msr^{-2}\(v_{2j}^{\tin}+v_{2j}^{\tout} + v_{2i}^{\tout}\),\label{eq:upb3}\\
U_iw_{2j}^{\tin} &\lesssim \la z_{ij} \ra^{-2}\(v_{2i}^{\tin}+v_{2j}^{\tin}\)+\msr^{-2}\la z_{ij} \ra^{-1} v_{2i}^{\tout},\label{eq:upb4}\\
U_iw_{2j}^{\tout} &\lesssim \la z_{ij} \ra^{-1} v_{2i}^{\tin} + \msr^{-2}v_{2i}^{\tout} + \la z_{ij} \ra^{-2}v_{2j}^{\tout}, \label{eq:upb5}
\end{align}
and
\begin{align}
U_i\(w_{2j}^{\tin}+w_{2j}^{\tout}\) \lesssim \left[\(\tfrac{\delta_{j}}{\delta_{i}}\)^2+\theta^2\right] v_{2j}^{\tout} &\quad \text{if } |x_i-z_{ij}| \le \theta,\label{eq:upb2}\\
w_{2j}^{\tin}+w_{2j}^{\tout} \lesssim \la z_{ij} \ra^{5} \theta^3 \(w_{2i}^{\tin}+w_{2i}^{\tout}\) &\quad \text{if } |x_i-z_{ij}| \ge \theta \label{eq:upb6}
\end{align}
for any $\theta \in (0,1)$.

\medskip
Lastly, letting $\om_i=\frac{\om-\xi_i}{\delta_i}$, we check that
\[\int_{\Omega} \frac{1}{|x-\om|^4}[v_{1i}+v_{2i}^{\tin}+v_{2i}^{\tout}+v_{3i}^{\tin}+v_{3i}^{\tout}](\om) d\om \le C(w_{1i}+w_{2i}^{\tin}+w_{2i}^{\tout}+w_{3i}^{\tin}+w_{3i}^{\tout})\]
and
\begin{align*}
&\ \int_{\Omega} \frac{1}{|x-\om|^4}\frac{\delta_i^{-2}}{\la \om_i\ra^4}[w_{1i}+w_{2i}^{\tin}+w_{2i}^{\tout}+w_{3i}^{\tin}+w_{3i}^{\tout}](\om) d\om \\
&\le C\left[\frac{\log(2+|x_i|)}{\la x_i\ra^2}(w_{1i}+w_{3i}^{\tin})+ \frac{\log(2+|x_i|)}{\la x_i\ra}w_{2i}^{\tin}+\frac{\log|x_i|}{|x_i|^2}(w_{2i}^{\tout}+w_{3i}^{\tout})\right].
\end{align*}

\medskip
Now, taking the above estimates yield \eqref{eq:323}.
\end{proof}

We are ready to complete the proof of Propositions \ref{prop:31} and \ref{prop:33}.
\begin{proof}[Proof of Proposition \ref{prop:31}]
Once \eqref{eq:311} is established, the Fredholm alternative principle will give us the existence and uniqueness of solution $\rho_0$ to \eqref{eq:lin} for a given $h$ with $\|h\|_{**}<\infty$. As a consequence, it is sufficient to prove \eqref{eq:311}.

\medskip
We argue by contradiction. Suppose that there exist parameters $\{(\delta_{i,m}, \xi_{i,m})\}_{m\in \N}$, functions $\{\rho_{0,m}\}_{m\in \N}$ and $\{h_m\}_{m\in \N}$, and numbers $\{c_{i,m}^k\}_{m\in \N}$
such that $\xi_{i,m} \in \Omega$, $d(\xi_{i,m},\pa\Omega) \gtrsim 1$ if $\nu\ge 2$,
\[\delta_{i,m} + \frac{\delta_{i,m}}{d(\xi_{i,m}, \pa\Omega)} + \|h_m\|_{**} \to 0 \quad \text{as } m\to \infty,
\quad \text{and} \quad \|\rho_{0,m}\|_*=1 \text{ for all } m\in \N.\]
We also assume that these sequences satisfy
\begin{align}\label{suct}
\begin{cases}
\displaystyle(-\Delta-\lambda)\rho_{0,m} - 2(u_0+\sigma_m)\rho_{0,m}= h_m + \sum_{i=1}^{\nu} \sum_{k=0}^6 c_{i,m}^k (-\Delta-\lambda) PZ_{i,m}^k \quad \text{in } \Omega\subset \R^6,\\
\displaystyle \rho_{0,m}=0 \quad \text{on } \pa\Omega,\ c_{1,m}^0, \ldots, c_{\nu,m}^6 \in \R,\\
\displaystyle \big\langle \rho_{0,m}, PZ_{i,m}^k \big\rangle_{H^1_0(\Omega)} = 0 \quad \text{for } i=1,\ldots,\nu \text{ and } k=0,\ldots,6,
\end{cases}
\end{align}
where $PU_{i,m}=PU_{\delta_{i,m},\xi_{i,m}}$. Moreover, let $V_m$, $W_m$, $\olw_m$, $Q_m$, and $\msr_m$ denote the functions and quantities corresponding to $V$, $W$, $\olw$, $Q$, and $\msr$, respectively,
where $(\delta_i, \delta_j, \xi_i,\xi_j)$ are replaced by $(\delta_{i,m},\delta_{j,m}, \xi_{i,m}, \xi_{j,m})$; see Definition \ref{norm}, \eqref{eq:323}, and \eqref{rq}.

By virtue of \eqref{eq:cjb} and Definition \ref{norm}, we observe
\begin{equation}\label{cimk}
\begin{aligned}
&\ \int_{\Omega} \frac{1}{|x-\om|^4}\bigg| \sum_{k=0}^6\sum_{i=1}^{\nu} c_{i,m}^k(-\Delta-\lambda) PZ_{i,m}^k\bigg|(\om)d\om \\
&\lesssim \int_{\Omega} \frac{1}{|x-\om|^4}\sum_{k=0}^6\sum_{i=1}^{\nu} c_{i,m}^k \(U_{i,m}^2+U_{i,m}\)(\om)d\om \lesssim \sum_{k=0}^6\sum_{i=1}^{\nu}|c_{i,m}^k|U_{i,m} \\
&\lesssim \begin{medsize}
\displaystyle \sum_{k=0}^6\sum_{i=1}^{\nu} \bigg[\delta_{i,m}^2\delta_{i,m}^{-2}w_{1i,m}+Q_m\msr_m^4(w_{2i,m}^{\tin}+w_{2i,m}^{\tout})
+\frac{\delta_{i,m}^4}{d(\xi_{i,m},\pa\Omega)^4}\delta_{i,m}^{-2}(w_{3i,m}^{\tin}+w_{3i,m}^{\tout})\bigg] (o(\|\rho_{0,m}\|_{*})+\|h_m\|_{**})
\end{medsize} \\
&= o_m(1)W_m(x).
\end{aligned}
\end{equation}
Here, $o_m(1) \to 0$ as $m \to \infty$, and we exploit the precise estimate of $c_i^k$ presented in \eqref{eq:cjb} to deduce the second inequality.

Given the nondegeneracy and boundedness of $u_0$, we know the Green's function of the operator $-\Delta-\lambda-2u_0$ with Dirichlet boundary condition is bounded by $C_0\frac{1}{|x-y|^4}$ for some constant $C_0>0$. Combining this fact with $\|h_m\|_{**}=o_m(1)$ and \eqref{cimk}, one has
\begin{equation}\label{poin}
|\rho_{0,m}(x)|\le C_0 \int_{\Omega} \frac{1}{|x-\om|^4} \(\sigma_m |\rho_{0,m}|\)(\om) d\om+o_m(1)W_m(x).
\end{equation}
To complete the proof, we will prove that for any given $\tau \in (0,1)$, there exists a number $m_{\tau} \in \N$ depending on $\tau$ such that
\begin{equation}\label{eq:claim2}
m \ge m_{\tau} \ \Rightarrow \ C_0\int_{\Omega} \frac{1}{|x-\om|^4} \(\sigma_m|\rho_{0,m}|\)(\om) d\om \le \tau W_m(x) \quad \text{for all } x \in \Omega.
\end{equation}

\medskip
Without loss of generality, we may assume that
\[\begin{cases}
\delta_{1,m} \ge \delta_{2,m} \ge \cdots \ge \delta_{\nu,m} \text{ for all } m \in \N,\\
\displaystyle \text{either } \lim_{m \to \infty} z_{ij,m} = z_{ij,\infty} \in \R^6 \text{ or } \lim_{m \to \infty} |z_{ij,m}| \to \infty,
\end{cases}\]
where $z_{ij,m} := \delta^{-1}_{i,m}(\xi_{j,m}-\xi_{i,m}) \in \R^6$.
We define
\[\mcd(i):=\left\{j \in \{1,\ldots,\nu\}: i < j \text{ and } \lim_{m \to \infty} |z_{ij,m}| \in \R\right\}\]
and $x_{i,m} := \delta^{-1}_{i,m}(x-\xi_{i,m}) \in \delta_{i,m}^{-1}(\Omega-\xi_{i,m})$. For large $L > 1$ and small $\vep \in (0,1)$, we introduce
\[\Omega_{i,m}:=\left\{x \in \Omega: |x_{i,m}| \le L,\, |x_{i,m}-z_{ij,m}| \ge \vep \text{ for all } j \in \mcd(i)\right\}\]
and
\[\mca_{i,m}:=\bigcup_{j \in \mcd(i)} \bigg[\left\{x \in \Omega: |x_{i,m}-z_{ij,m}| < \vep\right\} \setminus \bigcup_{\ell \in \mcd(i)} \left\{x \in \Omega: |x_{\ell,m}| \le L\right\}\bigg].\]
Using these definitions, we decompose $\Omega$ into three disjoint subsets:
\[\Omega = \Omega_{\ext} \cup \Omega_{\core} \cup \Omega_{\neck},\ \footnotemark
\]where\footnotetext{If $\nu=1$, then $\Omega_{\neck}=\emptyset$ and $\Omega_{\core}=\{x\in \Omega:|x_{1,m}|\le L\}$. For $\nu\ge 2$, we essentially use the bubble-tree structure introduced by \cite[Subsection 4.2]{DSW}.}
\[\Omega_{\ext} := \bigcap_{i=1}^{\nu} \{x \in \Omega: |x_{i,m}| > L\},\quad \Omega_{\core} := \bigcup_{i=1}^{\nu} \Omega_{i,m},\quad \Omega_{\neck} := \bigcup_{i=1}^{\nu} \mca_{i,m}.\]
Subsequently, we express
\begin{align*}
C_0\int_{\Omega} \frac{1}{|x-\om|^4} \(\sigma_m|\rho_{0,m}|\)(\om) d\om &= C_0\(\int_{\Omega_\ext} +\int_{\Omega_{\core}}+ \int_{\Omega_{\neck}}\) \frac{1}{|x-\om|^4} \(\sigma_m|\rho_{0,m}|\)(\om) d\om \\
&=: \I_{\ext}(x) + \I_{\core}(x) + \I_{\neck}(x) \quad \text{for all } x \in \Omega.
\end{align*}

\medskip
Owing to \eqref{poin} and \eqref{eq:323}, we have
\[|\rho_{0,m}(x)|\le CC_0\olw_m(x)+o_m(1)W_m(x) \quad \text{for } x\in \Omega.\]
Thus, there exists suitable constants $L, M>0$ such that
\[M^2L^{-1}\log L+M^4\msr_m^{-2} + M^{-1}+\max_i\delta_{i,m} M^4\lesssim c_{0,m},\]
where $c_{0,m}>0$ is sufficiently small, which leads to
\begin{equation}\label{pinex}
|\rho_{0,m}(x)|\le (CC_0c_{0,m}+o_m(1))W_m(x) \quad \text{for } x\in \Omega_{\ext}.
\end{equation}
Then, using \eqref{eq:323} again, we arrive at
\begin{equation}\label{ext}
\I_{\ext}(x) \le (CC_0c_{0,m}+o_m(1))CC_0\olw_m(x) \le \frac{\tau}{3} W_m(x)
\end{equation}
for $m \in \N$ large enough.

\medskip
If we establish
\begin{equation}\label{core0}
|\rho_{0,m}(x)|=o_m(1)W_m(x) \quad \text{for } x\in \Omega_{\core},
\end{equation}
it will follow from \eqref{eq:323} and $\olw_m \lesssim W_m$ again that
\begin{equation}\label{core}
\I_{\core}(x) \le o_m(1)C_0\int_{\Omega_{\core}} \frac{1}{|x-\om|^4} \(\sigma_mW_m\)(\om) d\om \le \frac{\tau}{3} W_m(x)
\end{equation}
for $m \in \N$ large enough.

In the following, we derive \eqref{core0}. Because of $\|\rho_{0,m}\|_*=1$ and \eqref{pinex}, there exist $i_0\in \{1,..,\nu\}$ and $\bar{x}_m\in B(\xi_{i_0,m}, \delta_{i_0,m}L)$ such that
\begin{equation}\label{rhbx}
|\rho_{0,m}(\bar{x}_m)|\ge \frac12 W_m(\bar{x}_m)
\end{equation}
for $m \in \N$ large enough. Denoting
$\bar{\rho}_{0,m}(y) := W_m(\bar{x}_m)^{-1} \rho_{0,m}(\delta_{i_0,m}y+\xi_{i_0,m})$, we observe that
\[\frac{\sum_{i=1}^{\nu}w_{1i,m}(\delta_{i_0,m}y+\xi_{i_0,m})}{W_m(\bar{x}_m)} \lesssim \la\frac{\bar{x}_m-\xi_{i_0,m}}{\delta_{i_0,m}}\ra^2 \lesssim L^2\]
and
\[\frac{w_{3i_0,m}(\delta_{i_0,m}y+\xi_{i_0,m})}{W_m(\bar{x}_m)} \mone_{\{\nu=1\}} \lesssim 1.\]
Arguing as Case 2 in \cite[Lemma 5.1]{DSW}, we also have that given $\ka > 2\max\{L,\vep^{-1}\}$,
\begin{multline*}
|\bar{\rho}_{0,m}(y)|\lesssim L^2+\sum_{j\in\mcd(i_0)}\frac{1}{|y-z_{i_0j,m}|^3}\mone_{\{\nu\ge 2\}} \\
\text{for } y \in \mathcal{K}_\ka : =\left\{y\in \R^6: |y|\le \ka \text{ and } |y-z_{i_0j,\infty}|\ge \ka^{-1} \text{ for } j\in\mcd(i_0)\right\},
\end{multline*}
and so there exists $\bar{\rho}_{0,\infty} \in D^{1,2}(\R^6)$ such that, up to subsequence,
\[\bar{\rho}_{0,m}\to \bar{\rho}_{0,\infty} \quad \text{in } C^0_{\loc}(\R^6\setminus \bar{\mcz}_\infty) \quad \text{as } m\to\infty,\]
where $\bar{\mcz}_\infty:=\{z_{i_0j,\infty}, j\in \mcd(i_0)\}$.\footnote{If $\nu=1$, then $\bar{\mcz}_\infty=\emptyset$.}
From \eqref{suct}, we obtain that
\begin{align}
&-\Delta \bar{\rho}_{0,\infty} = pU^{p-1}\bar{\rho}_{0,\infty} \quad \text{in } \R^6\setminus\bar{\mcz}_\infty, \label{brinl}\\
&|\bar{\rho}_{0,\infty}(y)|\lesssim L^2+\sum_{j\in\mcd(i_0)}\frac{1}{|y-z_{i_0j,\infty}|^3}\mone_{\{\nu\ge 2\}} \quad \text{for }\R^6\setminus\bar{\mcz}_\infty, \label{sin}\\
&\int_{\R^6} \bar{\rho}_{0,\infty} U^{p-1}Z^k =0 \quad \text{for } k=0,1,\ldots,6. \label{orbr}
\end{align}
We claim each singularity $z_{i_0j,\infty}$ of $\bar{\rho}_{0,\infty}$ is removable if $\nu\ge 2$. Inequality \eqref{sin} implies that $\bar{\rho}_{0,\infty} \lesssim 1$ if $|y-z_{i_0j,\infty}|\ge 1>0$ for each $j\in\mcd(i_0)$. So it suffices to prove that
\begin{equation}\label{eq:remove}
|\bar{\rho}_{0,\infty}(y)|\lesssim 1 \quad \text{if } y\in B(z_{i_0j,\infty}, 1) \quad \text{for any } j.
\end{equation}
We choose a small number $c>0$ such that $c\le \min \{\frac12 |z_{i_0j_1,\infty}-z_{i_0j_2,\infty}|: j_1\ne j_2,\, j_1, j_2\in \mcd(i_0)\}.$ Then
\begin{equation}\label{rem}
\begin{aligned}
|\bar{\rho}_{0,\infty}(y)| &\lesssim 1+\sum_{j\in\mcd(i_0)}\int_{B(z_{i_0j,\infty}, c)} \frac{1}{|y-\om|^4}\frac{1}{(1+|\om|^2)^2}\frac{1}{|\om-z_{i_0j,\infty}|^3} d\om\mone_{\{\nu\ge 2\}} \\
&\lesssim 1+\sum_{j\in \mcd(i_0)}\frac{1}{|y-z_{i_0j,\infty}|}\mone_{\{\nu\ge 2\}}.
\end{aligned}
\end{equation}
Applying \eqref{rem} again, we deduce \eqref{eq:remove}. Thus, $\bar{\rho}_{0, \infty}$ can be extended to a function in $L^{\infty}(\R^n)$ satisfying equation \eqref{brinl} in $\R^n$.
By the orthogonality conditions \eqref{orbr}, we conclude that $\bar{\rho}_{0,\infty}=0$, contradicting \eqref{rhbx}.
As a result, \eqref{core0} and so \eqref{core} are established.

\medskip
The only remaining task is to estimate $\I_{\neck}$ for $\nu\ge 2$. We claim that
\begin{equation}\label{siw1}%\label{ne1}
C_0\sum_{i=1}^{\nu}\int_{\mca_{i,m}} \frac{1}{|x-\om|^4}\bigg(\sigma_m\sum_{j=1}^{\nu}w_{1j,m}\bigg)(\om)d\om\le \frac{\tau}{6}W_m(x)
\end{equation}
for $m \in \N$ large enough. We write $x_{ji,m}:=\delta_{i,m}^{-1}(x-\xi_{j,m})$, $\om_{i,m}:=\delta_{i,m}^{-1}(\om-\xi_{i,m})$, and $\om_{ji,m}:=\delta_{i,m}^{-1}(\om-\xi_{j,m})$.
Also, we set $C_*:=1+\max\{|z_{ij}|: i,j=1,\ldots,\nu,\ j\in \mcd(i)\}$. A straightforward computation with the choice $L' \ge 2C_*$ yields that
\begin{equation}\label{ne2}
\begin{aligned}
\int_{\mca_{i,m}} \frac{1}{|x-\om|^4}\frac{\delta_{i,m}^{-2}}{\la \om_{i,m}\ra^6} d\om
&\lesssim \sum_{j\in\mcd(i)} \int_{B(\xi_{j,m},\delta_{i,m}\vep) \setminus \bigcup_{\ell \in \mcd(i)} B(\xi_{\ell,m},\delta_{\ell,m} L)} \frac{1}{|x-\om|^4}\frac{\delta_{i,m}^{-2}}{\la \om_{i,m}\ra^6} d\om \\
& \lesssim \begin{cases}
\displaystyle \vep^6 \sum_{j\in\mcd(i)} \frac{1}{|x_{ji,m}|^4} &\text{if } |x-\xi_{j,m}|\ge \delta_{i,m} L', \\
\displaystyle \vep^2 &\text{if } 2\ep\delta_{i,m}\le |x-\xi_{j,m}|\le \delta_{i,m} L',\\
\displaystyle \vep^2 &\text{if } |x-\xi_{j,m}|\le 2\ep\delta_{i,m}
\end{cases} \\
&\lesssim \begin{cases}
\displaystyle \vep^6(L')^{-2}\frac{1}{|x_{i,m}|^2} &\text{if } |x-\xi_{j,m}|\ge \delta_{i,m} L' \\
\displaystyle \vep^2[1+(L'+C_*)^2]\frac{1}{\la x_{i,m}\ra^2} &\text{if } 2\ep\delta_{i,m}\le |x-\xi_{j,m}|\le \delta_{i,m} L'\\
\displaystyle \vep^2[1+(2\ep+C_*)^2]\frac{1}{\la x_{i,m}\ra^2} &\text{if } |x-\xi_{j,m}|\le 2\ep\delta_{i,m}
\end{cases} \\
&\lesssim \vep^2L'^2W_m(x),
\end{aligned}
\end{equation}
where we used
\[\frac{|x_{ji,m}|}2 \lesssim |x_{ji,m}|-|z_{ij,m}|\le |x_{i,m}|\le |x_{ji,m}|+|z_{ij,m}|\lesssim \frac{3|x_{ji,m}|}{2} \quad \text{for } |x_{ji,m}|\ge L'\gtrsim 2|z_{ij,m}|\]
to get the third inequality.

Additionally, we conduct computations
\[\int_{\mca_{i,m}} \sum_{l\in \mcd(i)} \frac{1}{|x-\om|^4}\frac{\delta_{l,m}^{-2}}{\la \om_{l,m}\ra^6} d\om
\lesssim L^{-2} \sum_{l\in \mcd(i)} \int_{\Omega} \frac{1}{|x-\om|^4}\frac{\delta_{l,m}^{-2}}{\la \om_{l,m}\ra^4} d\om \lesssim L^{-2}W_m(x)\]
and
\begin{align*}
&\ \int_{\mca_{i,m}} \sum_{l\in \{\delta_{l,m}^{-1}\ll \delta_{i,m}^{-1},\ \lim\limits_{m \to \infty} |z_{il,m}|\in\R\}} \frac{1}{|x-\om|^4}\frac{\delta_{l,m}^{-2}}{\la \om_{l,m}\ra^6} d\om\\
&\lesssim \sum_{l\in \{\delta_{l,m}^{-1}\ll \delta_{i,m}^{-1},\ \lim\limits_{m\to \infty} |z_{il,m}|\in\R\}} \(\frac{\delta_{i,m}}{\delta_{l,m}}\)^2 \sum_{j\in\mcd(i)}\int_{B(0,\ep)} \frac{1}{|x_{ji,m}-\om_{ji,m}|^4} d\om_{ji,m} \lesssim o_m(1)W_m(x),
\end{align*}
where we adapted the strategy in \eqref{ne2} to obtain the last inequality.

In addition, we analyze
\begin{align*}
&\ \int_{\mca_{i,m}} \sum_{l\in \{\lim\limits_{m\to \infty}|z_{il,m}| = \infty\}} \frac{1}{|x-\om|^4}\frac{\delta_{l,m}^{-2}}{\la \om_{l,m}\ra^6} d\om \\
&\lesssim \sum_{l\in \{\lim\limits_{m\to \infty} |z_{il,m}| = \infty\}} \int_{B(\xi_{i,m},\delta_{i,m}L)}\delta_{i,m}^{-2} |z_{il,m}|^{-2}\frac{1}{|x-\om|^4}\frac{1}{\la \om_{l,m}\ra^4} d\om \lesssim o_m(1)W_m(x).
\end{align*}
By recalling \eqref{ujw1i} and \eqref{uiw1j}, and taking proper $\vep$, $L'$, $L$ and $m$, we obtain \eqref{siw1} for $m \in \N$ large enough.

On the other hand, using \eqref{eq:upb1}--\eqref{eq:upb6} and applying an analogous argument as above, we demonstrate that
\begin{equation}\label{siw2}
C_0\int_{\Omega_{\neck}} \frac{1}{|x-\om|^4} \bigg[\sigma_m\sum_{j=1}^{\nu}\(w_{2j,m}^{\tin}+w_{2j,m}^{\tout}\)\bigg](\om) d\om \le \frac{\tau}{6}W_m(x)
\end{equation}
for $m \in \N$ large enough.

It follows from \eqref{siw1} and \eqref{siw2} that
\begin{equation}\label{neck}
\I_{\neck}(x) \le \frac{\tau}{3} W_m(x)
\end{equation}
for $m \in \N$ large enough.

\medskip
Now, estimate \eqref{eq:claim2} is a consequence of \eqref{ext}, \eqref{core}, and \eqref{neck}. We complete the proof.
\end{proof}

\begin{proof}[Proof of Proposition \ref{prop:33}]
The proof follows the spirit of the argument used in \cite[Proposition 5.4]{DSW}. We initiate by checking the uniform bound
\[\left\| \I_1+\I_2+\I_{31}\right\|_{**} \le C,\]
which is a direct consequence of the estimates established in \eqref{i1po}, \eqref{vaor}, and Lemma \ref{lem2.1}. Denoting
\[h:=\I_1+\I_2+\I_{31}+\I_0[\rho_0]\]
and realizing the estimates
\[\frac{w_{1i}^2}{v_{1i}} \lesssim \delta_i^4, \quad \frac{(w_{3i}^{\tin})^2}{v_{3i}^{\tin}} \lesssim \(\frac{\delta_i}{d(\xi_i,\pa\Omega)}\)^4,\]
\[\frac{(w_{3i}^{\tout})^2}{v_{3i}^{\tout}} \lesssim \(\frac{\delta_i}{d(\xi_i,\pa\Omega)}\)^3 \frac{1}{|x_i|}\mone_{\left\{|x_i|\ge \frac{d(\xi_i,\pa\Omega)}{\delta_i}\right\}} \lesssim \(\frac{\delta_i}{d(\xi_i,\pa\Omega)}\)^4,\]
one can invoke Proposition \ref{prop:31} and the Banach fixed-point theorem to achieve the existence of a solution $\rho_0$ to \eqref{eq:31} satisfying \eqref{eq:33}.
Next, we test equation \eqref{eq:31} against $\rho_0$. From
\begin{align*}
\|W\|_{L^3(\Omega)} \lesssim
\max_i\delta_i^2|\log \delta_i|^{\frac13} +\max_{i}\(\frac{\delta_i}{d(\xi_i,\pa\Omega)}\)^4\left|\log \frac{d(\xi_i,\pa\Omega)}{\delta_i}\right|^{\frac13}\mone_{\{\nu=1\}} +
Q|\log Q|^{\frac13}, \\ %\label{wl2}
\|V\|_{L^{\frac32}(\Omega)} \lesssim
\max_i\delta_i^2|\log \delta_i|^{\frac23} +\max_{i}\(\frac{\delta_i}{d(\xi_i,\pa\Omega)}\)^4\left|\log \frac{d(\xi_i,\pa\Omega)}{\delta_i}\right|^{\frac23}\mone_{\{\nu=1\}}+
 Q|\log Q|^{\frac23}, %\label{vl2}
\end{align*}
we have
\begin{align*}
\|\rho_0\|^2_{H^1_0(\Omega)} &\le \int_{\Omega}2(u_0+\sigma)\rho_0^2 + \(|\I_1|+|\I_2|+|\I_{31}|+|\I_0[\rho_0]|\)|\rho_0| \\
&\lesssim \|\rho_0\|^2_{*}\int_{\Omega}2(u_0+\sigma)W^2(x) + \|\rho_0\|_{*}\int_{\Omega}V(x)W(x) +\|\rho_0\|_{*}^3\int_{\Omega}W^3 \\
&\lesssim \|W\|_{L^3(\Omega)}^2+\|V\|_{L^{\frac32}(\Omega)}\|W\|_{L^3(\Omega)}\\
&\lesssim \max_i\delta_i^4|\log \delta_i| +\max_{i}\(\frac{\delta_i}{d(\xi_i,\pa\Omega)}\)^{8}\left|\log \frac{d(\xi_i,\pa\Omega)}{\delta_i}\right|\mone_{\{\nu=1\}} + Q^2|\log Q|,
\end{align*}
yielding \eqref{eq:331}. This completes the proof.
\end{proof}

\section{Proof of Theorem \ref{thm1}}\label{sec3}
The proof of Theorem \ref{thm1} is divided into two parts: In Subsection \ref{subs3.1}, we prove that \eqref{eq:sqe32} holds.
In Subsection \ref{subs3.2}, we show that this estimate is optimal.

\subsection{Proof of estimate \eqref{eq:sqe32}}\label{subs3.1}
If $n=6$, we set $\rho_0$ by \eqref{eq:31} when $n=6$. If $n=3,4,5$ or $n \ge 7$, we set $\rho_0 = 0$.
Define also $\rho_1 := \rho - \rho_0$. In light of \eqref{eqrho2} and \eqref{eq:31}, the function $\rho_1$ satisfies the following boundary value problem
\begin{align}\label{eq:34}
\begin{cases}
\displaystyle(-\Delta-\lambda)\rho_1 - [(u_0+\sigma+\rho_0+\rho_1)^{p}- |u_0+\sigma+\rho_0|^{p-1}(u_0+\sigma+\rho_0)] \\
\hspace{55pt} = \begin{cases}
\displaystyle f+\I_1+\I_2+\I_3 &\text{if } n \ne 6,\\
\displaystyle f+(\I_3-\I_{31})-\sum_{i=1}^{\nu} \sum_{k=0}^n c_i^k (-\Delta-\lambda) PZ_i^k &\text{if } n = 6
\end{cases} \quad \text{in } \Omega,\\
\displaystyle \rho_1=0 \quad \text{on }\pa\Omega, \\
\displaystyle \big\langle \rho_1, PZ_i^k \big\rangle_{H^1_0(\Omega)} = 0 \quad \text{for all } i = 1,\ldots,\nu \text{ and } k = 0,\ldots,n.
\end{cases}
\end{align}
Next, we establish the $H^1_0(\Omega)$-norm estimate of $\rho_1$.
\begin{prop}\label{prop:34}
Assume that $\ep_1>0$ is small enough. There exists a constant $C>0$ depending only on $n,\, \nu,\, \lambda,\, u_0$, and $\Omega$ that
\begin{equation}\label{eq:341}
\begin{aligned}
\|\rho_1\|_{H^1_0(\Omega)} &\le C\bigg[\|f\|_{(H^1_0(\Omega))^*}+ \bigg(\|\I_1\|_{L^{\frac{p+1}{p}}(\Omega)} + \|\I_2\|_{L^{\frac{p+1}{p}}(\Omega)} + \|\I_3\|_{L^{\frac{p+1}{p}}(\Omega)}\bigg)\mone_{\{n\ne 6\}} \\
&\qquad + \left\|\I_3-\I_{31}\right\|_{L^{\frac{p+1}{p}}(\Omega)} \mone_{\{n=6\}} + \sum_{i=1}^{\nu} \sum_{k=0}^n |c_i^k| \mone_{\{n=6\}}\bigg].
\end{aligned}
\end{equation}
\end{prop}
\noindent To obtain analogous estimates to \eqref{eq:341} in \cite{DSW, CK, CKW}, the authors decomposed $\rho_1$ into smaller pieces by introducing auxiliary parameters, and analyzed each part relying on a coercivity inequality. See Subsection \ref{subc1.3}(5) for a prior discussion.
Our argument in this paper is direct. We first derive an $H^1_0(\Omega)$-norm estimate for the solution to the associated linear problem, whose proof is based on a blow-up argument.
\begin{lemma}
Let $\Pi^{\perp}: H^1_0(\Omega) \to \textup{span}\{PZ_i^k: i=1,\ldots,\nu \text{ and } k=0,\ldots,n\}^{\perp} \subset H^1_0(\Omega)$ be the projection operator.
For any functions $\vrh\in H^1_0(\Omega)$ and $h\in (H^1_0(\Omega))^*$ satisfying
\[\begin{cases}
\vrh-\Pi^{\perp}[(-\Delta -\lambda)^{-1}(p(u_0+\sigma)^{p-1})]=\Pi^{\perp}[(-\Delta -\lambda)^{-1}(h)] &\text{in }\Omega,\\
\vrh=0 &\text{on } \pa \Omega,\\
\la \vrh, PZ_i^k\ra=0 \quad \text{for } i=1,\ldots,\nu \text{ and } k=0,\ldots,n,\\
\end{cases}\]
it holds that
\begin{equation}\label{vrh1}
\|\vrh\|_{H^1_0(\Omega)} \lesssim \|h\|_{(H^1_0(\Omega))^*}.
\end{equation}
\end{lemma}
\begin{proof}
We proceed by contradiction. Suppose that there exist sequences of parameters $\{(\delta_{i,m},\xi_{i,m})\}_{m \in \N}$, and functions $\{\vrh_m\}_{m \in \N}$ and $\{h_m\}_{m \in \N}$ such that
\begin{equation}\label{vrh}
\begin{cases}
\displaystyle \max_i\delta_{i,m} + \max_i\frac{\delta_{i,m}}{d(\xi_{i,m},\pa\Omega)} + \|h_m\|_{(H^1_0(\Omega))^*} \to 0 &\text{as } m \to \infty,\\
\displaystyle \|\vrh_m\|_{H^1_0(\Omega)}=1 &\text{for all } m \in \N,
\end{cases}
\end{equation}
and
\begin{align}\label{2.12}
\begin{cases}
\displaystyle \vrh_m -(-\Delta-\lambda)^{-1}[p(u_0+\sigma_m)^{p-1}\vrh_m] = \Pi^{\perp}[(-\Delta-\lambda)^{-1}h_m] + \sum_{i=1}^{\nu}\sum_{k=0}^n \mu_{i,m}^k PZ^k_{i,m} &\text{in } \Omega,\\
\displaystyle \vrh_m=0 &\text{on } \pa\Omega,\\
\displaystyle \big\langle \vrh_m, PZ^k_{i,m} \big\rangle_{H^1_0(\Omega)} =0 \quad \text{for } i=1,\ldots,\nu \text{ and } k=0,1,\ldots,n.
\end{cases}
\end{align}
Here, $PU_{i,m} := PU_{\delta_{i,m},\xi_{i,m}}$, $PZ^0_{i,m} := \delta_{i,m} \frac{\pa PU_{i,m}}{\pa\delta_{i,m}}$, and $PZ^k_{i,m} := \delta_{i,m} \frac{\pa PU_{i,m}}{\pa\xi_{i,m}^k}$. Besides, $\mu_{i,m}^k \in \R$ denote Lagrange multipliers.

\medskip
First, we observe that
\begin{equation}\label{pih}
\begin{aligned}
\|\Pi^{\perp}[(-\Delta-\lambda)^{-1}h_m]\|_{H^1_0(\Omega)} &\lesssim \left\|(-\Delta-\lambda)^{-1}h_m + \sum_{i=1}^{\nu}\sum_{k=0}^n \frac{\int_{\Omega}h_mPZ_{i,m}^k }{\|PZ_{i,m}^k\|_{H^1_0(\Omega)}}\cdot PZ_{i,m}^k\right\|_{H^1_0(\Omega)} \\
&\lesssim \|h_m\|_{(H^1_0(\Omega))^*} + \sum_{i=1}^{\nu}\sum_{k=0}^n \left|\int_{\Omega}h_mPZ_{i,m}^k\right| \\
&\lesssim \|h_m\|_{(H^1_0(\Omega))^*}.
\end{aligned}
\end{equation}

\medskip
Second, we verify that
\begin{equation}\label{2.15}
\sum_{i=1}^{\nu}\sum_{k=0}^n |\mu_{i,m}^k| = o_m(1)
\end{equation}
where $o_m(1) \to 0$ as $m \to \infty$.

For this aim, we test \eqref{2.12} with $PZ^q_{j,m}$ for each $j \in \{1,\ldots,\nu\}$ and $q \in \{0,1,\ldots,n\}$. We only need to focus on
\begin{equation}\label{2.121}
\begin{aligned}
&\ \left|\int_{\Omega} \left[(-\Delta-\lambda)\vrh_m - p(u_0+\sigma_m)^{p-1}\vrh_m\right] PZ_{j,m}^q\right| \\
&\lesssim \left|\int_{\Omega} \left[(-\Delta-\lambda) PZ_{j,m}^q-p (PU_{j,m})^{p-1} PZ_{j,m}^q\right]\vrh_m\right| \\
&\ +\int_{\Omega} \left[\sigma_m^{p-1}- (PU_{j,m})^{p-1}\right]|\vrh_m||PZ_{j,m}^q| \mone_{\{\nu\ge 2\}} + \int_{\Omega} \left[(u_0+\sigma_m)^{p-1}-\sigma_m^{p-1}\right] |\vrh_m|U_{j,m}.
\end{aligned}
\end{equation}
We now estimate each of the integrals on the right-hand side of \eqref{2.121}.

It holds that
\begin{multline*}
\left|\int_{\Omega}((-\Delta-\lambda) PZ_{j,m}^q-p (PU_{j,m})^{p-1} PZ_{j,m}^q)\vrh_m\right| \lesssim \|\vrh_m\|_{H^1_0(\Omega)} \\
\times \left[\left\|(PU_{j,m})^{p-1} PZ_{j,m}^q-U_{j,m}^{p-1}Z_{j,m}^q\right\|_{L^{\frac{p+1}{p}}(\Omega)}
+\|U_{j,m}\|_{L^{\frac{p+1}{p}}(\Omega)}\mone_{\{\text{each } PU_{j,m} \text{ satisfies } \eqref{pu1}\}}\right].
\end{multline*}
Arguing as in \eqref{pup1} and \eqref{i3e}, we deduce
\[\left\|\left[(PU_{j,m})^{p-1}-U_{j,m}^{p-1}\right] PZ_{j,m}^q\right\|_{L^{\frac{p+1}{p}}(\Omega)} + \left\|U_{j,m}^{p-1} (PZ_{j,m}^q-Z_{j,m}^q)\right\|_{L^{\frac{p+1}{p}}(\Omega)} \lesssim J_{1,m},\]
where $J_{1,m}$ is the quantity $J_1$ with $(\delta_i, \delta_j, \xi_i,\xi_j)$ replaced by $(\delta_{i,m},\delta_{j,m}, \xi_{i,m}, \xi_{j,m})$.

Also, by applying the inequality $|PZ_j^q|\lesssim PU_j$ (which directly comes from the maximum principle) for $n\ge 6$, \eqref{iqu}, \eqref{ab1}, and H\"older's inequality, we obtain
\begin{align*}
&\ \int_{\Omega} \left[\sigma^{p-1}- (PU_{j,m})^{p-1}\right]|\vrh_m||PZ_{j,m}^q|\\
&\lesssim \begin{cases}
\displaystyle \int_{\Omega} \sum_{i\ne j} \left[(PU_{j,m})^{p-2}PU_{i,m} +(PU_{i,m})^{p-1}\right]|\vrh_m||PZ_{j,m}^q| &\text{if } n=3,4,5,\\
\displaystyle \int_{\Omega} \left[\sigma^{p-1}PU_{j,m}- (PU_{j,m})^p\right]|\vrh_m|
\lesssim \int_{\Omega} \left[\sigma^p-\sum_{i=1}^{\nu}(PU_{i,m})^p\right] |\vrh_m| &\text{if } n\ge 6
\end{cases}\\
&\lesssim \begin{cases}
\displaystyle \sum_{i\ne j} \|U_{i,m}^{p-1}U_{j,m}\|_{L^{\frac{p+1}{p}}(\Omega)} \|\vrh_m\|_{H^1_0(\Omega)} &\text{if } n=3,4,5,\\
\displaystyle \sum_{i\ne j} \left\|\min\{U_{i,m}^{p-1}U_{j,m}, U_{j,m}^{p-1}U_{i,m}\}\right\|_{L^{\frac{p+1}{p}}(\Omega)} \|\vrh_m\|_{H^1_0(\Omega)} &\text{if } n\ge 6.
\end{cases}
\end{align*}

On the other hand, using \eqref{ab1}, we have
\begin{align*}
\int_{\Omega} \left[(u_0+\sigma_m)^{p-1}-\sigma_m^{p-1}\right] |\vrh_m|U_{j,m} &\lesssim \int_{\Omega} \left[(u_0\sigma_m^{p-2}\mone_{\{u_0>0, p>2\}} + u_0^{p-1}\mone_{\{u_0>0\}}\right] |\vrh_m|U_{j,m} \\
&\lesssim \|U_{j,m}\|_{L^{\frac{p+1}{p}}(\Omega)}\mone_{\{u_0>0\}} + \sum_{i=1}^{\nu}\|U_{i,m}^{p-1}\|_{L^{\frac{p+1}{p}}(\Omega)}\mone_{\{u_0>0, p>2\}}.
\end{align*}

Therefore,
\begin{align}
&\ \bigg|\int_{\Omega} \big[(-\Delta-\lambda)\vrh_m - p(u_0+\sigma_m)^{p-1}\vrh_m\big] PZ_{j,m}^q\bigg| \nonumber \\
&\lesssim \|\vrh_m\|_{H^1_0(\Omega)} \left[\|U_{j,m}\|_{L^{\frac{p+1}{p}}(\Omega)}\mone_{\{u_0>0\}\cup\{\text{each } PU_{j,m} \text{ satisfies }\eqref{pu1}\}} + \sum_{i=1}^{\nu}\|U_{i,m}^{p-1}\|_{L^{\frac{p+1}{p}}(\Omega)} \mone_{\{u_0>0, p>2\}} \right. \nonumber \\
&\hspace{65pt} + \left. \left. \begin{cases}
\displaystyle \sum_{i\ne j} \|U_{i,m}^{p-1}U_{j,m}\|_{L^{\frac{p+1}{p}}(\Omega)} &\text{if } n=3,4,5 \\
\displaystyle \sum_{i\ne j} \left\|\min\{U_{i,m}^{p-1}U_{j,m}, U_{j,m}^{p-1}U_{i,m}\}\right\|_{L^{\frac{p+1}{p}}(\Omega)} &\text{if } n\ge 6
\end{cases}\right\} \mone_{\{\nu\ge 2\}} + J_{1,m}\right] \nonumber \\
&= o_m(1), \label{liar}
\end{align}
where the last equality follows from Lemmas \ref{a4} and \ref{a22}, \eqref{i1po}, \eqref{i12e}, and $\|\vrh_m\|_{H^1_0(\Omega)}=1$.

\medskip
Third, we assert that
\[\begin{cases}
\vrh_m \rightharpoonup 0 &\text{weakly in } H^1_0(\Omega),\\
\vrh_m \to 0 &\text{strongly in } L^s(\Omega) \text{ for } s\in (1, 2^*)
\end{cases}
\quad \text{as } m \to \infty.\]
Since $\|\vrh_m\|_{H^1_0(\Omega)}=1$, there exists $\vrh_{\infty} \in H^1_0(\Omega)$ such that
\[\begin{cases}
\vrh_m \rightharpoonup \vrh_{\infty} &\text{weakly in } H^1_0(\Omega),\\
\vrh_m \to \vrh_{\infty} &\text{strongly in } L^s(\Omega) \text{ for } s \in (1, 2^*)
\end{cases}
\quad \text{as } m \to \infty,\]
along a subsequence. Given any $\chi\in C_c^{\infty}(\Omega)$, we test \eqref{2.12} with $\chi$ and passing to the limit $m \to \infty$. We can derive from \eqref{ab1} and Lemma \ref{ape1} that
\begin{align*}
\bigg|\int_{\Omega} \left[(u_0+\sigma_m)^{p-1} - u_0^{p-1}\right]\vrh_m\chi\bigg| &\lesssim \int_{\Omega} \left[\sigma_m^{p-1}+u_0^{p-2}\sigma_m\mone_{\{p>2\}}\right] |\vrh_m\chi| \\
&\lesssim \|\sigma_m^{p-1}\|_{L^{\frac{p+1}{p}}(\Omega)} + \|\sigma_m\|_{L^{\frac{p+1}{p}}(\Omega)}\mone_{\{p>2\}} = o_m(1) .
\end{align*}
This fact and \eqref{vrh}--\eqref{2.15} imply that
\[\begin{cases}
(-\Delta-\lambda) \vrh_{\infty} = pu_0^{p-1}\vrh_{\infty} &\text{in } \Omega,\\
\vrh_{\infty}=0 &\text{on } \pa\Omega,
\end{cases}\]
which together with the non-degeneracy of $u_0$ yields $\vrh_{\infty} = 0$ in $\Omega$.

\medskip
Let us now fix an index $j \in \{1,\ldots,\nu\}$, and define the rescaled function
\[\tvrh_{j,m}(y) := \delta_{j,m}^{\frac{n-2}{2}} \vrh_m\big(\delta_{j,m}y+\xi_{j,m}\big) \quad \text{for any } y \in \frac{\Omega-\xi_{j,m}}{\delta_{j,m}}\]
for all sufficiently large $m \in \N$. We extend $\tvrh_{j,m}(y)$ to $\R^n$ by setting it to zero outside its original domain. We will show that
\begin{equation}\label{tvrhjm}
\begin{cases}
\tvrh_{j,m} \rightharpoonup 0 &\text{weakly in } D^{1,2}(\R^n),\\
\tvrh_{j,m} \to 0 &\text{strongly in } L^s_{\loc}(\R^n) \text{ for } s \in (1,2^*)
\end{cases} \quad \text{as } m \to \infty.
\end{equation}
Because $\|\vrh_m\|_{H^1_0(\Omega)}=1$, the sequence $\{\tvrh_{j,m}\}_{n \in \N}$ is uniformly bounded in $D^{1,2}(\R^n)$, and so there exists $\tvrh_{j,\infty} \in D^{1,2}(\R^N)$ such that
\[\begin{cases}
\tvrh_{j,m} \rightharpoonup \tvrh_{j,\infty} &\text{weakly in } D^{1,2}(\R^n),\\
\tvrh_{j,m} \to \tvrh_{j,\infty} &\text{strongly in } L^s_{\loc}(\R^n) \text{ for } s \in (1,2^*)
\end{cases} \quad \text{as } m \to \infty,\]
up to a subsequence. Given a function $\chi \in C^{\infty}_c(\R^n)$, we set
\[\tchi_{j,m}(x) =\delta_{j,m}^{\frac{2-n}{2}} \chi\(\delta_{j,m}^{-1}(x-\xi_{j,m})\) \quad \text{for } x \in \Omega.\]
After testing \eqref{2.12} with $\tchi_{j,m}$, the only technical point we encounter is to derive
\begin{equation}\label{u0sm}
\int_{\Omega} (u_0+\sigma_m)^{p-1}\vrh_m\tchi_{j,m} = \int_{\R^n} U^{p-1}\tvrh_{j,\infty}\chi + o_m(1)
\end{equation}
as $m \to \infty$.

Indeed, direct calculations give us that
\begin{align*}
\int_{\Omega} (PU_{j,m})^{p-1} \vrh_m \tchi_{j,m} &= \int_{\frac{\Omega-\xi_{j,m}}{\delta_{j,m}}} U^{p-1}\tvrh_{j,m}\chi
+ O\bigg(\(\frac{\delta_{j,m}}{d(\xi_{j,m},\pa\Omega)}\)^{\frac{n-2}{n}}\bigg) \\
&= \int_{\R^n} U^{p-1}\tvrh_{j,\infty}\chi + o_m(1),
\end{align*}
because
\begin{equation}\label{pup}
\begin{aligned}
\int_{\Omega} \left|(PU_{j,m})^{p-1}-U_{j,m}^{p-1}\right|^{\frac{p+1}{p-1}} &\lesssim \left\||PU_{j,m}-U_{j,m}| U_{j,m}^{p-2} \mone_{\{p>2\}}\right\|^{\frac{p+1}{p-1}}_{L^{\frac{p+1}{p-1}}(B(\xi_{j,m},d(\xi_{j,m},\pa\Omega))} \\
&\ + \| |PU_{j,m}-U_{j,m}|^{p+1}\|_{L^1(B(\xi_{j,m},d(\xi_{j,m},\pa\Omega))} \\
&\ + \int_{B(\xi_{j,m},d(\xi_{j,m},\pa\Omega))^c} U_{j,m}^{p+1}
\lesssim \(\frac{\delta_{j,m}}{d(\xi_{j,m},\pa\Omega)}\)^{\frac{n-2}{2}},
\end{aligned}
\end{equation}
while we know
\[\int_{\Omega}u_0^{p-1}\vrh_m\tchi_{j,m} \simeq \delta_{j,m}^2 \int_{\supp(\chi)} u_0^{p-1}(\xi_{j,m}+\delta_{j,m}y) (\tvrh_{j,m}{\chi})(y) dy = o_m(1)\]
thanks to the boundedness of $u_0$. Furthermore, for $1 \le i \ne j \le \nu$,
\[\left|\int_{\Omega} (PU_{i,m})^{p-1}\vrh_m\tchi_{j,m} \right| \lesssim \left\|\left[\delta_{j,m}^{\frac{n-2}{2}} U_{i,m}\(\xi_{j,m}+\delta_{j,m}\cdot\)\right]^{p-1}\right\|_{L^{\frac{p+1}{p}}(\supp(\chi))} = o_m(1),\]
since
\begin{align*}
&\ \(\frac{\delta_{j,m}}{\delta_{i,m}}\)^{\frac{4n}{n+2}} \int_{\supp(\chi)} \frac{dy}{\big(1+\big(\frac{\delta_{j,m}}{\delta_{i,m}}|y - z_{ij.m}|\big)^2\big)^{\frac{4n}{n+2}}} \\
&\lesssim \begin{cases}
\displaystyle \(\frac{\delta_{j,m}}{\delta_{i,m}}\)^{-\frac{4n}{n+2}} |z_{ij,m}|^{-\frac{8n}{n+2}} &\text{if } |z_{ij,m}|\to \infty, \\
\displaystyle \(\frac{\delta_{j,m}}{\delta_{i,m}}\)^{\frac{4n}{n+2}-n} \int_{|z|\le \frac{\delta_{j,m}}{\delta_{i,m}}} \frac{1}{(1+|z|)^{\frac{8n}{n+2}}} dz &\text{if } |z_{ij,m}| \text{ is bounded}, \delta_{i,m}\ll \delta_{j,m}, \\
\displaystyle \(\frac{\delta_{j,m}}{\delta_{i,m}}\)^{\frac{4n}{n+2}} &\text{if } |z_{ij,m}|\text{ is bounded}, \delta_{i,m}\gg \delta_{j,m}
\end{cases}\\
&\lesssim \begin{medsize}
\begin{cases}
\displaystyle \msr_{ij,m}^{-\frac{8n}{n+2}} &\text{if } |z_{ij,m}|\to \infty, \\
\displaystyle \(\frac{\delta_{j,m}}{\delta_{i,m}}\)^{\frac{4n}{n+2}-n} \left[\mone_{\{p>2\}} + \left|\log \frac{\delta_{j,m}}{\delta_{i,m}}\right| \mone_{\{p=2\}} + \(\frac{\delta_{j,m}}{\delta_{i,m}}\)^{n-\frac{8n}{n+2}} \mone_{\{p<2\}} \right] &\text{if } |z_{ij,m}| \text{ is bounded}, \delta_{i,m}\ll \delta_{j,m}, \\
\displaystyle \(\frac{\delta_{j,m}}{\delta_{i,m}}\)^{\frac{4n}{n+2}} &\text{if } |z_{ij,m}| \text{ is bounded}, \delta_{i,m}\gg \delta_{j,m},
\end{cases}
\end{medsize}
\end{align*}
where $\msr_{ij,m}$ is the quantity introduced in \eqref{rq} with $(\xi_i,\xi_j,\delta_i,\delta_j)$ replaced by $(\xi_{i,m},\xi_{j,m},\delta_{i,m},\delta_{j,m})$.
By \eqref{iqu}, and Lemmas \ref{a4} and \ref{a22}, we also have that
\begin{align*}
&\ \int_{\Omega}(PU_{j,m})^{p-2} \bigg(u_0+\sum_{i\ne j}PU_{i,m}\bigg)|\vrh_m\tchi_{j,m} | \mone_{\{p>2\}} \\
&\lesssim \bigg[\sum_{i\ne j}\|U_{i,m} U_{j,m}^{p-2}\|_{L^{\frac{p+1}{p-1}}(\Omega)} + \max_i\|U_{i,m}^{p-2}\|_{L^{\frac{p+1}{p-1}}(\Omega)}\bigg] \mone_{\{p>2\}} = o_m(1).
\end{align*}
Combining the above calculations, we derive \eqref{u0sm}.

Taking $m \to \infty$, we observe from \eqref{2.12} that
\[\begin{cases}
\displaystyle -\Delta \tvrh_{j,\infty} = pU^{p-1}\tvrh_{j,\infty} \quad \text{in } \R^n, \quad \tvrh_{j,\infty} \in D^{1,2}(\R^n), \\
\displaystyle \int_{\R^n}\nabla \tvrh_{j,\infty} \cdot\nabla Z^k = 0 \quad \text{for all } k=0,\ldots,n.
\end{cases}\]
The nondegeneracy of $U$ implies that $\tvrh_{j,\infty}=0$, yielding \eqref{tvrhjm}.

\medskip
Finally, we will prove
\begin{equation}\label{vrhm}
\lim_{m \to \infty} \|\vrh_m\|_{H^1_0(\Omega)}=0.
\end{equation}
Since \eqref{vrhm} contradicts \eqref{vrh}, we will be able to conclude that \eqref{vrh1} must hold.

To deduce \eqref{vrhm}, we test \eqref{2.12} with $\vrh_m$. Then, we only have to consider
\begin{align*}
\int_{\Omega} (u_0+\sigma_m)^{p-1}\vrh_m^2 &\lesssim \int_{\Omega} u_0^{p-1}\vrh_m^2 + \sum_{i=1}^{\nu} \int_{\Omega} (PU_{i,m})^{p-1}\vrh_m^2 \\
&\lesssim o_m(1)+ \int_{\R^n} U^{p-1}\tvrh_{i,m}^2 + O\(\max_i\(\frac{\delta_{i,m}}{d(\xi_{i,m},\pa\Omega)}\)^{\frac{n-2}{n}}\) \|\vrh_m\|_{H^1_0(\Omega)}^2 \\
&= o_m(1).
\end{align*}
Here, we employed \eqref{pup} and the facts that $\vrh_m\to 0$ strongly in $L^2(\Omega)$ and $\tvrh_{i,m}^2 \rightharpoonup 0$ weakly in $L^{\frac{n}{n-2}}(\R^n)$. We are done.
\end{proof}

\begin{proof}[Proof of Proposition \ref{prop:34}]
We set
\begin{align*}
h_1 &:= \left[(u_0+\sigma+\rho_0+\rho_1)^{p} - |u_0+\sigma+\rho_0|^{p-1}(u_0+\sigma+\rho_0)-p|u_0+\sigma+\rho_0|^{p-1}\rho_1\right] \\
&\ + p\left[|u_0+\sigma+\rho_0|^{p-1}-(u_0+\sigma)^{p-1}\right]\rho_1\\
&\ + \begin{cases}
\displaystyle f+\I_1+\I_2+\I_3 &\text{if } n \ne 6,\\
\displaystyle f+(\I_3-\I_{31})-\sum_{i=1}^{\nu} \sum_{k=0}^n c_i^k (-\Delta-\lambda) PZ_i^k &\text{if } n = 6.
\end{cases}
\end{align*}

From \eqref{eq:34}, we have
\[\begin{cases}
\begin{aligned}
&\rho_1 -\Pi^{\perp}[(-\Delta-\lambda)^{-1}(p(u_0+\sigma)^{p-1}\rho_1)]=\Pi^{\perp}[(-\Delta-\lambda)^{-1}h_1] \quad \text{in } \Omega,
\end{aligned}\\
\displaystyle \rho_1=0 \quad \text{on } \pa\Omega, \\
\displaystyle \big\langle \rho_1, PZ_i^k \big\rangle_{H^1_0(\Omega)} = 0 \quad \text{for } i = 1,\ldots,\nu \text{ and } k = 0,\ldots,n.
\end{cases}
\]
By making use of \eqref{vrh1}, \eqref{ab1}, \eqref{ab6} and H\"older's inequality
\begin{align*}
&\ \|\rho_1\|_{H^1_0(\Omega)} \\
&\lesssim \|f\|_{(H^1_0(\Omega))^*} + \|\rho_1\|_{H^1_0(\Omega)}^2\mone_{\{p>2\}}+\|\rho_1\|_{H^1_0(\Omega)}^p + \(\|\rho_0\|_{H^1_0(\Omega)}\mone_{\{p>2\}}+ \|\rho_0\|_{H^1_0(\Omega)}^{p-1}\) \|\rho_1\|_{H^1_0(\Omega)} \\
&\ + \begin{cases}
\displaystyle \|\I_1\|_{L^{\frac{p+1}{p}}(\Omega)}+\|\I_2\|_{L^{\frac{p+1}{p}}(\Omega)}+\|\I_3\|_{L^{\frac{p+1}{p}}(\Omega)} &\text{if } n \ne 6, \\
\displaystyle \|\I_3-\I_{31}\|_{L^{\frac{p+1}{p}}(\Omega)} + \sum_{i=1}^{\nu} \sum_{k=0}^n |c_i^k| &\text{if } n = 6.
\end{cases}
\end{align*}
Since $p>1$ and $\|\rho_0\|_{H^1_0(\Omega)}=o_{\ep_1}(1)$, we immediately deduce \eqref{eq:341}.
\end{proof}

\begin{cor}
For each $i = 1,\ldots,\nu$, we assume that $PU_i$ satisfies \eqref{pu1} if $n \ge 5$ or \textup{[}$n = 3, 4$ and $u_0>0$\textup{]}, and satisfies \eqref{pu2} if $n=3,4$ and $u_0=0$. We define
\[\mcj_{11}(\delta_1,\ldots,\delta_{\nu}) := \begin{cases}
\displaystyle \max_i\delta_i &\text{if } n=3 \text{ and } u_0=0, \\
\displaystyle \max_i\delta_i^2|\log\delta_i| &\text{if } n=4 \text{ and } u_0=0, \\
\displaystyle \max_i\delta_i^{\frac{n-2}{2}} &\text{if } [n=3,4 \text{ and } u_0>0] \text{ or } n=5, \\
\displaystyle \max_i\delta_i^2|\log\delta_i|^{\frac12} &\text{if } n=6, \\
\displaystyle \max_i\delta_i^2 &\text{if } n\ge 7,
\end{cases}\]
\[\mcj_{12}(\ka_1,\ldots,\ka_{\nu}) := \begin{cases}
\displaystyle \max_i \ka_i^{n-2} &\text{if } n=3,4,5, \\
\displaystyle \ka_1^4 \left|\log\ka_1\right|^{\frac12} &\text{if } n=6 \text{ and } \nu=1, \\
\displaystyle \max_i \ka_i^4 \left|\log\ka_i\right|^{\frac23} &\text{if } n=6 \text{ and } \nu\ge 2, \ \footnotemark \\
\displaystyle \max_i \ka_i^{\frac{n+2}{2}} &\text{if } n\ge 7,
\end{cases}\]
\footnotetext{The bound for $n=6$ and $\nu\ge 2$ may not be optimal. We present it here for the sake of completeness.} where $\ka_i = \frac{\delta_i}{d(\xi_i,\pa\Omega)}$, and
\[\mcj_{13}(Q) := \left.\begin{cases}
Q &\text{if } n=3,4,5\\Q|\log Q|^{\frac12} &\text{if } n=6\\
Q^{\frac{n+2}{2(n-2)}} &\text{if } n\ge 7
\end{cases}\right\}\mone_{\{\nu\ge 2\}}.\]
Then
\begin{equation}\label{sirp}
\|\rho\|_{H^1_0(\Omega)} \lesssim \|f\|_{(H^1_0(\Omega))^*} + \mcj_{11}(\delta_1,\ldots,\delta_{\nu}) + \mcj_{12}(\ka_1,\ldots,\ka_{\nu}) + \mcj_{13}(Q).
\end{equation}
\end{cor}
\begin{proof}
The result is a consequence of \eqref{eq:331} and \eqref{eq:341}.
\end{proof}
\begin{prop}
For each $i = 1,\ldots,\nu$, we assume that $PU_i$ satisfies \eqref{pu1} if $n \ge 5$ or \textup{[}$n = 3, 4$ and $u_0>0$\textup{]}, and satisfies \eqref{pu2} if $n=3,4$ and $u_0=0$. We set
\[\mcj_{21}(\delta_1,\ldots,\delta_{\nu}) := \begin{cases}
\displaystyle \max_i\delta_i &\text{if } n=3 \text{ and } u_0=0, \\
\displaystyle \max_i\delta_i^2|\log \delta_i| &\text{if } n=4 \text{ and } u_0=0, \\
\displaystyle \max_i\delta_i^{\frac{n-2}{2}} &\text{if } n=3,4,5 \text{ and } u_0>0,  \\
\displaystyle \max_i\delta_i^2 &\text{if } [n=5 \text{ and } u_0=0] \text{ or } n\ge 6,
\end{cases}\]
and $\mcj_{23}(Q) := Q\mone_{\{\nu\ge 2\}}$. If each $\xi_1,\ldots,\xi_{\nu}$ lies on a compact set of $\Omega$, then it holds that
\begin{equation}\label{qj2}
\mcj_{21}(\delta_1,\ldots,\delta_{\nu}) + \mcj_{23}(Q) \lesssim \|f\|_{(H^1_0(\Omega))^*}.
\end{equation}
\end{prop}
\begin{proof}
Let $j \in \{1,\ldots,\nu\}$ be fixed. By testing \eqref{eqrho2} with $ PZ_j^0$, we obtain
\begin{align*}
\int_{\Omega} \I_1 PZ_j^0 + \int_{\Omega} \I_2 PZ_j^0 + \int_{\Omega} \I_3 PZ_j^0 &= -\int_{\Omega} f PZ_j^0 -\int_{\Omega} \I_0[\rho] PZ_j^0 \\
&\ + \int_{\Omega} \big[(-\Delta-\lambda)\rho - p(u_0+\sigma)^{p-1}\rho\big] PZ_j^0.
\end{align*}

As in \eqref{liar}, we apply Lemmas \ref{a4} and \ref{a22}, and the assumption that $\xi_i$ lies on a compact set of $\Omega$ for $i=1,\ldots,\nu$ to deduce
\begin{align*}
&\ \bigg|\int_{\Omega} \big[(-\Delta-\lambda)\rho - p(u_0+\sigma)^{p-1}\rho\big] PZ_j^0 \bigg| \\
&\lesssim \|\rho\|_{H^1_0(\Omega)} \left[\|U_j\|_{L^{\frac{p+1}{p}}(\Omega)}\mone_{\{u_0>0\}\cup\{PU_{j} \text{ satisfies } \eqref{pu1}\}} + \sum_{i=1}^{\nu}\|U_i^{p-1}\|_{L^{\frac{p+1}{p}}(\Omega)}\mone_{\{u_0>0, p>2\}} \right. \\
&\hspace{60pt} \left. +\left.\begin{cases}
\displaystyle \sum_{i\ne j} \|U_{i}^{p-1}U_{j}\|_{L^{\frac{p+1}{p}}(\Omega)} &\text{if } n=3,4,5\\
\displaystyle \sum_{i\ne j} \left\|\min\{U_{i}^{p-1}U_{j}, U_{j}^{p-1}U_{i}\}\right\|_{L^{\frac{p+1}{p}}(\Omega)} &\text{if } n\ge 6
\end{cases}\right\} \mone_{\{\nu\ge 2\}} \right] \\
&= o(\mcj_{21}(\delta_1,\ldots,\delta_{\nu}) + \mcj_{23}(Q)).
\end{align*}
Using \eqref{ape1} and the fact that $|PZ_j^0|\le \sum_{i=1}^{\nu}U_i$, we also know that
\begin{equation}\label{i4rp}
\begin{aligned}
\bigg|\int_{\Omega} \I_0[\rho] PZ_j^0\bigg| &\lesssim \begin{cases}
\displaystyle \int_{\Omega} \min\{\sigma^{p-2}\rho^2, |\rho|^{p}\} |PZ_j^0| &\text{if } 1<p<2, \\
\displaystyle \int_{\Omega} \(\sigma^{p-2}\rho^2+|\rho|^p\) |PZ_j^0| &\text{if }p\ge 2
\end{cases} \\
&\lesssim \int_{\Omega} \sum_{i=1}^{\nu}U_i^{p-1}|\rho|^2 +\|\rho\|_{H^1_0(\Omega)}^{p}\mone_{\{p>2\}} \lesssim \|\rho\|^2_{H^1_0(\Omega)}.
\end{aligned}
\end{equation}

Without loss of generality, one may assume that $\delta_1 \ge \delta_2 \ge \cdots \ge \delta_{\nu}$.
By employing Lemmas \ref{lem2.31}--\ref{lem2.33}  together with $\mfa_n, \mfb_n > 0$, $d(\xi_i, \pa\Omega)\gtrsim 1$, and $-\vph_{\lambda}^3(\xi_i)>0$ provided $n=3$, $u_0=0$, and $\nu\ge 2$, noticing that 
\[\mfd_nq_{ij}+\left.\begin{cases}
\displaystyle  \sum_{i\ne j} [-\mfb_3\lambda|\xi_j-\xi_i|-\mfc_3H_{\lambda}^3(\xi_i, \xi_j)]\delta_i^{\frac12}\delta_j^{\frac12} &\text{if } n=3, u_0=0\\
\displaystyle \sum_{i\ne j} [-\mfb_4\lambda\log |\xi_j-\xi_i| - \mfc_4H_{\lambda}^4(\xi_i, \xi_j)]\delta_i\delta_j  &\text{if } n=4,u_0=0
\end{cases}\right\}\simeq q_{ij},\]
we adopt the same reasoning as in \cite[Lemma 2.3]{DSW} (which is based on mathematical induction) to achieve
\begin{equation}\label{fi1}
\mcj_{23}(Q) \lesssim \|f\|_{(H^1_0(\Omega))^*} + o(\mcj_{21}(\delta_1,\ldots,\delta_{\nu})).
\end{equation}
Then, one may take the test function $PZ_1^0$, where $\delta_1=\max_i\delta_i$, to prove
\begin{equation}\label{madl}
\begin{aligned}
%\begin{medsize}
%\displaystyle
\mcj_{21}(\delta_1,\ldots,\delta_{\nu})
%\left.\begin{cases}
%\max_i\delta_i &\text{if } n=3 \text{ and }u_0=0\\
%\max_i\delta_i^2|\log \delta_i| &\text{if } n=4 \text{ and }u_0=0\\
%\max_i\delta_i^{\frac{n-2}{2}} &\text{if } n=3,4,5 \text{ and } u_0>0 \\
%\max_i\delta_i^2 &\text{if } [n=5 \text{ and } u_0=0] \text{ or } n\ge 6
%\end{cases}\right\}
%\end{medsize}
%&\begin{medsize}
&\lesssim \|f\|_{(H^1_0(\Omega))^*}+\bigg|\int_{\Omega} \I_2 PZ_1^0\bigg| + o(\mcj_{21}(\delta_1,\ldots,\delta_{\nu}) + \mcj_{23}(Q))\\
%\end{medsize} \\
%&\begin{medsize}
&\lesssim \|f\|_{(H^1_0(\Omega))^*} + o(\mcj_{21}(\delta_1,\ldots,\delta_{\nu}) + \mcj_{23}(Q)).
%\end{medsize}
\end{aligned}
\end{equation}
Here, we used $\big|\int_{\Omega} \I_2 PZ_1^0\big|\lesssim Q$, which comes from \eqref{i3pz0} and Lemma \ref{a22}.

Putting \eqref{fi1} and \eqref{madl}, we establish \eqref{qj2}, concluding the proof.
\end{proof}

We are now in a position to establish estimate \eqref{eq:sqe32}.
\begin{proof}[Proof of Estimate \eqref{eq:sqe32}]
Since $d(\xi_i, \pa\Omega)\gtrsim 1$, we have
\[\mcj_{12}(\ka_1,\ldots,\ka_{\nu}) \lesssim \mcj_{11}(\delta_1,\ldots,\delta_{\nu}),\]
where $\ka_i = \frac{\delta_i}{d(\xi_i,\pa\Omega)}$. From \eqref{sirp} and \eqref{qj2}, one can identify two optimal functions $\tilde{\zeta}_1(t)$ and $\tilde{\zeta}_3(t)$ of the form $t^a|\log t|^b$, with $a>0$ and $b \ge 0$ ($b = 0$ unless $n=6$), such that
\[\mcj_{11}(\delta_1,\ldots,\delta_{\nu}) \lesssim \tilde{\zeta}_1(\mcj_{21}(\delta_1,\ldots,\delta_{\nu})) \quad \text{and} \quad \mcj_{13}(Q) \lesssim \tilde{\zeta}_3(\mcj_{23}(Q)).\]
Recognizing that $\tilde{\zeta}_1(t)$ and $\tilde{\zeta}_3(t)$ are non-decreasing for $t>0$, we obtain
\[\|\rho\|_{H^1_0(\Omega)} \lesssim \max\left\{\|f\|_{(H^1_0(\Omega))^*},\, \tilde{\zeta}_1(\|f\|_{(H^1_0(\Omega))^*}),\, \tilde{\zeta}_3(\|f\|_{(H^1_0(\Omega))^*})\right\} = \zeta(\|f\|_{(H^1_0(\Omega))^*}),\]
where $\zeta(t)$ is the function introduced in \eqref{eq:zeta}.
\end{proof}

\subsection{Sharpness of estimate \eqref{eq:sqe32}}\label{subs3.2}
Let us divide it into two cases.

\medskip \noindent \fbox{Case 1:} We prove the optimality of \eqref{eq:sqe32} when [$n=3,4,\; \nu\ge 1$], or [$n=5,\; \nu\ge 1,\; u_0>0$] or [$n\ge 7,\; \nu=1$]. In this case, we have that $\zeta(t) = t$.

\medskip
We select numbers $\delta = \delta_i \in (0,1)$ for each $i \in \{1, \ldots, \nu\}$ and points $\xi_i \in \Omega$ such that $d(\xi_i,\pa\Omega) \gtrsim 1$ and $|\xi_i - \xi_j| \gtrsim 1$ for all distinct indices $1 \le i \ne j \le \nu$.
Under these conditions, it holds that $Q \simeq \delta^{n-2} \cdot \mone_{{\nu \ge 2}}$.

Taking
\[\ep \simeq \begin{cases}
\delta &\text{if } n=3 \text{ and } u_0=0,\\
\delta^2|\log\delta| &\text{if } n=4 \text{ and } u_0=0,\\
\delta^{\frac{n-2}{2}} &\text{if } n=3,4,5 \text{ and } u_0>0,\\
\delta^2 &\text{if } n\ge 7 \text{ and } \nu=1,
\end{cases}\]
and using $|PZ_i^k|\le CPU_i$ in $\Omega$, we construct a nonnegative function of the form
\[\phi_\delta=\sum_{i=1}^{\nu}PU_i + \sum_{i=1}^{\nu}\sum_{k=0}^n \beta_i^kPZ_i^k,\]
where $\beta_i^k=o_{\delta}(1)$, $\la\phi_\delta,PZ_i^k\ra=0$, and $\|\phi_\delta\|_{H^1_0(\Omega)}\simeq 1$.

Letting $\rho:=\ep\phi_\delta$, we define $u_*:=u_0+\sum_{i=1}^{\nu}PU_i+\rho$ so that $u_*=0$ on $\pa\Omega$. Then we set
\[f:=-\Delta u_*-\lambda u_*-u_*^{p-1}=-\Delta\rho-\lambda \rho-p\(u_0+\sum_{i=1}^{\nu}PU_i\)^{p-1}\rho+\I_1+\I_2+\I_3+\I_0[\rho]\]
where $\I_1$, $\I_2$, $\I_3$, and $\I_0[\rho]$ are defined as in \eqref{eq:I2} with parameters $(\delta_i,\xi_i)$ satisfying the above conditions.
By Lemmas \ref{lem2.1} and \ref{lem2.2}, we have that $\|\rho\|_{H^1_0(\Omega)}\simeq \ep$ and
\begin{align*}
\|f\|_{(H^1_0(\Omega))^*} &\lesssim \|\rho\|_{H^1_0(\Omega)} + \|\rho\|_{H^1_0(\Omega)}^{\min\{2,p\}} + \|\I_1\|_{L^{\frac{p+1}{p}}(\Omega)}\mone_{\{u_0\ne 0\}} + \|\I_2\|_{L^{\frac{p+1}{p}}(\Omega)}\mone_{\{\nu\ge 2\}} + \|\I_3\|_{L^{\frac{p+1}{p}}(\Omega)} \\
&\simeq \ep\simeq \|\rho\|_{H^1_0(\Omega)}.
\end{align*}

Proceeding as in Step 2 of \cite[Subsection 5.1]{CK}, we deduce that
\[\inf_{\substack{(\tilde{\delta}_i,\tilde{\xi}_i)\in (0,1)\times\Omega,\\ i=1,\ldots,\nu}} \Big\|u_*-\Big(u_0+\sum_{i=1}^{\nu}PU_{\tilde{\delta}_i,\tilde{\xi}_i}\Big)\Big\|_{H^1_0(\Omega)} \gtrsim \|\rho\|_{H^1_0(\Omega)},\]
thereby establishing the optimality of \eqref{eq:sqe32}.

\medskip \noindent \fbox{Case 2:} We prove the optimality of \eqref{eq:sqe32} when $[n=5,\; \nu\ge 1,\; u_0=0]$ or $[n=6,\; \nu\ge 1]$ or $[n\ge 7,\; \nu\ge 2]$. In this case, we have that $\zeta(t) \gg t$.
The proof is split into three steps.

\medskip \noindent \doublebox{\textsc{Step 1.}}
We select $\delta = \delta_i \in (0,1)$ and $\xi_i \in \Omega$ such that $d(\xi_i,\pa\Omega) \gtrsim 1$ and $|\xi_i - \xi_j| \simeq \delta^b$ for each $i\ne j$, where $i,j \in \{1, \ldots, \nu\}$ and $b \in [0,1)$.
This choice ensures that $Q \simeq \delta^{(1-b)(n-2)}$. We impose a further restriction $b \in (\frac{n-4}{n-2}, 1)$ for $n \ge 7$, and set $b=0$ in dimensions $n=5,6$.

We now consider the function $\rho$ solving the boundary value problem
\begin{equation}\label{opp}
\begin{cases}
\displaystyle -\Delta \rho-\lambda \rho - p(u_0+\sigma)^{p-1}\rho = \I_1+\I_2+\I_3+\I_0[\rho] + \sum_{i=1}^{\nu}\sum_{k=0}^n \tc_i^k(-\Delta-\lambda)PZ_i^k \quad \text{in } \Omega,\\
\displaystyle \rho=0 \quad \text{on } \pa\Omega,\quad \tc_i^k\in \R \quad \text{for } i=1,\ldots,\nu \text{ and } k=0,\ldots,n,\\
\displaystyle \big\langle \rho, PZ_i^k \big\rangle_{H^1_0(\Omega)} = 0 \quad \text{for } i=1,\ldots,\nu \text{ and } k=0,\ldots,n,
\end{cases}
\end{equation}
where $PZ_i^k$, $\I_1$, $\I_2$, $\I_3$, and $\I_0[\rho]$ are defined as in \eqref{eq:I2} with parameters $(\delta_i,\xi_i)$ satisfying the above conditions. We set $f:=\sum_{k=0}^n\sum_{i=1}^{\nu} \tc_i^k(-\Delta-\lambda)PZ_i^k$. Then
\begin{equation}\label{tcik}
\begin{aligned}
\|f\|_{(H^1_0(\Omega))^*} &\lesssim \sum_{k=0}^n\sum_{i=1}^{\nu}|\tc_i^k| \\
&\lesssim \vsi_1(\delta) :=\begin{cases}
\delta^2 &\text{if } [n=5,\; u_0=0,\; \nu\ge 1] \text{ or } [n=6,\; \nu\ge 1],\\
\delta^{(1-b)(n-2)} &\text{if } n\ge 7 \text{ and } \nu\ge 2.
\end{cases}
\end{aligned}
\end{equation}
By applying Lemma \ref{lem2.2} and \eqref{tcik}, we see that
\begin{equation}\label{apr1}
\begin{aligned}
\|\rho\|_{H^1_0(\Omega)} &\lesssim
\|f\|_{(H^1_0(\Omega))^*} + \|\I_1\|_{L^{\frac{p+1}{p}}(\Omega)} + \|\I_2\|_{L^{\frac{p+1}{p}}(\Omega)} + \|\I_3\|_{L^{\frac{p+1}{p}}(\Omega)} \\
&\lesssim \begin{cases}
\delta^{\frac32} &\text{if } n=5, u_0=0, \text{ and } \nu\ge 1,\\
\delta^2|\log\delta|^{\frac23} &\text{if } n=6 \text{ and } \nu\ge 1,\\
\delta^{\frac{(1-b)(n+2)}{2}} &\text{if } n\ge 7 \text{ and } \nu\ge 2.
\end{cases}
\end{aligned}
\end{equation}

We now decompose $\rho = \trh_0 + \trh_1$, where the functions $\trh_0$ and $\trh_1$ satisfy
\[\begin{cases}
\begin{aligned}
&\displaystyle -\Delta \trh_0-\lambda \trh_0 - p(u_0+\sigma)^{p-1}\trh_0 = \I_1\mone_{\{n\ge 6,\; u_0>0\}}+\I_2\mone_{\{n\ge 6,\; \nu\ge 2\}}\\
&\displaystyle \hspace{100pt} -\sum_{i=1}^{\nu}\lambda PU_i\mone_{\{n=5,6\}} + \I_0[\trh_0] + \sum_{i=1}^{\nu}\sum_{k=0}^n \bar{c}_i^k(-\Delta-\lambda)PZ_i^k
\end{aligned} \quad \text{in } \Omega,\\
\displaystyle \trh_0=0 \quad \text{on } \pa\Omega,\quad \bar{c}_i^k\in \R\quad \text{for } i=1,\ldots,\nu \text{ and } k=0,\ldots,n,\\
\displaystyle \big\langle \trh_0, PZ_i^k \big\rangle_{H^1_0(\Omega)} = 0 \quad \text{for } i=1,\ldots,\nu \text{ and } k=0,\ldots,n,
\end{cases}\]
and
\[\begin{cases}
\begin{aligned}
&\displaystyle -\Delta \trh_1-\lambda \trh_1 - p(u_0+\sigma)^{p-1}\trh_1 = \I_2\mone_{\{n=5,\; u_0=0,\; \nu\ge 2\}}+\sum_{i=1}^{\nu}(\Delta PU_i+PU_i^p)\\
&\displaystyle \hspace{75pt} +\sum_{i=1}^{\nu}\lambda PU_i\mone_{\{n\ge 7\}}
+\I_0[\rho]-\I_0[\trh_0] + \sum_{i=1}^{\nu}\sum_{k=0}^n (\tc_i^k+\bar{c}_i^k) (-\Delta-\lambda)PZ_i^k
\end{aligned} \quad \text{in } \Omega,\\
\displaystyle \trh_1=0 \quad \text{on } \pa\Omega,\\
\displaystyle \big\langle \trh_1, PZ_i^k \big\rangle_{H^1_0(\Omega)} = 0 \quad \text{for } i=1,\ldots,\nu \text{ and } k=0,\ldots,n,
\end{cases}\]
respectively. By realizing $|\I_1+\I_2-\sum_{i=1}^{\nu}\lambda PU_i|\lesssim \sum_{i=1}^{\nu}U_i$ for $n=6$ (since $|\xi_i-\xi_j|\gtrsim 1$) and recalling \eqref{i1po}, one can deduce a coefficient bound
\begin{equation} \label{bcik}
\sum_{k=0}^n \sum_{i=1}^{\nu} |\bar{c}_i^k| \lesssim \vsi_1(\delta)
\end{equation}
and a pointwise estimate for $\trh_0$:
\begin{equation}\label{aptrp}
|\trh_0|(x) \lesssim \wtw(x).
\end{equation}
Here,
\begin{align*}
\tw_{1i}^{\tin}(x) &:= \frac{1}{\la x_i\ra^2} \bs{1}_{\{|x_i|\le {\delta}^{-\frac12}\}}, & \tw_{1i}^{\tout}(x) &:= \frac{\delta^{-\frac{n-6}{2}}}{\la x_i \ra^{n-4}} \bs{1}_{\{|x_i|\ge {\delta}^{-\frac12}\}}, \\
\tw_{2i}^{\tin}(x) &:= \frac{\delta^{(\frac12-b)(n-2)}}{\la x_i \ra^2} \mone_{\big\{|x_i| < \min\limits_{i\ne j}\frac{|\xi_i-\xi_j|}{2\delta}\big\}}, &
\tw_{2i}^{\tout}(x) &:= \frac{\delta^{4(1-b)-\frac{n-2}{2}}}{|x_i|^{n-4}} \mone_{\big\{|x_i| \ge \min\limits_{i\ne j}\frac{|\xi_i-\xi_j|}{2\delta}\big\}}, \\
\tw_{3i}(x) &:= \frac{\delta^{-\frac{n-6}{2}}}{\la x_i \ra^{n-4}},
\end{align*}
and
\[\wtw(x) := \sum_{i=1}^{\nu} \left[\tw_{1i}^{\tin} + \tw_{1i}^{\tout}\right](x) \mone_{\{n\ge 7,\; u_0>0\}}  + \sum_{i=1}^{\nu} \left[\tw_{2i}^{\tin} + \tw_{2i}^{\tout}\right](x) \mone_{\{n\ge 7,\; \nu\ge2\}} + \sum_{i=1}^{\nu} \tw_{3i}(x) \mone_{\{n=5,6\}}.\]
Moreover, we observe
\begin{equation}\label{aptrp1}
\begin{aligned}
\|\trh_1\|_{H^1_0(\Omega)} &\lesssim \left\| \I_2 \bs{1}_{\{n=5,\; u_0=0,\; \nu\ge 2\}} + \sum_{i=1}^{\nu} \(\Delta PU_i + PU_i^p -\lambda PU_i \bs{1}_{\{n\ge 7\}} \) \right\|_{L^{\frac{p+1}{p}}(\Omega)} \\
&\lesssim \begin{cases}
\delta^3 & \text{if } n=5,\; u_0=0,\; \nu\ge 2,\\
\delta^4 |\log\delta|^{\frac{2}{3}} & \text{if } n=6,\; \nu\ge 1,\\
\delta^2 & \text{if } n \ge 7,\; \nu \ge 2.
\end{cases}
\end{aligned}
\end{equation}
Combining these computations and adapting the approach of Proposition \ref{prop:34}, we reach the improved estimate
\begin{align}\label{exrh}
\|\rho\|_{H^1_0(\Omega)} \lesssim \vsi_2(\delta) := \begin{cases} \delta^{\frac{3}{2}} & \text{if } n=5,\; u_0=0, \nu\ge 1,\\
\delta^2 |\log\delta|^{\frac{1}{2}} & \text{if } n=6,\; \nu\ge 1,\nonumber\\
\delta^{\frac{(1-b)(n+2)}{2}} & \text{if } n \ge 7,\; \nu \ge 2.
\end{cases}
\end{align}

\medskip \noindent \doublebox{\textsc{Step 2.}} We now establish the lower bound
\begin{equation}\label{inrp}
\|\rho\|_{H^1_0(\Omega)}\gtrsim \vsi_2(\delta),
\end{equation}
which in turn implies
\[\|\rho\|_{H^1_0(\Omega)}\gtrsim \zeta(\|f\|_{H^{-1}(\Omega)}).\]

\medskip
Testing equation \eqref{opp} against $\rho$ and applying Holder's inequality yield
\begin{align*}
\|\rho\|^2_{H^1_0(\Omega)} &=\int_{\Omega}p(u_0+\sigma)^{p-1}\rho^2 +\int_{\Omega}(\I_1+\I_2+\I_3+\I_0[\rho]){\rho} \\
&\ge \int_{\Omega} \(\I_1\mone_{\{n\ge 6,\; u_0>0\}} + \I_2\mone_{\{n\ge 6,\; \nu\ge 2\}} + \sum_{i=1}^{\nu}\lambda PU_i\mone_{\{n=5,6\}}\) \trh_0 + o\(\vsi_2(\delta)^2\) \\
&=: J_2 + o\(\vsi_2(\delta)^2\),
\end{align*}
where we have invoked \eqref{aptrp}, \eqref{aptrp1}, \eqref{apr1}, \eqref{i12e}, \eqref{i3e}, and the bound $|\I_0[\rho]\rho|\lesssim|\rho|^{\min\{p+1, 3\}}$.

Let $G_\lambda$ be defined as \eqref{gla} for $n\ge 3$. We recall the integral representation of $\trh_0$ given by
\begin{align*}
\trh_0(x)=\int_{\Omega}G_\lambda(x,y) &\bigg[p(u_0+\sigma)^{p-1}\trh_0+ \I_1\mone_{\{n\ge 6,\; u_0>0\}}+\I_2\mone_{\{n\ge 6,\; \nu\ge 2\}} \\
&+\sum_{i=1}^{\nu}\lambda PU_i\mone_{\{n=5,6\}}+\I_0[\trh_0] +\sum_{i=1}^{\nu}\sum_{k=0}^n \bar{c}_i^k(-\Delta-\lambda)PZ_i^k\bigg]
\end{align*}
and the lower bound estimate of $G_{\lambda}$:
\[G_\lambda(x,y)\gtrsim \frac{1}{|x-y|^{n-2}}.\]
We also introduce the quantities
\[J_{21} := \begin{cases}
\displaystyle \int_{\Omega} \sum_{i=1}^{\nu}\lambda PU_i(x)\int_{\Omega}G_\lambda(x,\om) \sum_{j=1}^{\nu}\lambda PU_j(\om) dx d\om &\text{if } n=5,6,\\
\displaystyle \int_{\Omega} \I_2(x)\int_{\Omega}G_\lambda(x,\om)\I_2(\om) dx d\om &\text{if } n\ge 7,
\end{cases}\]
and
\[\begin{medsize}
\displaystyle J_{22} := \int_{\Omega} \(\I_1\mone_{\{n\ge 6,\; u_0>0\}} + \I_2\mone_{\{n\ge 6,\; \nu\ge 2\}} + \sum_{i=1}^{\nu} \lambda PU_i\mone_{\{n=5,6\}}\)(\om) \int_{\Omega}G_\lambda(x,\om) [p(u_0+\sigma)^{p-1}\trh_0](x)dx d\om.
\end{medsize}\]
Then, by appealing to the inequality $\|\trh_0\|_{H^1_0(\Omega)} \lesssim \vsi_2(\delta)$, \eqref{bcik}, \eqref{i12e}, \eqref{i3e}, the non-negativity of the functions $\I_1$, $\I_2$ and $\lambda PU_i$, and Lemma \ref{lem2.1}, we obtain
\begin{align*}
J_2 &\gtrsim J_{21} + J_{22} \\
&\begin{medsize}
\displaystyle \ + O\(\bigg\|\I_1\mone_{\{n\ge 6,\; u_0>0\}} + \I_2\mone_{\{n\ge 6,\; \nu\ge 2\}} + \sum_{i=1}^{\nu}\lambda PU_i\mone_{\{n=5,6\}} \bigg\|_{L^{\frac{p+1}{p}}(\Omega)} \(\|\trh_0\|_{H^1_0(\Omega)}^2 + \sum_{i=1}^{\nu}\sum_{k=0}^n|\tc_i^k|\)\)
\end{medsize} \\
&= J_{21}+J_{22}+o\(\vsi_2(\delta)^2\).
\end{align*}

Assume that $n = 5, 6$. A direct computation shows
\begin{equation}\label{in56}
\begin{aligned}
&\ \int_{\Omega} \sum_{i=1}^{\nu} PU_i \int_{\Omega}G_\lambda(x,\om) \sum_{j=1}^{\nu} PU_j dxd\om \\
&\gtrsim \int_{\Omega} \sum_{i,j=1}^{\nu} \(\frac{\delta}{\delta^2+|x-\xi_i|^2}\)^{\frac{n-2}{2}} \frac{\delta^{\frac{n-2}{2}}}{(\delta^2+|x-\xi_j|^2)^{\frac{n-4}{2}}} %\\
%&\gtrsim \int_{\Omega} \frac{\delta^{n-2}}{(\delta^2+|x-\xi_1|^2)^{n-3}}
\simeq \begin{cases}
\delta^3 &\text{if } n=5,\\
\delta^4|\log\delta| &\text{if } n=6.
\end{cases}
\end{aligned}
\end{equation}
Assume that $n\ge 7$ and $\nu\ge 2$. If $|x_1|\lesssim \frac12 \delta^{b-1}$, then $|x_2|\le |x_1|+\frac{|\xi_1-\xi_2|}{\delta} \lesssim \delta^{b-1}$. From this, we derive
\[\I_2=\sigma^{p} -\sum_{i=1}^{\nu}(PU_i)^{p} \gtrsim (PU_1)^{p-1}PU_2\]
and
\[U_1^{p-1}U_2\gtrsim \frac{\delta^{-2}}{\la x_1\ra^4} \frac{\delta^{-\frac{n-2}{2}}}{\la x_2\ra^{n-2}}\gtrsim \frac{\delta^{-2}}{\la x_1\ra^4} \delta^{-\frac{n-2}{2}}\delta^{(1-b)(n-2)}\gtrsim \frac{\delta^{(\frac12-b)(n-2)-2}}{\la x_1\ra^4}.\]
As a consequence, we have
\begin{equation}\label{in67}
\begin{aligned}
J_{21} &= \int_{\Omega} \I_2(x)\int_{\Omega}G_\lambda(x,\om)\I_2(\om) dx d\om \\
&\gtrsim
\delta^{(2-2b)(n-2)} \int_{\{|x_1| \lesssim \frac12\delta^{b-1}\}} \int_{\{|\om_1| <\frac12 |x_1|\}} \frac{1}{\la x_1 \ra^4} \frac{1}{|x_1-\om_1|^{n-2}} \frac{1}{\la \om_1 \ra^4} dx_1d\om_1 + o(\delta^{(1-b)(n+2)})\\
&\gtrsim
\delta^{(2-2b)(n-2)} \int_{\{|x_1| \lesssim \frac12\delta^{b-1}\}} \frac{1}{\la x_1 \ra^6} dx_1 + o(\delta^{(1-b)(n+2)}) \gtrsim \delta^{(1-b)(n+2)},
\end{aligned}
\end{equation}
where $\om_1:= \delta_1^{-1}(\om-\xi_1)$.

We next estimate $J_{22}$. We denote
\[\tv_{ji}^{\tin}:=\frac{\delta_i^{-2}}{\la x_i\ra^2}\tw_{ji}^{\tin},\quad \tv_{ji}^{\tout}:=\frac{\delta_i^{-2}}{\la x_i\ra^2}\tw_{ji}^{\tout} \quad \text{for } j=1,2,\quad \tv_{3i}:=\frac{\delta_i^{-2}}{\la x_i\ra^2}\tw_{3i},\]
and
\[\wtv:=\sum_{i=1}^{\nu}\left[(\tv_{1i}^{\tin}+\tv_{1i}^{\tout})\mone_{\{n\ge 7,\; u_0>0\}} + (\tv_{2i}^{\tin}+\tv_{2i}^{\tout})\mone_{\{n\ge 7,\; \nu\ge 2\}} + \tv_{3i}\mone_{\{n=5,6\}}\right].\]
Recalling \eqref{i1po}, Lemma \ref{lem2.1} and \eqref{iqu}, we easily observe that
\[\bigg|\I_1\mone_{\{n\ge 6,\; u_0>0\}} + \I_2\mone_{\{n\ge 6,\; \nu\ge 2\}} + \sum_{i=1}^{\nu}\lambda PU_i\mone_{\{n=5,6\}}\bigg| \lesssim \wtv,\]
which implies
\begin{equation}\label{J22}
J_{22} \lesssim \int_{\Omega} \wtv(x) \int_{\Omega} \frac{1}{|x-\om|^{n-2}} [(\mone_{\{u_0>0\}}+\sigma^{p-1})|\trh_0|](\om) d\om dx.
\end{equation}
Hence, it suffices to estimate the right-hand side of \eqref{J22}.

It holds that
\[\|\wtv\|_{L^{\frac{p+1}{p}}(\Omega)} \lesssim \begin{cases}
\delta^2|\log\delta|^{\frac23} &\text{if } n=6,\\
\delta^{\frac{(1-b)(n+2)}{2}} &\text{if } n\ge 7,
\end{cases}\]
and
\[\left\|\tw_{1i}^{\tin}\right\|_{L^{\frac{p+1}{p}}(\Omega)} + \left\|\tw_{1i}^{\tout}\right\|_{L^{\frac{p+1}{p}}(\Omega)} \lesssim \delta^{\frac{n+2}{4}} \quad \text{if } n\ge 7,\quad \|\tw_{3i}\|_{L^{\frac{p+1}{p}}(\Omega)} \lesssim \delta^2 \quad \text{if } n=6,\]
\[\left\|\tw_{2i}^{\tin}\right\|_{L^{\frac{p+1}{p}}(\Omega)} +
\left\|\tw_{2i}^{\tout}\right\|_{L^{\frac{p+1}{p}}(\Omega)} \lesssim
\delta^{\frac{(1-b)(n+2)}{2}+2} \quad \text{if } n\ge 7.\]
By these bounds and the Hardy-Littlewood-Sobolev inequality, we have
\begin{equation}\label{uo67}
\begin{aligned}
\int_{\Omega} \wtv(x) \int_{\Omega} \frac{1}{|x-\om|^{n-2}}|\trh_0|(\om) dxd\om \mone_{\{u_0>0\}} &\lesssim \|\wtv\|_{L^{\frac{p+1}{p}}(\Omega)} \cdot \|\wtw\|_{L^{\frac{p+1}{p}}(\Omega)}\\
&=o\(\left.\begin{cases}
\delta^4|\log\delta| &\text{if } n=6,\\
\delta^{(1-b)(n+2)} &\text{if } n\ge 7
\end{cases}\right\}\).
\end{aligned}
\end{equation}

In the following, we estimate the integral $\int_{\Omega} \frac{1}{|x-\om|^{n-2}} (\sigma^{p-1}|\trh_0|)(\om) d\om$ by dividing cases according to the dimension.

We redefine $\olw$ as
\[\olw(x) := \begin{cases}
\displaystyle \sum_{i=1}^{\nu}\tw_{3j}\frac{\log(2+|x_i|)}{\la x_i\ra^{2}} &\text{if } n=5,6,\\
\begin{aligned}
\displaystyle &\sum_{i=1}^{\nu}\bigg[\tw_{1i}^{\tin}\frac{1}{\la x_i\ra^2} + \tw_{1i}^{\tout} \frac{\log(2+|x_i|)}{\la x_i\ra^2}\\
\displaystyle &\ + \tw_{2i}^{\tin}\frac{1}{\la x_i\ra^2}\mone_{\{|x_i|\le \min_{i\ne j}\frac{|\xi_i-\xi_j|}{2\delta}\}} + \tw_{2i}^{\tout}\frac{\log(2+|x_i|)}{\la x_i\ra^2}\mone_{\{|x_i|\ge \min_{i\ne j}\frac{|\xi_i-\xi_j|}{2\delta}\}}\bigg]
\end{aligned} &\text{if } n\ge 7.
\end{cases}\]

\medskip \noindent (1) When $n=5, 6$, we notice from Young's inequality that
\[U_i^{p-1}\tw_{3j}\lesssim \left[\frac{\delta^{2+\frac{4}{n}}}{(\delta^2+|x-\xi_i|^2)^2}\right]^{\frac{n}{4}} + \left[\frac{\delta^{\frac{n-2}{2}-\frac{4}{n}}}{(\delta^2+|x-\xi_j|^2)^{\frac{n-4}{2}}}\right]^{\frac{n}{n-4}}
\lesssim \frac{\delta^{-2}}{\la x_i\ra^4}\tw_{3i}+\frac{\delta^{-2}}{\la x_j\ra^4}\tw_{3j}.
\]
for any $i, j \in \{1,\ldots,\nu\}$. Therefore,
\[\int_{\Omega} \frac{1}{|x-\om|^{n-2}} (\sigma^{p-1}|\trh_0|)(\om) d\om \lesssim  \olw(x).\]

\medskip \noindent (2) Assume that $n \ge 7$.

Suppose that $|x-\xi_j|\leq \sqrt{\delta}$. If we also have $|x-\xi_i|\leq \sqrt{\delta}$, then Young's inequality yields
\[U_i^{p-1}\tw_{1j}^{\tin}
\lesssim \left[\frac{\delta^{\frac83}}{(\delta^2+|x-\xi_i|^2)^2}\right]^{\frac32} + \left[\frac{\delta^{\frac43}}{\delta^2+|x-\xi_j|^2}\right]^3 \lesssim \frac{1}{\la x_i\ra^2}\tv_{1i}^{\tin}+\frac{1}{\la x_j\ra^2}\tv_{1j}^{\tin}\]
for any $i, j \in \{1,\ldots,\nu\}$. If $|x-\xi_i|\ge \sqrt{\delta}$ is valid, then using the inequalities $\frac{\delta}{\delta^2+|x-\xi_i|^2} \lesssim 1\lesssim \frac{\delta}{\delta^2+|x-\xi_j|^2}$, we find
\[U_i^{p-1}\tw_{1j}^{\tin} \lesssim \frac{\delta^3}{(\delta^2+|x-\xi_j|^2)^2} \lesssim \delta \tv_{1j}^{\tin}.\]

Suppose next that $|x-\xi_j|\geq \sqrt{\delta}$. If we also have $|x-\xi_i|\geq \sqrt{\delta}$, then Young's inequality again gives
\[U_i^{p-1}\tw_{1j}^{\tout} \lesssim\frac{1}{\la x_i\ra^2}\tv_{1i}^{\tout}+\frac{1}{\la x_j\ra^2}\tv_{1j}^{\tout}
\lesssim \delta\(\tv_{1i}^{\tout}+ \tv_{1j}^{\tout}\).\]
If $|x-\xi_i|\leq \sqrt{\delta}$ is valid, noting that $\tw_{1j}^{\tout} \lesssim \delta$, we also have
\[U_i^{p-1}\tw_{1j}^{\tout} \lesssim \delta \tv_{1i}^{\tin}.\]

Moreover, if $| x_j| \le \frac{|\xi_i-\xi_j|}{2\delta}$ and $|x_i|\ge \frac{|\xi_i-\xi_j|}{2\delta} \ge |x_j|$, then
\[U_i^{p-1}\tw_{2j}^{\tin} \lesssim \delta^{2(1-b)} \tv_{2j}^{\tin}.\]
Suppose that $|x_j| \ge \frac{|\xi_i-\xi_j|}{2\delta}$ so that $1+|x_i|\le |x_j|+\frac{|\xi_i-\xi_j|}{\delta}+1\lesssim |x_j|$. If we also have $|x_i|\le\frac{|\xi_i-\xi_j|}{2\delta}$, then
\[U_i^{p-1}\tw_{2j}^{\tout} \lesssim \frac{\delta^{-2}}{\la x_i\ra^6} \delta^{4(1-b)-\frac{n-2}{2}+(1-b)(n-6)}=\frac{1}{\la x_i\ra^2}v_{2i}^{\tin}.\]
If $|x_i|\ge\frac{|\xi_i-\xi_j|}{2\delta}$ holds, then by Young's inequality,
\[U_i^{p-1}\tw_{2j}^{\tout} \lesssim \(\frac{\tv_{2i}^{\tout}}{\la x_i\ra^2}\)^{\frac{n-4}{n}}\(\frac{\tv_{2j}^{\tout}}{\la x_j\ra^2}\)^{\frac{4}{n}} \lesssim \frac{1}{\la x_i\ra^{2}}\tv_{2i}^{\tout}+\frac{1}{\la x_j\ra^{2}}\tv_{2j}^{\tout}.\]

Putting the estimates above together, we conclude
\begin{align*}
&\ \int_{\Omega} \frac{1}{|x-\om|^{n-2}} (\sigma^{p-1}|\trh_0|)(\om)d\om\\
&\lesssim \int_{\Omega} \frac{1}{|x-\om|^{n-2}} \sum_{i=1}^{\nu}\left[\(\frac{1}{\la x_i\ra^2}+\delta\)\(\tv_{1i}^{\tin}+\tv_{1i}^{\tout}\) + \(\frac{1}{\la x_i\ra^2}+\delta^{2(1-b)}\) \(\tv_{2i}^{\tin}+\tv_{2i}^{\tout}\) \right]dx\\
&\lesssim \olw(x) + \left[\delta^{2(1-b)}+\delta\right] \wtw(x).
\end{align*}

\medskip
Now, if we select $L>0$ satisfying $\{|x_i|\le L\}\cap \{|x_j|\le L\}=\emptyset$ for $1 \le i\neq j \le \nu$ and $L^{-1}\gtrsim \delta^{2(1-b)}+\delta$, then
\begin{equation}\label{uo68}
\begin{aligned}
&\ \int_{\Omega} \wtv(x) \int_{\Omega} \frac{1}{|x-\om|^{n-2}} (\sigma^{p-1}|\trh_0|)(\om) d\om dx \\
&\lesssim \int_{\Omega} \wtv(x) \left[\olw(x)\mone_{\cup_{i=1}^{\nu}\{|x_i|\le L\}}+L^{-1}\wtw(x)\mone_{\cap_{i=1}^{\nu}\{|x_i|\ge L\}}\right] dx \\
&\lesssim \sum_{i=1}^{\nu} \int_{\{|x_i|\le L\}} \left[\frac{\delta^{(1-2b)(n-2)-2} }{\la x_i\ra^8}\mone_{\{n\ge 7,\; \nu\ge 2\}}+ \frac{\delta^{-2}\log(2+|x_i|)}{\la x_i\ra^8}\mone_{\{n\ge 7,\; u_0>0\}} \right. \\
&\hspace{50pt} \left. + \frac{\delta^{(\frac12-b)(n-2)-2}\log(2+|x_i|)}{\la x_i\ra^8}\mone_{\{n\ge 7,\; u_0>0,\; \nu\ge 2\}}
+\frac{\delta^{4-n}\log(2+|x_i|)}{\la x_i\ra^{2n-4}}\mone_{\{n=5,6\}} \right] dx \\
&\ + L^{-1}\|\wtv\|_{L^{\frac{p+1}{p}}(\Omega)}\|\wtw\|_{L^{p+1}(\Omega)} \\
&\lesssim \left.\begin{cases}
\delta^4&\text{if } n=5,6\\
\delta^{(2-2b)(n-2)} &\text{if } n\ge 7\text{~and~}\nu\ge 2
\end{cases}\right\}
+L^{-1} \vsi_2(\delta)^2.
\end{aligned}
\end{equation}
Estimates \eqref{in56}--\eqref{uo67}, along with \eqref{uo68} for $L>0$ large enough, imply the validity of \eqref{inrp}.

\medskip \noindent \doublebox{\textsc{Step 3.}} Let $u_{\sharp}:=u_0+\sum_{i=1}^{\nu}PU_i+\rho$, $(u_{\sharp})_{\pm} := \max\{\pm u_{\sharp},0\}$, and $u_*:=\(u_{\sharp}\)_+$.

\medskip
Observe that
\begin{equation}\label{usharp}
\begin{cases}
(-\Delta-\lambda) u_{\sharp}=|u_{\sharp}|^{p-1}u_{\sharp} + \sum_{k=0}^n\sum_{i=1}^{\nu} \tc_i^k(-\Delta-\lambda)PZ_i^k &\text{in } \Omega, \\
u_{\sharp} = 0 &\text{on } \pa\Omega.
\end{cases}
\end{equation}

Assuming that $\tilde{\xi}_i$ satisfies the assumption of Theorem \ref{thm1}, we introduce
\[d_*(u) := \inf\left\{\Big\|u-\Big(u_0+\sum_{i=1}^{\nu} PU_{\tde_i,\txi_i}\Big)\Big\|_{H^1_0(\Omega)}: \(\tde_i,\txi_i\) \in (0,\infty) \times \Omega,\, i = 1,\ldots,\nu\right\}.\]
Arguing as in Case 1, we can verify
\begin{equation}\label{dstr}
d_*(u_{\sharp})\gtrsim \|\rho\|_{H^1_0(\Omega)}\simeq \vsi_2(\delta).
\end{equation}
Testing \eqref{usharp} with $(u_{\sharp})_{-}$ gives
\[\|(u_{\sharp})_{-}\|^2 = \|(u_{\sharp})_{-}\|_{L^{p+1}(\Omega)}^{p+1} + \int_{\Omega} \sum_{k=0}^n\sum_{i=1}^{\nu} \tc_i^k(-\Delta-\lambda)PZ_i^k(u_{\sharp})_{-}.\]
 Using the estimate $ |(u_{\sharp})_{-}|\lesssim |\rho|$, we get
\[\|(u_{\sharp})_{-} \|^2\lesssim \|\rho\|_{H^1_0(\Omega)}^{p+1} + \int_{\Omega} \sum_{k=0}^n\sum_{i=1}^{\nu} |\tc_i^k||(-\Delta-\lambda)PZ_i^k||\rho| = o(1),\]
and since
\[\|(u_{\sharp})_{-}\|_{L^{p+1}(\Omega)}^{p+1} \lesssim \|(u_{\sharp})_{-}\|^{p+1}_{H^1_0(\Omega)} = o(\|(u_{\sharp})_{-}\|^2_{H^1_0(\Omega)}),\]
we obtain
\begin{equation}\label{shpm}
\begin{aligned}
\|(u_{\sharp})_{-}\|^2 &\lesssim \sum_{k=0}^n\sum_{i=1}^{\nu} |\tc_i^k| \left[\int_{\Omega} |(-\Delta-\lambda)PZ_i^k||\trh_0| +\|\trh_1\|_{H^1_0(\Omega)}\right] \\
&\lesssim \sum_{k=0}^n\sum_{i=1}^{\nu}|\tc_i^k| \bigg[\int_{\Omega} \sum_{j=1}^{\nu}(U_j^{p}+U_j) \sum_{i=1}^{\nu} \left[\(\tw_{1i}^{\tin}+\tw_{1i}^{\tout}\)\mone_{\{n\ge 7,\; u_0>0\}}\right. \nonumber\\
&\hspace{150pt} \left.+ \(\tw_{2i}^{\tin}+\tw_{2i}^{\tout}\)\mone_{\{n\ge 7,\; \nu\ge 2\}} + \tw_{3i}\mone_{\{n=5,6\}}\right]+ \|\trh_1\|_{H^1_0(\Omega)}\bigg] \\
&\lesssim \vsi_1(\delta)^2.
\end{aligned}
\end{equation}
Here, the last inequality holds thanks to the following estimates:
\[\int_{\Omega} \tw_{3i} (U_i^{p}+U_i) \mone_{\{n=5,6\}} \lesssim \begin{cases}
\delta^2 &\text{if } n=5,\\
\delta^4|\log\delta| &\text{if } n=6,
\end{cases}\]
and for $1 \le i\ne j \le \nu$,
\[\int_{\Omega} \tw_{3i} (U_j^{p}+U_j) \mone_{\{n=5,6\}} \lesssim \delta^{\frac{n-2}{2}}\int_{B(\xi_i,\frac{|\xi_i-\xi_j|}{2})} \tw_{3i} + \delta^{\frac{n-2}{2}}\int_{B(\xi_j,\frac{|\xi_i-\xi_j|}{2})} (U_j^{p}+U_j)
+\delta^{n-2}\simeq \delta^{n-2}.\]
If $n\ge 7$, then
\[\int_{\Omega} \tw_{1i}^{\tin} (U_i^{p}+U_i) \lesssim \delta^{\frac{n-2}{2}},\quad \int_{\Omega} \tw_{1i}^{\tout} (U_i^{p}+U_i) \lesssim \delta^{\frac{n+2}{2}},\]
\[\int_{\Omega} \tw_{2i}^{\tin} (U_i^{p}+U_i) \lesssim\delta^{(1-b)(n-2)},\quad \int_{\Omega} \tw_{2i}^{\tout} (U_i^{p}+U_i) \lesssim \delta^{(1-b)(n-2)+2},\]
and for $1 \le i\ne j \le \nu$,
\[\int_{\Omega} \tw_{1i}^{\tin} (U_j^{p}+U_j) =\(\int_{|x_i|\le \frac{|\xi_i-\xi_j|}{2\delta}}+\int_{\frac{|\xi_i-\xi_j|}{2\delta}\le |x_i|\le \delta^{-1/2}}\) \tw_{1i}^{\tin} (U_j^{p}+U_j) \lesssim \delta^{2(1-b)+\frac{n-2}{2}},\]
\[\int_{\Omega} \tw_{1i}^{\tout} (U_j^{p}+U_j) \lesssim \delta \int_{\Omega} (U_j^{p}+U_j) \lesssim \delta^{\frac{n}{2}},\]
\[\int_{\Omega} \tw_{2i}^{\tin} (U_j^{p}+U_j) \lesssim
\delta^{\frac{(1-2b)(n-2)}{2}}\int_{\Omega} \tw_{2i}^{\tin} \lesssim \delta^{(1-b)(n-2)+2},\]
\[\int_{\Omega} \tw_{2i}^{\tout} (U_j^{p}+U_j) \lesssim \delta^{(1-b)n-\frac{n-2}{2}}\int_{\Omega}(U_j^p+U_j) \lesssim \delta^{(1-b)n}.\]

\medskip
Therefore, by combining estimates \eqref{tcik}, \eqref{dstr} and \eqref{shpm}, we infer
\[d_*(u_*)\gtrsim d_*(u_{\sharp})-\|(u_{\sharp})_{-}\| \gtrsim \|\rho\|_{H^1_0(\Omega)}\simeq \vsi_2(\delta).
\]
Moreover,
\[\Gamma(u_*)\lesssim \tilde{\Gamma}(u_{\sharp})+\|(u_{\sharp})_{-}\|\lesssim \vsi_1(\delta),\]
where $\tilde{\Gamma}(u_{\sharp}):=\|\Delta u_{\sharp}+\lambda u_{\sharp}+|u_{\sharp}|^{p-1}u_{\sharp}\|_{H^{-1}(\Omega)} \lesssim \vsi_1(\delta).$
In conclusion, we obtain a function $u_*\ge 0$ satisfying
\[d_*(u_*) \gtrsim \zeta(\Gamma(u_*)),\]
thereby establishing the optimality of \eqref{eq:sqe32}.

\section{Proof of Theorem \ref{thm2}}\label{sec4}
In this section, we investigate the single-bubble case ($\nu=1$), allowing the distance between $\xi_1$ and $\pa\Omega$ to be arbitrarily small, and prove Theorem \ref{thm2}
We assume that the function $PU_1$ satisfies \eqref{pu2} when $n=3$ or $[n=4,5,\; u_0=0]$, and satisfies \eqref{pu1} when $[n=4,5,\; u_0>0]$ or $n \geq 6$; see Remark \ref{rmk:thm2}(2).
Throughout this section, we write $\ka_1 = \frac{\delta_1}{d(\xi_1,\pa\Omega)}$.

\medskip
We first examine the case when $n=5$ and $PU_1$ satisfies \eqref{pu2}. By Lemma \ref{lem2.1} and \eqref{tdn}, we have
\begin{align}
\| \I_3 \|_{L^{\frac{p+1}{p}}(\Omega)} & \lesssim\|(PU_1-U_1)U_1^{p-1}\|_{L^{\frac{p+1}{p}}(\Omega)} \nonumber \\
&\lesssim \delta_1^{\frac{3}{2}} \(\bigg\|\frac{1}{|\cdot-\xi_1|}U_1^{p-1}\bigg\|_{L^{\frac{p+1}{p}}(\Omega)} + |\vph^5_\lambda(\xi_1)|\|U_1^{p-1}\|_{L^{\frac{p+1}{p}}(\Omega)}\) + \|\wtmcd_5U_1^{p-1}\|_{L^{\frac{p+1}{p}}(\Omega)} \nonumber \\
&\lesssim \delta_1^2 + \ka_1^3 \label{i35}
\end{align}
and
\begin{align}
\int_{B(\xi_1, d(\xi_1,\pa\Omega))} \left[(PU_1)^{p}-U_1^{p}\right] PZ_1^0 &= \frac{\lambda}{2} a_5\delta_1^{\frac32}p \int_{B(\xi_1,d(\xi_1,\pa\Omega))} \frac{1}{|x-\xi_1|}(U_1^{p-1}Z_1^0)(x) dx \nonumber \\
&\ + \lambda a_5^2p\delta_1^2 \int_{B(0,\frac{d(\xi_1,\pa\Omega)}{\delta_1})} \left[\frac{1}{(1+|z|^2)^{\frac32}}-\frac{1}{|z|^3}\right] \frac{|z|^2-1}{(1+|z|^2)^{\frac52}} dz \nonumber \\
&\ -\delta_1^{\frac{3}{2}}a_5p H^5_{\lambda}(\xi_1,\xi_1) \int_{B(\xi_1, d(\xi_1,\pa\Omega))}U_1^{p-1}Z_1^0 dx+O(\delta_1^3) + O(\ka_1^5) \nonumber \\
&=\bar{\mfb}_5\lambda\delta_1^2 - \mfc_5\delta_1^3\vph^5_{\lambda}(\xi_1) + O(\delta_1^3)+O(\ka_1^5). \label{puz54}
\end{align}
Here,
\[\bar{\mfb}_5:=\frac{a_5}{2} \int_{\R^5}\frac{1}{|x|}(U^{p-1}Z^0)(x) dx + a_5^2p \int_{\R^5} \left[\frac{1}{(1+|z|^2)^{\frac32}}-\frac{1}{|z|^3}\right] \frac{|z|^2-1}{(1+|z|^2)^{\frac52}} dz > 0\]
and $\mfc_5:=a_5p\int_{\R^5}U^{p-1}Z^0 > 0$.

Combining \eqref{i35} with \eqref{sirp}, we obtain
\begin{equation}\label{rhsin}
\begin{aligned}
\|\rho\|_{H^1_0(\Omega)} &\lesssim \|f\|_{(H^1_0(\Omega))^*}+\left.\begin{cases}
\delta_1 &\text{if } n=3 \text{ and } u_0=0 \\
\delta_1^2|\log\delta_1| &\text{if } n=4 \text{ and } u_0=0 \\
\delta_1^{\frac{n-2}{2}} &\text{if } n=3,4,5 \text{ and } u_0>0 \\
\delta_1^2|\log\delta_1|^{\frac12} &\text{if } n=6\\
\delta_1^2 &\text{if } [n=5,\; u_0=0] \text{ or } n\ge 7
\end{cases}\right\} \\
&\ + \left.\begin{cases}
\displaystyle \ka_1^{n-2} &\text{if } n=3,4,5 \\
\displaystyle \ka_1^4|\log\ka_1|^{\frac12} &\text{if } n=6 \\
\displaystyle \ka_1^{\frac{n+2}{2}} &\text{if } n\ge 7
\end{cases}\right\}.
\end{aligned}
\end{equation}
Also, applying \eqref{eq:I2Z}, \eqref{i3pz1}--\eqref{i3pz2}, and \eqref{puz54}, we deduce
\begin{align}
&\ \int_{\Omega} (\I_1+\I_3) PZ_1^0 \nonumber \\
&= \left[\mfa_nu_0(\xi_1)\delta_1^{\frac{n-2}{2}}+ \left.\begin{cases}
O(\delta_1) &\text{if } n=3\\
O(\delta_1^2|\log\delta_1|) &\text{if } n=4\\
O(\delta_1^2) &\text{if } n=5
\end{cases}\right\} \mone_{\{p>2\}}
+ O\(\delta_1^{\frac{n}{2}} + \ka_1^n\)\right] \mone_{\{u_0>0\}} \label{iz0} \\
&\ + \begin{cases}
\displaystyle -\mfc_3 \vph^3_{\lambda}(\xi_1)\delta_1 + O(\delta_1^2) + O(\ka_1^3) &\text{if } n=3,\\
\displaystyle \mfb_4 \lambda\delta_1^2|\log\delta_1| - \mfc_4 \delta_1^2\vph^4_{\lambda}(\xi_1) - 96|\S^3|\lambda\delta_1^2 + O(\delta_1^3) + O(\ka_1^4) &\text{if } n=4 \text{ and } u_0=0,\\
\displaystyle -\delta_1^2\mfc_4 \vph(\xi_1) + O(\delta_1^2|\log\delta_1|) + O(\ka_1^4) &\text{if } n=4 \text{ and } u_0>0,\\
\displaystyle \bar{\mfb}_5\lambda\delta_1^2 - \mfc_5\delta_1^3\vph^5_{\lambda}(\xi_1) + O(\delta_1^3) + O(\ka_1^5) &\text{if } n=5 \text{ and } u_0=0,\\
\displaystyle \lambda\mfb_n \delta_1^2-\delta_1^{n-2}\mfc_n \vph(\xi_1) + O\(\delta_1^2\ka_1^{n-4}\) + O(\ka_1^n) &\text{if } [n=5,\; u_0>0] \text{ or } n\ge 6.
\end{cases} \nonumber
\end{align}

As mentioned earlier, certain cancellations between terms with opposite signs may occur in \eqref{iz0}.
To handle this issue, we establish an estimate for the projection of the term $\I_1 + \I_3$ onto the direction of spatial derivatives of $PU_1$, as stated in the following lemma.
\begin{lemma}%\label{i13k}
For any $k\in\{1,..,n\}$, there exists a constant $\mfe_n > 0$ such that
\begin{equation}\label{izk}
\begin{aligned}
\bigg|\int_{\Omega}(\I_1+\I_3)PZ_1^k\bigg| &= (1+o(1))\mfe_n \delta_1^{n-1} \times
\left. \begin{cases}
\displaystyle \left|\frac{\pa \vph^n_{\lambda}}{\pa \xi_1^k}(\xi_1)\right| &\text{if } n=3 \text{ or } [n=4,5,\; u_0=0]\\
\displaystyle \left|\frac{\pa\vph}{\pa\xi_1^k}(\xi_1)\right| &\text{if } [n=4,5,\; u_0>0] \text{ or } n\ge 6
\end{cases} \right\}\\
&\ + \begin{cases}
O(\delta_1^3) &\text{if } n=3,4,5 \text{ and } u_0=0,\\
O(\delta_1^{\frac32}|\log\delta_1|) &\text{if } n=3 \text{ and } u_0>0,\\
O(\delta_1^2|\log\delta_1|) &\text{if } n=4 \text{ and } u_0>0,\\
O(\delta_1^2 d(\xi_1,\pa\Omega)+\delta_1^{\frac{n}{2}}) &\text{if } n=5 \text{ and } u_0>0,\\
O(\delta_1^{\frac{n}{2}}) &\text{if } n\ge 6.
\end{cases}
\end{aligned}
\end{equation}
\end{lemma}
\begin{proof}
By using Corollary \ref{cor:PZ^kexp}, we obtain
\begin{align*}
\int_{B(\xi_1, d(\xi_1,\pa\Omega))} u_0 (PU_1)^{p-1} PZ_1^k &= \frac{\pa u_0}{\pa \xi_1^k}(\xi_1) \int_{B(\xi_1, d(\xi_1,\pa\Omega))} (x-\xi_1)^k(U_1^{p-1}Z_1^k)(x) dx + O\(\delta_1^{\frac{n}{2}}+\ka_1^n\) \\
&= O\(\delta_1^{\frac{n}{2}}+\ka_1^n\)
\end{align*}
for $n \ge 3$ (cf. \eqref{21}).

Let us refine estimate \eqref{24} for the cases $n=3,5$ and $u_0>0$. If $n=5$ and $u_0>0$, then we have
\begin{align*}
\int_{B(\xi_1,\eta\sqrt{\delta_1})} u_0^2 (PU_1)^{p-2} |PZ_1^k| &\lesssim \int_{B(\xi_1, d(\xi_1,\pa\Omega))} U_1^{p-1} + \int_{B(\xi_1,d(\xi_1,\pa\Omega))^c} U_1^{p+1} \\
&\lesssim \delta_1^2d(\xi_1,\pa\Omega)+\ka_1^n.
\end{align*}
Suppose that $n=3$ and $u_0>0$. Applying \eqref{ab7}, we expand $\I_1$ by
\begin{align*}
\I_1 &= \left[pu_0(PU_1)^{p-1}+\frac{p(p-1)}{2}u_0^2(PU_1)^{p-2}+O(u_0^3(PU_1)^{p-3})+ O\(u_0^{p}\)\right] \mone_{B(\xi_1,\eta\sqrt{\delta_1})} \\
&\ +\left[pu_0^{p-1}PU_1 + O(u_0^{p-2}(PU_1)^2)\mone_{\{p>2\}}+ O((PU_1)^{p})\right] \mone_{B(\xi_1,\eta\sqrt{\delta_1})^c}.
\end{align*}
Also, we have
\begin{align*}
\int_{B(\xi_1,\eta\sqrt{\delta_1})} u_0^2 (PU_1)^{p-2} PZ_1^k &= 2(1+o(1)) u_0(\xi_1)\frac{\pa u_0}{\pa \xi_1^k}(\xi_1) \int_{B(\xi_1, d(\xi_1,\pa\Omega))}(x-\xi_1)^k(U_1^{p-2}Z_1^k)(x) dx \\
&\ + O\(\frac{\delta_1^{\frac12}}{d(\xi_1,\pa\Omega)}\int_{B(\xi_1, d(\xi_1,\pa\Omega))}U_1^{p-2}\) + \int_{B(\xi_1,d(\xi_1,\pa\Omega))^c}U_1^{p+1} \\
&\lesssim \delta_1^2|\log\delta_1| + \delta_1\ka_1\left|\log\ka_1^{-1}\right| + \ka_1^n
\end{align*}
and
\[\int_{B(\xi_1,\eta\sqrt{\delta_1})} u_0^3 (PU_1)^{p-3} |PZ_1^k| \lesssim \int_{B(\xi_1,\eta\sqrt{\delta_1})} U_1^{p-2} \lesssim \delta_1^{\frac32}|\log\delta_1|.\]

By combining the above estimates with \eqref{i2ex} and \eqref{25}--\eqref{29}, we conclude
\begin{equation}\label{eq:I1PZk}
\begin{aligned}
\int_{\Omega} \I_1PZ_1^k = \left[\left.\begin{cases}
O(\delta_1^{\frac32}|\log\delta_1|) &\text{if } n=3\\
O(\delta_1^2|\log\delta_1|) &\text{if } n=4\\
O(\delta_1^2d(\xi_1,\pa\Omega)) &\text{if } n=5
\end{cases}\right\} \mone_{\{p>2\}}
+ O\(\delta_1^{\frac{n}{2}} + \ka_1^n\)\right] \mone_{\{u_0>0\}}.
\end{aligned}
\end{equation}

\medskip
On the other hand, arguing as in \eqref{puz5}--\eqref{puz4} and \eqref{puz54}, we can find a constant $\mfe_n>0$ such that
\begin{equation}\label{eq:I3PZk1}
\begin{aligned}
&\ \int_{\Omega} [(PU_1)^{p}-U_1^{p}] PZ_1^k \\
&=\begin{cases}
\displaystyle -\delta_1^{\frac{n-2}{2}}a_np\frac{\pa \vph^n_{\lambda}}{\pa \xi_1^k}(\xi_1) \int_{B(\xi_1, d(\xi_1,\pa\Omega))}(x-\xi_1)^k(U_1^{p-1}Z_1^k)(x) dx + O(\delta_1^3) + O(\ka_1^n) \\
\hspace{250pt}\text{if } n=3 \text{ or } [n=4,5,\; u_0=0] \\
\displaystyle -\delta_1^{\frac{n-2}{2}}a_np \frac{\pa\vph}{\pa\xi_1^k}(\xi_1) \int_{B(\xi_1, d(\xi_1,\pa\Omega))} (x-\xi_1)^k(U_1^{p-1}Z_1^k)(x) dx + O(\ka_1^n) \\
\hspace{250pt}\text{if } [n=4,5,\; u_0>0] \text{ or } n\ge 6
\end{cases}\\
&= \begin{cases}
\displaystyle - (1+o(1))\mfe_n \delta_1^{n-1} \frac{\pa \vph^n_{\lambda}}{\pa \xi_1^k}(\xi_1) + O(\delta_1^3) + O(\ka_1^n) &\text{if } n=3 \text{ or } [n=4,5,\; u_0=0],\\
\displaystyle - (1+o(1))\mfe_n \delta_1^{n-1} \frac{\pa\vph}{\pa\xi_1^k}(\xi_1) + O(\ka_1^n) &\text{if } [n=4,5,\; u_0>0] \text{ or }n\ge 6.
\end{cases}
\end{aligned}
\end{equation}
Here, we also used Corollary \ref{cor:PZ^kexp}.

Moreover, we see from \eqref{lmpz} and \eqref{puzl} that
\begin{equation}\label{eq:I3PZk2}
\int_{\Omega} \lambda PU_1 PZ_1^k \mone_{\{PU_1 \text{ satisfies } \eqref{pu2}\}}
=\begin{cases}
O(\delta_1^2|\log\delta_1|) &\text{if } n=4 \text{ and } u_0>0,\\
O\(\delta_1^{\frac{n}{2}}+\ka_1^n\) &\text{if } [n=5,\; u_0>0] \text{ or } n\ge 6.
\end{cases}
\end{equation}

Consequently, \eqref{izk} follows immediately from \eqref{eq:I1PZk}--\eqref{eq:I3PZk2}, and \eqref{vaor}.
\end{proof}

Now we are in a position to prove Theorem \ref{thm2}.
\begin{proof}[Proof of Theorem \ref{thm2}]
Throughout the proof, we keep in mind \eqref{vaor}.

\medskip \noindent \doublebox{\textsc{Step 1.}} Let us prove estimate \eqref{eq:sqe322}.

By testing \eqref{eqrho2} with $PZ_1^k$ for $k\in\{0,1,\ldots,n\}$, arguing as in \eqref{liar}, and using \eqref{i4rp}, we obtain
\begin{align}
&\ \bigg|\int_{\Omega} (\I_1+\I_3) PZ_1^k\bigg| \nonumber\\
&=\bigg| -\int_{\Omega} f PZ_1^k -\int_{\Omega} \I_0[\rho] PZ_1^k + \int_{\Omega} \big[(-\Delta-\lambda)\rho - p(u_0+PU_1)^{p-1}\rho\big] PZ_1^k\bigg| \nonumber\\
&\lesssim \|f\|_{(H^1_0(\Omega))^*} + \|\rho\|^2_{H^1_0(\Omega)} + \|\rho\|_{H^1_0(\Omega)}\bigg[\|((PU_1)^{p-1}-U_1^{p-1})PZ_1^k\|_{L^{\frac{p+1}{p}}(\Omega)}\label{13zk}\\
&\hspace{50pt} +\|U_1^{p-1}(PZ_1^k-Z_1^k)\|_{L^{\frac{p+1}{p}}(\Omega)} + \|U_1\|_{L^{\frac{p+1}{p}}(\Omega)} + \|U_1^{p-1}\|_{L^{\frac{p+1}{p}}(\Omega)}\mone_{\{u_0>0,\ p>2\}}\bigg].\nonumber
\end{align}
Having \eqref{rhsin}-\eqref{izk} in mind, we proceed by distinguishing several cases according to the dimension $n$ and the function $u_0$.

\medskip \noindent \fbox{Case 1:} Assume that $n\ge 7$.

We consider the following subcases:
\begin{itemize}
\item If $\mfb_n\lambda\delta_1^2>\mfc_n\vph(\xi_1)\delta_1^{n-2}$, we have that $\delta_1^2\lesssim \|f\|_{(H^1_0(\Omega))^*}$.

\begin{itemize}
\item[-] When $\delta_1^2\gtrsim \ka_1^{\frac{n+2}{2}}$, it follows that $\|\rho\|_{H^1_0(\Omega)} \lesssim \|f\|_{(H^1_0(\Omega))^*}+\delta_1^2$. Hence, $\|\rho\|_{H^1_0(\Omega)} \lesssim \|f\|_{(H^1_0(\Omega))^*}$.

\item[-] When $\delta_1^2\lesssim \ka_1^{\frac{n+2}{2}}$, it follows that $\|\rho\|_{H^1_0(\Omega)} \lesssim \|f\|_{(H^1_0(\Omega))^*}+\ka_1^{\frac{n+2}{2}}$. Hence, $\|\rho\|_{H^1_0(\Omega)} \lesssim \|f\|_{(H^1_0(\Omega))^*}^{\frac{n+2}{2(n-2)}}$.
\end{itemize}

\item If $\mfb_n\lambda\delta_1^2<\mfc_n\vph(\xi_1)\delta_1^{n-2}$, we have that $\ka_1^{n-2} \lesssim \|f\|_{(H^1_0(\Omega))^*}$ and $\|\rho\|_{H^1_0(\Omega)} \lesssim \|f\|_{(H^1_0(\Omega))^*}+\ka_1^{\frac{n+2}{2}}$. Thus, $\|\rho\|_{H^1_0(\Omega)} \lesssim \|f\|_{(H^1_0(\Omega))^*}^{\frac{n+2}{2(n-2)}}$.

\item If $\mfb_n\lambda\delta_1^2=\mfc_n\vph(\xi_1)\delta_1^{n-2}$, we have that $\|\rho\|_{H^1_0(\Omega)} \lesssim \|f\|_{(H^1_0(\Omega))^*}+\ka_1^{\frac{n+2}{2}}$. It follows from \eqref{izk} and \eqref{13zk} that $\ka_1^{n-1} \lesssim \|f\|_{(H^1_0(\Omega))^*}+\delta_1^{\frac{n+2}{n-2}+2}$.
    Consequently, $\|\rho\|_{H^1_0(\Omega)} \lesssim \|f\|_{(H^1_0(\Omega))^*}^{\frac{n+2}{2(n-1)}}$.
\end{itemize}

\medskip \noindent \fbox{Case 2:} Assume that $n=6$.
\begin{itemize}
\item If $\mfa_nu_0(\xi_1)\delta_1^2 + \mfb_6\lambda\delta_1^2 \ne \mfc_6\vph(\xi_1)\delta_1^4$, we have that
\[\|\rho\|_{H^1_0(\Omega)} \lesssim \|f\|_{(H^1_0(\Omega))^*}|\log \|f\|_{(H^1_0(\Omega))^*}|^{\frac12}.\]

\item If $\mfa_6u_0(\xi_1)\delta_1^2 + \mfb_6\lambda\delta_1^2 = \mfc_6\vph(\xi_1)\delta_1^4$, a cancellation happens in \eqref{iz0}, which leads to
\[\I_1+\I_3 = 2(u_0(x)-u_0(\xi_1))PU_1 + 2(PU_1-U_1)U_1 + 2a_6\vph(\xi_1)\delta_1^2PU_1.\]
Therefore,
\[\|\rho\|_{H^1_0(\Omega)} \lesssim \|f\|_{(H^1_0(\Omega))^*}+\|\I_1+\I_3\|_{L^{\frac{p+1}{p}}(\Omega)} \lesssim \|f\|_{(H^1_0(\Omega))^*}+\delta_1^3+\ka_1^5.\]
Applying \eqref{izk} and \eqref{13zk}, we find that $\ka_1^5 \simeq \delta_1^{\frac52} \lesssim \|f\|_{(H^1_0(\Omega))^*}+\delta_1^4|\log\delta_1|$, and so $\|\rho\|_{H^1_0(\Omega)} \lesssim \|f\|_{(H^1_0(\Omega))^*}$.
\end{itemize}

\medskip \noindent \fbox{Case 3:} Assume that $n=3, 4, 5$ and $u_0=0$.
\begin{itemize}
\item If $n=3$ and $u_0=0$, we have that $\|\rho\|_{H^1_0(\Omega)} \lesssim \|f\|_{(H^1_0(\Omega))^*}$.

\item Assume that $n=4$ and $u_0=0$.
\begin{itemize}
\item If $\mfb_4\lambda\delta_1^2|\log\delta_1| \ne \mfc_4\vph_\lambda^4(\xi_1)\delta_1^2$, we have that $\|\rho\|_{H^1_0(\Omega)} \lesssim \|f\|_{(H^1_0(\Omega))^*}$.

\item If $\mfb_4\lambda\delta_1^2|\log\delta_1| = \mfc_4\vph_\lambda^4(\xi_1)\delta_1^2$, we have that
\[\I_3=PU_1^p-U_1^p-p\delta_1\lambda|\log\delta_1|U_1^{p-1} + pa_4\delta_1\vph_\lambda^4(\xi_1) U_1^{p-1}.\]
Therefore,
\[\|\rho\|_{H^1_0(\Omega)} \lesssim \|f\|_{(H^1_0(\Omega))^*}+\|\I_1+\I_3\|_{L^{\frac{p+1}{p}}(\Omega)} \lesssim \|f\|_{(H^1_0(\Omega))^*}+\delta_1^2.\]
Applying \eqref{iz0} and \eqref{13zk}, we find that $|\int_{\Omega} \I_3PZ_1^0|\simeq \delta_1^2\lesssim \|f\|_{(H^1_0(\Omega))^*}+\delta_1^3$, and so $\|\rho\|_{H^1_0(\Omega)} \lesssim \|f\|_{(H^1_0(\Omega))^*}$.
\end{itemize}

\item Assume that $n=5$ and $u_0=0$.
\begin{itemize}
\item If $\bar{\mfb}_5\lambda\delta_1^2 \ne \mfc_5\vph_\lambda^5(\xi_1)\delta_1^3$, we have the same estimate as above.

\item If $\bar{\mfb}_5\lambda\delta_1^2 = \mfc_5\vph_\lambda^5(\xi_1)\delta_1^3$, we have that    $\|\rho\|_{H^1_0(\Omega)} \lesssim \|f\|_{(H^1_0(\Omega))^*}+\delta_1^2$. By Corollary \ref{cor:PZ^kexp}, the inequality \[\|U_1^{p-1}\delta_1\pa_{\xi_1^k}\mcs_{\delta_1,\xi_1}\|_{L^{\frac{p+1}{p}}(\Omega)} \lesssim \|U_1^{p-1}\|_{L^5(\Omega)} \|\delta_1\pa_{\xi_1^k}\mcs_{\delta_1,\xi_1}\|_{L^2(\Omega)} \lesssim \delta_1^2,\]
    \eqref{izk}, and \eqref{13zk}, one derives that $\ka_1^4 \simeq \delta_1^{\frac83} \lesssim \|f\|_{(H^1_0(\Omega))^*}+\delta_1^{\frac72}$, which gives $\|\rho\|_{H^1_0(\Omega)} \lesssim \|f\|_{(H^1_0(\Omega))^*}^{\frac{3}{4}}$.
\end{itemize}
\end{itemize}
\medskip

\noindent\fbox{Case 4:} Assume that $n=3, 4, 5$ and $u_0>0$.
\begin{itemize}
\item If $\mfa_nu_0(\xi_1)\delta_1^{\frac{n-2}{2}} \ne \mfc_n\delta_1^{n-2}
    \left.\begin{cases}
    \vph^3_\lambda(\xi_1) &\text{if } n=3\\
    \vph(\xi_1) &\text{if }n=4,5
    \end{cases}\right\}$,
we obtain that $\|\rho\|_{H^1_0(\Omega)} \lesssim \|f\|_{(H^1_0(\Omega))^*}$.

\item If $\mfa_nu_0(\xi_1)\delta_1^{\frac{n-2}{2}} = \mfc_n\delta_1^{n-2}
    \left.\begin{cases}
    \vph^3_\lambda(\xi_1) &\text{if } n=3\\
    \vph(\xi_1)&\text{if } n=4,5
    \end{cases}\right\}$,
the expansion
\begin{align*}
\I_1+\I_3=&(u_0+PU_1)^p-u_0^p-PU_1^p-pu_0(\xi_1)PU_1^{p-1}+PU_1^p-U_1^p\\
&\ +pa_n\delta_1^{\frac{n-2}{2}}
\left.\begin{cases}
\vph^3_\lambda(\xi_1)&\text{if } n=3\\
\vph(\xi_1)&\text{if } n=4,5
\end{cases}\right\} PU_1^{p-1}-\lambda PU_1\mone_{\{n=4, 5\}}
\end{align*}
gives
\[\|\rho\|_{H^1_0(\Omega)} \lesssim \|f\|_{(H^1_0(\Omega))^*}+\|\I_1+\I_3\|_{L^{\frac{p+1}{p}}(\Omega)} \lesssim\|f\|_{(H^1_0(\Omega))^*}+
\begin{cases}
\delta_1 &\text{if } n=3,\\
\delta_1^{\frac{n-2}{2}} &\text{if } n=4,5.
\end{cases}\]
Besides, making use of \eqref{izk} and \eqref{13zk}, we know that
\[\ka_1^{n-1}\simeq\delta_1^{\frac{n-1}{2}} \lesssim \|f\|_{(H^1_0(\Omega))^*}
+\begin{cases}
\delta_1^{\frac32} &\text{if } n=3,\\
\delta_1^{n-2} &\text{if } n=4,5.
\end{cases}\]
We conclude
\[\|\rho\|_{H^1_0(\Omega)} \lesssim \begin{cases}
\|f\|_{(H^1_0(\Omega))^*} &\text{if } n=3 \text{ and } u_0>0,\\
\|f\|_{(H^1_0(\Omega))^*}^{\frac{n-2}{n-1}} &\text{if } n=4,5 \text{ and } u_0>0.
\end{cases}\]
This completes the derivation of \eqref{eq:sqe322}.
\end{itemize}

\medskip \noindent \doublebox{\textsc{Step 2.}} We prove the optimality of  \eqref{eq:sqe322}.

\medskip \noindent\fbox{Case 1:} Assume that $\zeta(t)=t$.

One can treat this case as in Case 1 of Subsection \ref{subs3.2}, by choosing a point $\xi_1 \in \Omega$ and setting \[\ep\simeq \left.\begin{cases}
\delta_1 &\text{if } n=3\\
\delta_1^2|\log\delta_1| &\text{if } n=4, u_0=0
\end{cases}\right\} + \ka_1^{n-2}.\]

\noindent\fbox{Case 2:} Assume that $\zeta(t) \gg t$.

Let us choose $\delta_1>0$ and $\xi_1 \in \Omega$ satisfying the following conditions
\begin{equation}\label{eq:dxcond}
\begin{cases}
\displaystyle \mfa_nu_0(\xi_1)\delta_1^{\frac{n-2}{2}}=\mfc_n\vph(\xi_1)\delta_1^{n-2} &\text{if } n=4,5 \text{ and } u_0>0,\\
\displaystyle \bar{\mfb}_5\lambda\delta_1^2 = \mfc_5\vph_\lambda^5(\xi_1)\delta_1^3 &\text{if } n=5 \text{ and } u_0=0,\\
\displaystyle \mfa_6u_0(\xi_1)\delta_1^2+\mfb_6\lambda\delta_1^2 > \mfc_n\vph(\xi_1)\delta_1^4 &\text{if } n=6, \\
\displaystyle \mfb_n\lambda\delta_1^2 = \mfc_n\vph(\xi_1)\delta_1^{n-2} &\text{if } n\ge 7.
\end{cases}
\end{equation}
We now consider a function $\rho$ solving the following linearized problem
\[\begin{cases}
\displaystyle -\Delta \rho-\lambda \rho - p(u_0+PU_1)^{p-1}\rho = \I_1+\I_3+\I_0[\rho] +\sum_{k=0}^n \tc_1^k(-\Delta-\lambda)PZ_1^k \quad \text{in } \Omega,\\
\displaystyle \rho=0 \quad \text{on } \pa\Omega,\quad \tc_1^k \in \R \quad \text{for } k=0,\ldots,n,\\
\displaystyle \big\langle \rho, PZ_1^k \big\rangle_{H^1_0(\Omega)} = 0 \quad \text{for } k=0,\ldots,n.
\end{cases}\]
where $\I_1$, $\I_3$, and $\I_0[\rho]$ are defined as in \eqref{eq:I2} with $(\delta_1,\xi_1)$ satisfying \eqref{eq:dxcond}.

Denote $f:=\sum_{k=0}^n\tc_1^kU_1^{p-1}Z_1^k$. Using \eqref{iz0} and \eqref{izk}, we obtain
\begin{align*}
\|f\|_{(H^1_0(\Omega))^*} &\lesssim |\tc_1^0|+\max_{k\in\{1,\ldots,n\}}|\tc_1^k|\\
&\lesssim \bigg|\int_{\Omega}(\I_1+\I_3)PZ_1^0\bigg| + \max_{k\in\{1,\ldots,n\}}\bigg|\int_{\Omega}(\I_1+\I_3)PZ_1^k\bigg|\\
&\lesssim \begin{cases}
\ka_1^{n-1} &\text{if } [n=4,\; u_0>0], \text{ or } n=5, \text{ or } n\ge 7, \\
\delta_1^2 &\text{if } n=6.
\end{cases}
\end{align*}
It follows that
\[\|\rho\|_{H^1_0(\Omega)} \lesssim \vsi_3(\delta) := \begin{cases}
\delta_1^{\frac{n-2}{2}} &\text{if } n=4,5 \text{ and } u_0>0, \\
\delta_1^2 &\text{if } n=5 \text{ and } u_0=0, \\
\delta_1^2|\log\delta_1|^{\frac12} &\text{if } n=6, \\
\ka_1^{\frac{n+2}{2}} &\text{if } n\ge 7.
\end{cases}\]
On the other hand, we can deduce a lower bound estimate
\begin{align*}
&\ \|\rho\|^2_{H^1_0(\Omega)} \\
&\gtrsim \begin{cases}
\displaystyle \int_{\Omega} \int_{\Omega}(\lambda PU_1)(x)\frac{1}{|x-\om|^{n-2}} (\lambda PU_1)(\om) dxd\om &\text{if } [n=4,5,\; u_0>0] \text{ or } n=6, \\
\displaystyle \int_{\Omega} \int_{\Omega}(PU_1^p-U_1^p)(x)\frac{1}{|x-\om|^{n-2}} (PU_1^p-U_1^p)(\om) dxd\om &\text{if } [n=5,\; u_0=0] \text{ or } n\ge 7
\end{cases}\\
&\gtrsim (\vsi_3(\delta))^2.
\end{align*}
We set $u_*:=(u_0+PU_1+\rho)_+$. Then, by proceeding as in Case 2 of Subsection \ref{subs3.2}, we finish the proof.
\end{proof}
\begin{rmk}
Assume that $\nu \ge 2$. Arguing as above, one can find a nonnegative function $u_* \in H^1_0(\Omega)$ with $\delta_i=\delta_j$ and $|\xi_i-\xi_j|\gtrsim 1$ for $1 \le i\ne j \le \nu$ such that
\[\inf\left\{\Big\|u_*-\Big(u_0+\sum_{i=1}^{\nu}PU_{\tilde{\delta}_i,\tilde{\xi}_i}\Big)\Big\|_{H^1_0(\Omega)}: \(\tde_i,\txi_i\) \in (0,\infty) \times \Omega,\, i = 1,\ldots,\nu\right\}\gtrsim \zeta(u_*),\]
where $\zeta$ is given by \eqref{eq:zeta122}, except for the cases $[n=3,\; u_0>0]$ and $[n=4,\; u_0=0]$.
In these exceptional cases, additional technical difficulties arise.
\end{rmk}

\appendix
\section{Some useful estimates}\label{a}
\begin{lemma}
Let $a, b>0$. Then the following estimates hold:
\begin{equation}\label{iqu}
\quad |(a+b)^s-a^s-b^s| \lesssim \begin{cases}
\min\{a^{s-1}b, ab^{s-1}\} &\text{if } 1\le s \le 2,\\
a^{s-1}b+ab^{s-1} &\text{if }s>2.
\end{cases}
\end{equation}
Moreover, we have the following asymptotic expansions:
\begin{align}
&(a+b)^s-a^s=O(a^{s-1}b)\mone_{s>1}+O(b^s) \quad \text{for } s>0,\label{ab1}\\
&(a+b)^s = a^s + sa^{s-1}b + O(a^{s-2}b^2)\mone_{s>2}+ O(b^s) \quad \text{for } s > 1, \label{ab6}\\
&(a+b)^s = a^s + sa^{s-1}b + \frac{p(p-1)}{2}a^{s-2}b^2+O(a^{s-3}b^3)\mone_{s>3}+ O(b^s) \quad \text{for } s > 2. \label{ab7}
\end{align}
For any $a>0$, $b \in \R$ such that $a+b \ge 0$ and $1<s<2$, it holds that
\begin{equation}\label{ape1}
\left|(a+b)^s-a^s-sa^{s-1}b\right| \lesssim \min\left\{a^{s-2}|b|^2, |b|^s\right\}.
\end{equation}
\end{lemma}
\begin{lemma}\label{a4}
Let $s>0$ and $U_{\delta,\xi}$ be the bubble defined in \eqref{AT bubbles}. Then
\[\int_{\Omega} U_{\delta,\xi}^s \lesssim
\begin{cases}
\delta^{\frac{n-2}{2}s} & \text{if } 0 < s < \frac{n}{n-2}, \\
\delta^{\frac{n}{2}} |\log \delta| & \text{if } s = \frac{n}{n-2}, \\
\delta^{n - \frac{n-2}{2}s} & \text{if } s > \frac{n}{n-2}.
\end{cases}\]
\end{lemma}
\begin{lemma}\label{a22}
Let $U_{\delta_i,\xi_i}$ and $U_{\delta_j,\xi_j}$ be the bubbles for $1 \leq i \neq j \leq \nu$.
If $s,t \geq 0$ satisfy $s+t=2^*$, then for any fixed $\tau>0$, we have
\[\int_{\R^n} U_{\delta_i,\xi_i}^s U_{\delta_j,\xi_j}^t \lesssim
\begin{cases}
q_{ij}^{\min\{s,t\}} & \text{if } |s - t| \geq \tau, \\
q_{ij}^{\frac{n}{n-2}} |\log q_{ij}| & \text{if } s = t,
\end{cases}\]
provided $q_{ij}$ in \eqref{rq} is sufficiently small.
\end{lemma}
\begin{proof}
See \cite[Lemma A.3]{CK}.
\end{proof}
\begin{lemma}\label{a1}
Suppose $\alpha>0$. Then
\[\int_{\Omega} \frac{1}{|x - z|^{n-2}} \(\frac{\delta}{\delta^2 + |z - \xi|^2}\)^{\frac{\alpha}{2}} dz
\lesssim
\begin{medsize}
\begin{cases}
\displaystyle \delta^{\frac{\alpha}{2}} &\text{if } 0 < \alpha < 2, \\
\displaystyle \delta(1+|\log |x-\xi||) &\text{if } \alpha = 2, \\
\displaystyle \delta^{\frac{\alpha}{2}} (\delta^2+|x-\xi|^2)^{-\frac{\alpha - 2}{2}} &\text{if } 2 < \alpha < n, \\
\displaystyle \delta^{\frac{n}{2}} (\delta^2+|x-\xi|^2)^{-\frac{n-2}{2}} \log(2+|x-\xi|\delta^{-1}) &\text{if } \alpha = n, \\
\displaystyle \delta^{n-\frac{\alpha}{2}} (\delta^2+|x-\xi|^2)^{-\frac{n-2}{2}} &\text{if } \alpha > n.
\end{cases}
\end{medsize}\]
\end{lemma}
\begin{proof}
It follows from direct computations.
\end{proof}

\section{Proof of \eqref{vaor}}\label{b}
\begin{lemma}\label{lema6}
Let $\vph_\lambda^n(x) := H^n_\lambda(x, x)$ for $n=3,4,5$, where $H^n_\lambda(x, y)$ satisfies equations \eqref{h1xy}--\eqref{h3xy}. If $d(x, \pa \Omega)$ is small, then we have
\[\begin{cases}
\vph_\lambda^n(x) = \dfrac{1}{(2d(x, \pa \Omega))^{n-2}}\(1 + O(d(x, \pa \Omega))\),\\[1em]
|\nabla\vph_\lambda^n (x)| = \dfrac{2(n-2)}{(2d(x, \pa \Omega))^{n-1}}\(1 + O(d(x, \pa \Omega))\).
\end{cases}\]
\end{lemma}

\begin{proof}
Since $\Omega$ is a smooth domain, there exists $d_0>0$ such that for every $x \in \Omega$ with $d(x, \pa \Omega) < d_0$, there exists a unique $x' \in \pa \Omega$ such that $d(x, \pa \Omega) = |x - x'|$.
By an appropriate translation and rotation, we may assume without loss of generality that $x = (0, d)$, $x' = 0$, and the boundary near the origin is locally given by a $C^2$ function $\phi$ with $\phi(0) = 0$, $\nabla \phi(0) = 0$. Specifically,
\begin{align*}
\pa \Omega \cap B(0,\tau) &= \{ y = (y', y^n) \in \R^n : y^n=\phi(y') \} \cap B(0,\tau), \\
\Omega \cap B(0,\tau) &= \{ y \in \R^n : y^n > \phi(y') \} \cap B(0,\tau)
\end{align*}
for some small $\tau>0$. Let $x'' = (0, -d)$ be the reflection of $x$ across the boundary. For sufficiently small $d$, $x''
\not\in\Omega$, and the function $\frac{1}{|y-x''|^{n-2}}$ is harmonic in $\Omega$. Define
\[F_\lambda^n(y) := H_\lambda^n(y, x) -
\begin{cases}
\dfrac{1}{|y-x''|^{n-2}}-\frac{\lambda}{2}|y-x''| &\text{if } n=3, \\
\dfrac{1}{|y-x''|^{n-2}} - \dfrac{\lambda}{2} |\log|x''-y|| &\text{if } n=4, \\
\dfrac{1}{|y-x''|^{n-2}} + \dfrac{\lambda}{2} \dfrac{1}{|x''-y|}-2\lambda^2|y-x''| &\text{if } n=5.
\end{cases}\]
Then $F^n_\lambda$ satisfies
\[\begin{cases}
\Delta_y F^n_\lambda +\lambda F^n_\lambda = f^n_\lambda &\text{in } \Omega, \\
F^n_\lambda = g^n_\lambda &\text{on } \pa \Omega,
\end{cases}\]
where
\[f^n_\lambda(y) :=
\begin{cases}
-\frac{\lambda^2}{2} \(|y-x| - |y-x''| \) &\text{if } n=3, \\
-\lambda |\log |x-y|| - \dfrac{\lambda^2}{2} |\log |x''-y|| &\text{if } n=4, \\
-2\lambda^2|y-x| + 2\lambda^3|x''-y| &\text{if } n=5,
\end{cases}\]
and
\begin{align*}
&\ g^n_\lambda(y) \\
&:= \begin{cases}
\dfrac{1}{|y-x|^{n-2}}-\dfrac{1}{|y-x''|^{n-2}} - \dfrac{\lambda}{2}(|y-x|-|y-x''|) &\text{if } n=3, \\
\dfrac{1}{|y-x|^{n-2}}-\dfrac{1}{|y-x''|^{n-2}} - \dfrac{\lambda}{2}(|\log |x-y||-|\log |x''-y||) &\text{if } n=4, \\
\dfrac{1}{|y-x|^{n-2}}-\dfrac{1}{|y-x''|^{n-2}} + \dfrac{\lambda}{2}\(\dfrac{1}{|x-y|}-\dfrac{1}{|x''-y|}\) - 2\lambda^2(|y-x|-|y-x''|) &\text{if } n=5.
\end{cases}\end{align*}

For $y \in \pa \Omega \cap B(0,\tau)$, we have the Taylor expansions
\begin{align*}
&|y-x|=\sqrt{|y|^2+d^2-2 d y^n}=\sqrt{|y|^2+d^2}\(1+O\(\frac{d y^n}{|y|^2+d^2}\)\),\\
&\left|y-x^{\prime \prime}\right|=\sqrt{|y|^2+d^2+2 d y^n}=\sqrt{|y|^2+d^2}\(1+O\(\frac{d y^n}{|y|^2+d^2}\)\),
\end{align*}
where we used the smoothness of $\phi$. Since $|y^n|=|\phi(y^{\prime})|=O(|y^{\prime}|^2)$, we observe
\begin{align*}
\frac{1}{|y-x|^{n-2}}-\frac{1}{|y-x^{\prime \prime}|^{n-2}} &= (|y|^2+d^2)^{-\frac{n-2}{2}} O\(\frac{d y^n}{|y|^2+d^2}\) \\
&= (|y|^2+d^2)^{-\frac{n-2}{2}} O(d) = O(d^{-n+3}).
\end{align*}
Similarly,
\[\begin{cases}
|y-x|-|y-x''|=O(1) &\text{for } n=3,5,\\
|\log|y-x||-|\log|y-x''||=O(1) &\text{for } n=4,\\
\frac{1}{|y-x|}-\frac{1}{|y-x''|}=O(1) &\text{for } n=5.
\end{cases}\]

For $y \in \pa \Omega \cap\(\R^n \backslash B(0,\tau)\)$, the above differences are also uniformly bounded. In other words,
\[\|g^n_\lambda\|_{L^{\infty}(\pa\Omega)}=O(d^{-n+3}).\]
In particular, $\|f^n_\lambda\|_{L^t(\Omega)} \lesssim 1$ for any $t>n$. By standard elliptic estimates, we obtain
\[\|F_\lambda^n\|_{L^{\infty}(\Omega)}=O(d^{-n+3}).\]
Hence, evaluating at $x$, we get
\begin{align*}
\vph^n_\lambda(x)=H^n_\lambda(x,x)&=\left.\begin{cases}
\frac{1}{|y-x^{\prime \prime}|^{n-2}}-\frac{\lambda}{2}|y-x''| &\text{if } n=3\\
\frac{1}{|y-x^{\prime \prime}|^{n-2}}-\frac{\lambda}{2}|\log|x''-y|| &\text{if } n=4\\
\frac{1}{|y-x^{\prime \prime}|^{n-2}}+\frac{\lambda}{2}\frac{1}{|x''-y|}-2\lambda^2|y-x''| &\text{if } n=5
\end{cases}\right\} + O(d^{-n+3})\\
&=\frac{1}{(2d(x,\pa\Omega))^{n-2}}(1+O(d(x,\pa\Omega))).
\end{align*}
The estimate for $|\nabla\vph^n_\lambda(x)|$ follows analogously by applying interior gradient estimates under the same reflections.
\end{proof}
\begin{rmk}
The estimate for $\vph$ in \eqref{vaor} follows with slight modifications to the above proof.
\end{rmk}

\bigskip \noindent \textbf{Acknowledgement.}
H. Chen and S. Kim were supported by Basic Science Research Program through the National Research Foundation of Korea (NRF) funded by the Ministry of Science and ICT (2020R1C1C1A01010133, RS-2025-00558417).
The research of J. Wei is partially supported by GRF grant entitled ``New frontiers in singular limits of elliptic and parabolic equations".

\end{document}